\documentclass[11pt]{amsart}

\usepackage[
paper=a4paper,
text={147mm,230mm},centering
]{geometry}

\iftrue
\makeatletter
\def\@settitle{%
  \vspace*{-20pt}
  \begin{flushleft}%
    \baselineskip14\p@\relax
    \normalfont\bfseries\LARGE
    \@title
  \end{flushleft}%
}
\def\@setauthors{%
  \begingroup
  \def\thanks{\protect\thanks@warning}%
  \trivlist
  \large \@topsep30\p@\relax
  \advance\@topsep by -\baselineskip
  \item\relax
  \author@andify\authors
  \def\\{\protect\linebreak}%
  \authors
  \ifx\@empty\contribs
  \else
    ,\penalty-3 \space \@setcontribs
    \@closetoccontribs
  \fi
  \normalfont
  \@setaddresses
  \endtrivlist
  \endgroup
}
\def\@setaddresses{\par
  \nobreak \begingroup\raggedright
  \small
  \def\author##1{\nobreak\addvspace\smallskipamount}%
  \def\\{\unskip, \ignorespaces}%
  \interlinepenalty\@M
  \def\address##1##2{\begingroup
    \par\addvspace\bigskipamount\noindent
    \@ifnotempty{##1}{(\ignorespaces##1\unskip) }%
    {\ignorespaces##2}\par\endgroup}%
  \def\curraddr##1##2{\begingroup
    \@ifnotempty{##2}{\nobreak\noindent\curraddrname
      \@ifnotempty{##1}{, \ignorespaces##1\unskip}\/:\space
      ##2\par}\endgroup}%
  \def\email##1##2{\begingroup
    \@ifnotempty{##2}{\smallskip\nobreak\noindent E-mail address%
      \@ifnotempty{##1}{, \ignorespaces##1\unskip}\/:\space
      \ttfamily##2\par}\endgroup}%
  \def\urladdr##1##2{\begingroup
    \def~{\char`\~}%
    \@ifnotempty{##2}{\nobreak\noindent\urladdrname
      \@ifnotempty{##1}{, \ignorespaces##1\unskip}\/:\space
      \ttfamily##2\par}\endgroup}%
  \addresses
  \endgroup
  \global\let\addresses=\@empty
}
\def\@setabstracta{%
    \ifvoid\abstractbox
  \else
    \skip@25\p@ \advance\skip@-\lastskip
    \advance\skip@-\baselineskip \vskip\skip@
    \box\abstractbox
    \prevdepth\z@ 
    \vskip-10pt
  \fi
}
\renewenvironment{abstract}{%
  \ifx\maketitle\relax
    \ClassWarning{\@classname}{Abstract should precede
      \protect\maketitle\space in AMS document classes; reported}%
  \fi
  \global\setbox\abstractbox=\vtop \bgroup
    \normalfont\small
    \list{}{\labelwidth\z@
      \leftmargin0pc \rightmargin\leftmargin
      \listparindent\normalparindent \itemindent\z@
      \parsep\z@ \@plus\p@
      
    }%
    \item[\hskip\labelsep\bfseries\abstractname.]%
}{%
  \endlist\egroup
  \ifx\@setabstract\relax \@setabstracta \fi
}
\def\section{\@startsection{section}{1}%
  \z@{-1.2\linespacing\@plus-.5\linespacing}{.8\linespacing}%
  {\normalfont\bfseries\large}}
\def\subsection{\@startsection{subsection}{2}%
  \z@{-.8\linespacing\@plus-.3\linespacing}{.3\linespacing\@plus.2\linespacing}%
  {\normalfont\bfseries}}
\def\subsubsection{\@startsection{subsubsection}{3}%
  \z@{.7\linespacing\@plus.1\linespacing}{-1.5ex}%
  {\normalfont\itshape}}
\def\@secnumfont{\bfseries}
\makeatother
\fi 

\usepackage{amssymb}
\usepackage{mathrsfs}
\usepackage[all]{xy}
\usepackage{graphicx,color,float}
\usepackage[bookmarks]{hyperref}
\usepackage{pinlabel}
\usepackage[textwidth=1.3in,color=yellow]{todonotes}
\usepackage{amsmath}
\usepackage{pb-diagram,pb-xy}
\usepackage{amsmath}
\usepackage{amsfonts}
\usepackage{graphicx}
\usepackage{amscd}
\usepackage{url}
\usepackage{caption}
\usepackage{subcaption}
\usepackage{fancybox}
\usepackage{wrapfig}
\usepackage{color}
\usepackage{multirow, url}

\usetikzlibrary{calc}
\textwidth=\paperwidth \advance\textwidth by-2\oddsidemargin
\advance\oddsidemargin by-1in
\evensidemargin=\oddsidemargin
\calclayout

\renewcommand{\bar}{\overline}

\renewcommand{\hat}{\widehat}

\newcommand{\dge}{\rotatebox[origin=c]{45}{$\ge$}}
\newcommand{\uge}{\rotatebox[origin=c]{315}{$\ge$}}

\newcommand{\updots}{\hbox to1.65em{\rotatebox[origin=c]{45}{$\cdots$}}}
\newcommand{\dndots}{\hbox to1.65em{\rotatebox[origin=c]{315}{$\cdots$}}}

\newcommand{\lmd}[1]{\hbox to1.65em{$\hfill \lambda_{#1} \hfill$}}

\newcommand{\C}{\mathbb{C}}

\newcommand{\N}{\mathbb{N}}

\newcommand{\R}{\mathbb{R}}

\newcommand{\Z}{\mathbb{Z}}
\newcommand{\p}{\mathbb{P}}
\newcommand{\F}{\mathcal{F}}

\newcommand{\uu}{\mathfrak{u}}
\newcommand{\fa}{\mathfrak{a}}
\newcommand{\fb}{\mathfrak{b}}
\newcommand{\fc}{\mathfrak{c}}
\newcommand{\fd}{\mathfrak{d}}

\newcommand{\hookuparrow}{\mathrel{\rotatebox[origin=c]{90}{$\hookrightarrow$}}}

\def\mcal{\mathcal}
\def\frak{\mathfrak}

\numberwithin{equation}{section} 

\theoremstyle{plain}
\newtheorem{theorem}{Theorem}[section]
\newtheorem{thmx}{Theorem}
\newtheorem{proposition}[theorem]{Proposition}
\newtheorem{corollary}[theorem]{Corollary}

\newtheorem{lemma}[theorem]{Lemma}

\theoremstyle{definition}
\newtheorem{definition}[theorem]{Definition}

\newtheorem{example}[theorem]{Example}
\newtheorem{remark}[theorem]{Remark}

\def\to{\mathchoice{\longrightarrow}{\rightarrow}{\rightarrow}{\rightarrow}}
\makeatletter
\newcommand{\shortxra}[2][]{\ext@arrow 0359\rightarrowfill@{#1}{#2}}
\def\longrightarrowfill@{\arrowfill@\relbar\relbar\longrightarrow}
\newcommand{\longxra}[2][]{\ext@arrow 0359\longrightarrowfill@{#1}{#2}}
\renewcommand{\xrightarrow}[2][]{\mathchoice{\longxra[#1]{#2}}%
  {\shortxra[#1]{#2}}{\shortxra[#1]{#2}}{\shortxra[#1]{#2}}}
\makeatother
\numberwithin{equation}{section}

\begin{document}

\title{Lagrangian fibers of Gelfand-Cetlin systems}

\author{Yunhyung Cho}
\address{Department of Mathematics Education, Sungkyunkwan University, Seoul, Republic of Korea}
\email{yunhyung@skku.edu}

\author{Yoosik Kim}
\address{Department of Mathematics, Brandeis University, Waltham, USA and Center of Mathematical Sciences and Applications, Harvard University, Cambridge, USA}
\email{yoosik@brandeis.edu, yoosik@cmsa.fas.harvard.edu}

\author{Yong-Geun Oh}
\address{Center for Geometry and Physics, Institute for Basic Science (IBS),  Pohang, Republic of Korea, and Department of Mathematics, POSTECH, Pohang, Republic of Korea}
\email{yongoh1@postech.ac.kr}

\begin{abstract}
	A {\em Gelfand-Cetlin system} is a completely integrable system defined on a partial flag manifold 
	whose image is a rational convex polytope called a {\em Gelfand-Cetlin polytope}. Motivated by the study of Nishinou-Nohara-Ueda \cite{NNU1} on the Floer theory 
	of Gelfand-Cetlin systems, we provide a detailed description of topology of Gelfand-Cetlin fibers. 
	In particular, we prove that any fiber over an interior point of a $k$-dimensional face of the Gelfand-Cetlin polytope is an isotropic submanifold and 
	is diffeomorphic to $(S^1)^k \times N$ for some smooth manifold $N$. We also prove that such $N$'s are exactly the vanishing cycles shrinking to points in the associated toric variety 
	via the toric degeneration.
	We also devise an algorithm of reading off
	Lagrangian fibers from the combinatorics of the ladder diagram.
\end{abstract}
\date{\today ~(\textit{Last update} : 03/03/2019)}
\maketitle
\setcounter{tocdepth}{1} 
\tableofcontents

\section{Introduction}
\label{secIntroduction} 

A (complex) partial flag manifold of Lie type $A$ can be defined as the co-adjoint orbit $\mcal{O}_\lambda$ of an element $\lambda$
in the dual Lie algebra of the unitary group.
The choice of $\lambda$ determines a {\em Kirillov-Kostant-Souriau symplectic form} $\omega_\lambda$, a $U(n)$-invariant K\"{a}hler form on $\mcal{O}_\lambda$.
The \emph{Gelfand-Cetlin system} $	\Phi_\lambda \colon (\mathcal{O}_\lambda, \omega_\lambda) \rightarrow \R^{\dim_\C \mathcal{O}_\lambda}$ is a completely integrable system on the orbit $\mcal{O}_\lambda$  constructed by Guillemin and Sternberg \cite{GS2}.
The Gelfand-Cetlin system resembles a toric moment map in the sense that the image is a convex polytope $\Delta_\lambda$ and the fiber over every interior point of $\Delta_\lambda$ is a Lagrangian torus. 
One major
difference from the toric moment map is that the map $\Phi_\lambda$ is
smooth over the interior of $\Delta_\lambda$ but only continuous up to the boundary of $\Delta_\lambda$.

One notable feature of the Gelfand-Cetlin system $\Phi_\lambda$ is that non-torus Lagrangian fibers may appear at some boundary strata of $\Delta_\lambda$.
Those non-torus Lagrangian fibers are particularly important to realize Strominger-Yau-Zaslow and homological mirror symmetry of partial flag varieties in full generality.
The appearance of such non-torus Lagrangian fibers is responsible for the incompleteness of the Givental-Hori-Vafa mirror, the Floer theoretical SYZ mirror of the Lagrangian torus fibration $\Phi_\lambda$ derived by \cite{NNU1}, in order to study closed mirror symmetry. For instance, the Jacobian ring of the torus mirror sometimes has too small rank to compute the quantum cohomology ring.
Moreover, some non-torus Lagrangian fibers are indeed non-zero objects in the Fukaya category.
Partly because of lack of understanding of Lagrangian fibers in higher dimensional partial flag varieties, Floer theory of non-torus Lagrangians in only limited cases is understood, see \cite{NU2, EL}.

In this paper, motivated by those works, we study Lagrangian fibers of the Gelfand-Cetlin systems in details and classify all Lagrangian fibers some of which have their topological type not of tori.
The first main result of the paper is stated as follows. 		

\begin{thmx}[Theorem~\ref{theorem_main}]\label{theoremA}
	Let $\Phi_\lambda$ be the Gelfand-Cetlin system on
	the co-adjoint orbit $(\mathcal{O}_\lambda, \omega_\lambda)$ for $\lambda \in \mathfrak{u}(n)^*$ and let $\Delta_\lambda$ be the corresponding Gelfand-Cetlin
	polytope. For any point $\textbf{\textup{u}} \in \Delta_\lambda$, the fiber $\Phi_\lambda^{-1}(\textbf{\textup{u}})$ is an isotropic submanifold of $(\mathcal{O}_\lambda, \omega_\lambda)$
	and is the total space of an iterated bundle
	\[
		\Phi^{-1}_\lambda(\textbf{\textup{u}}) = E_{n-1} \stackrel{p_{n-1}} \longrightarrow E_{n-2} \stackrel{p_{n-2}} \longrightarrow \cdots \stackrel{p_2} \longrightarrow E_1
		\stackrel{p_1} \longrightarrow E_0= \mathrm{point}
	\]
	such that the fiber at each stage is either a point or a product of odd dimensional spheres.
	Two fibers $\Phi^{-1}_\lambda(\textbf{\textup{u}}_1)$ and $\Phi^{-1}_\lambda(\textbf{\textup{u}}_2)$ are diffeomorphic if two points $\textbf{\textup{u}}_1$ and $\textbf{\textup{u}}_2$ are
	 contained in the relative interior of the same face.
\end{thmx}		

The process building up the iterated bundle in Theorem~\ref{theoremA} is constructive and algorithmic.
In the light of the work of the first named author with An and Kim \cite{ACK}, the face structure of the polytope $\Delta_\lambda$ can be described in terms of certain subgraphs of the ladder diagram corresponding to $\lambda$.
Their description enables us to reveal the iterated bundle structures of fibers of $\Phi_\lambda$ from the ladder diagram, responding to the face structure of $\Delta_\lambda$.
As a byproduct, we obtain a complete classification of Lagrangian fibers of $\Phi_\lambda$.
The explicit process will be crucial for applications of the Gelfand-Cetlin systems to symplectic geometry and mirror symmetry.

\begin{remark}\label{remarkA} There is a B or D-type analogue of Theorem~\ref{theoremA} in \cite{CK}.
In this case, an even-dimensional sphere can appear as a factor of fibers in the iterated bundle described in Theorem~\ref{theoremA}.
\end{remark}

The second goal of this paper is to analyze how Gelfand-Cetlin fibers deform under the toric degeneration of a partial flag manifold to the toric variety associated with the polytope $\Delta_\lambda$ in \cite{GL,KM}.
To compute the disc potential in \cite{FOOO, COtoric, FOOOToric1}, Nishinou-Nohara-Ueda \cite{NNU1} constructed a toric degeneration from the Gelfand-Cetlin system $\Phi_\lambda$ to the toric moment map $\Phi$.
It leads to the following commutative diagram
\begin{equation}\label{equation_toric_degeneration_diagram}
	\xymatrix{
		  (\mathcal{O}_\lambda, \omega_\lambda)  \ar[dr]_{\Phi_\lambda} \ar[rr]^{ \phi}
                              & & (X_\lambda, \omega)
      \ar[dl]^{\Phi} \\
  & \Delta_\lambda &}
\end{equation}
where $\phi$ is a contraction map from $\mathcal{O}_\lambda$ onto $X_\lambda$, cf. \cite{Ru, HK}.
Our second main theorem is stated as follows.

\begin{thmx}[Theorem ~\ref{theorem_contraction}]\label{theoremB}
	Let $\textbf{\textup{u}}$ be a point lying on the relative interior of an $r$-dimensional face of the Gelfand-Cetlin polytope $\Delta_\lambda$. Then every $S^1$-factor appeared
	in any stage of the iterated bundle given in Theorem \ref{theoremA} comes out as a trivial factor and we get
	\[
		\Phi_\lambda^{-1}(\textbf{\textup{u}}) \cong T^r \times Y(\textbf{\textup{u}})
	\]
	where $Y(\textbf{\textup{u}})$ is the iterated bundle obtained from the original bundle by removing all $S^1$-factors.
	Moreover, the map $\phi \colon \Phi_\lambda^{-1}(\textbf{\textup{u}}) \rightarrow \Phi_{\vphantom{\lambda}}^{-1}(\textbf{\textup{u}})$ is nothing but the projection $T^r \times Y(\textbf{\textup{u}}) \rightarrow T^r$ onto the first factor.
\end{thmx}

According to the work of Batyrev, Ciocan-Fontanine, Kim, and Van Straten \cite{BCKV}, the toric degeneration can be interpreted as a smoothing of conifold strata (via a part of \emph{conifold} transition). Namely, through the map $\phi$ in~\eqref{equation_toric_degeneration_diagram}, a partial flag manifold is deformed into a singular toric variety having conifold strata. Theorem \ref{theoremB} tells us how each fiber degenerates into a toric fiber. Every odd-dimensional sphere of dimension $>1$ appeared in each stage of the iterated bundle $\{E_\bullet\}$ contracts to a point, while each $S^1$-factor persists.

\begin{remark}
The persistence of the $S^1$-factors in type $A$ can be highlighted by comparing the Gelfand-Cetlin systems of type $B, D$. Some $S^1$-factor in the analogue of Theorem~\ref{theoremA} might degenerate to a point, cf \cite{CK}.
\end{remark}
The following Table~\ref{table_featuresofgefiber} summarizes our results in the present paper.

\vspace{0.1cm}
\begin{center}
\begin{table}[H]
  \begin{tabular}{| l || c | c | }  \hline
      & \,\, Gelfand-Cetlin fiber \,\, &   Toric moment fiber \\ \hline \hline
    over any point & \multicolumn{2}{|c|}{isotropic submanifold}  \\ \hline
    over any interior point & \multicolumn{2}{|c|}{Lagrangian torus} \\ \hline
    \multirow{3}{*} {\shortstack[l]{over any point in the relative\\ interior of a $k$-dim face $f$}}
     & \multicolumn{2}{|c|}{$\pi_1 \textup{(fiber)} = \Z^k$, \, $\pi_2 \textup{(fiber)} = 0$}  \\ \cline{2-3}
    			& $(S^1)^k \times N_f$ & $(S^1)^k$   \\ \cline{2-2}\cline{2-3}
			& can be Lagrangian & can \emph{not} be Lagrangian \\ \hline
\end{tabular}
\bigskip
\caption{Features of Gefand-Cetlin fibers and toric fibers}\label{table_featuresofgefiber}
\end{table}
\end{center}
\vspace{-1.0cm}
				
The present paper serves as the foundation for subsequent papers of the authors on the study of 
symplectic topology of Gelfand-Cetlin fibers and relevant Lagrangian Floer theory.
	In the second paper \cite{CKO2}, using the description of Lagrangians faces in this article,
we prove that non-torus Lagrangian fibers satisfying certain Bohr-Sommerfeld conditions cannot be displaced by
any time-dependent Hamiltonian diffeomorphisms on monotone full flag manifolds.
The third paper \cite{CKO3} discusses the cotangent bundle of a homogeneous manifold which arises 
as a Gelfand-Cetlin Lagrangian fiber. Using \cite{CKO2}, we produce monotone or non-displaceable 
Lagrangians therein. 

This paper is organized as follows.
In Section~\ref{secTheGelfandCetlinSystems}, we review basic properties of Gelfand-Cetlin systems.
Section~\ref{secLadderDiagramAndItsFaceStructure} discusses the face structure of a Gelfand-Cetlin polytope and describe it in terms of certain subgraphs
of the ladder diagram associated with the polytope. Section~\ref{secLagrangianFibersOfGelfandCetlinSystems} is devoted to introducing combinatorics of ladder diagrams
that will be used to describe the Gelfand-Cetlin fibers and classify all Lagrangian faces.
In Section~\ref{secIteratedBundleStructuresOnGelfandCetlinFibers}, we provide the proof of Theorem~\ref{theoremA}.
Finally, the proof of Theorem~\ref{theoremB} will be given in Section~\ref{secDegenerationsOfFibersToTori}.

The material of the present paper is taken from Part I of our arXiv posting arXiv:1704.07213
which has been circulated since April 2017. Part II thereof forms the material of \cite{CKO2}.

\subsection*{Acknowledgements}
The first named author is supported by the National Research Foundation of Korea grant funded by the Korea government (MSIP; Ministry of Science, ICT \& Future Planning) (NRF-2017R1C1B5018168). The third named author is supported by IBS-R003-D1. This work was initiated when the first named author and the second named author were affiliated to IBS-CGP and supported by IBS-R003-D1.
The second named author would like to thank Cheol-Hyun Cho, Hansol Hong, Siu-Cheong Lau for useful discussions, Yuichi Nohara for explaining work with Ueda, and Byung Hee An, Morimichi Kawasaki and Fumihiko Sanda for helpful comments.

\section{Gelfand-Cetlin systems}
\label{secTheGelfandCetlinSystems}

In this section, we briefly overview Gelfand-Cetlin systems on partial flag manifolds.

For a given $r \in \mathbb{N}$ and an integer sequence such that 
	\begin{equation}\label{nidef}
		0 = n_0 < n_1 < n_2 < \cdots <n_r < n_{r+1} = n, 
	\end{equation}
the {\em partial flag manifold} $\F(n_1, \cdots, n_r; n)$ is the space of nested sequences of complex vector subspaces whose dimensions are $n_1, \cdots, n_r$, respectively.
That is,
	\[
		\F(n_1, \cdots, n_r; n) = \{V_{\bullet} := 0 \subset V_1 \subset \cdots \subset V_r \subset \C^n ~|~ \dim_{\C} V_i = n_i \mbox{ for $i= 1, \cdots, r$}\}.
	\]
An element $V_\bullet$ of $\F(n_1, \cdots, n_r; n)$ is called a {\em flag}. 
In particular, we call $\mcal{F}(1, 2, \cdots, n-1; n)$ the {\em full flag manifold} and denoted by $\mcal{F}(n)$.

The linear $U(n)$-action on $\C^n$ induces a transitive $U(n)$-action on $\F(n_1, \cdots, n_r; n)$ and each flag 
has an isotropy subgroup isomorphic to $U(k_1) \times \cdots \times U(k_{r+1})$ where
	\begin{equation}\label{kidef}
		k_i = n_i - n_{i-1}
	\end{equation}
for $i=1,\cdots,r+1$. Thus, $\F(n_1, \cdots, n_r; n)$ is a homogeneous space diffeomorphic to $U(n) / (U(k_1) \times \cdots \times U(k_{r+1}))$.
In particular, we have
	\begin{equation}\label{dimofflagma}
		\dim_{\R} \F(n_1, \cdots, n_r; n) = n^2 - \sum_{i=1}^{r+1} k_i^2.
	\end{equation}

\subsection{Description of $\F(n_1, \cdots, n_r; n)$ as a co-adjoint orbit of $U(n)$}~\label{ssecDescriptionOfFMathfrakNCoadjointOrbitOfUN}

Consider the conjugate action of $U(n)$ on itself. The action fixes the identity matrix $I_n \in U(n)$, which induces the $U(n)$-action,
called the {\em adjoint action} and denoted by $Ad$, on the Lie algebra $\mathfrak{u}(\mathfrak{n}) :=T_{I_n} U(n)$.
Note that $\uu(n)$ is the set of $(n \times n)$ skew-Hermitian matrices
	\[
		\uu(n) = \{ A \in M_n(\C) ~|~ A^* = -A \}, \quad A^* = \overline{A}^t
	\]
and the adjoint action can be written as
	\[
		\begin{array}{ccccc}
			Ad \colon & U(n) \times \uu(n)& \rightarrow & \uu(n) \\[0.1em]
	      	   &  (M, A) & \mapsto  &  MAM^*.\\
		\end{array}
	\]
The {\em co-adjoint action} $Ad^*$ of $U(n)$ is the action on the dual Lie algebra $\uu(n)^*$ induced by $Ad$, explicitly given by
	\[
		\begin{array}{ccccc}
			Ad^* \colon  & U(n) \times \uu(n)^*& \rightarrow & \uu(n)^* \\[0.1em]
	         	   &  (M, X) & \mapsto  &  X_M\\
		\end{array}
	\]
where $X_M \in \uu(n)^*$ is defined by $X_M(A) = X(M^*AM)$ for every $A \in \uu(n)$.

\begin{proposition}[p.51 in \cite{Au}]\label{proposition_herm_coadjoint}
	Let $\mathcal{H}_n = i\uu(n) \subset M_n(\C)$ be the set of
	$(n \times n)$ Hermitian matrices with the conjugate $U(n)$-action. Then there is a $U(n)$-equivariant $\R$-vector space isomorphism
	$\phi \colon \mathcal{H}_n \rightarrow \uu(n)^*$.
\end{proposition}

Henceforth, we always think of the co-adjoint action of $U(n)$ on $\uu(n)^*$ as the conjugate $U(n)$-action on $\mathcal{H}_n$.
Let $\lambda = \{\lambda_1, \cdots, \lambda_n\}$ be a non-increasing sequence of real numbers such that
	\begin{equation}\label{lambdaidef}
		\lambda_1 = \cdots = \lambda_{n_1} > \lambda_{n_1 + 1} = \cdots = \lambda_{n_2} > \cdots > \lambda_{n_r +1} =
		\cdots = \lambda_{n_{r+1}} (= \lambda_n)
	\end{equation}
and let $I_\lambda = \text{diag}(\lambda_1, \cdots, \lambda_n) \in \mathcal{H}_n$ be the diagonal matrix whose $i$-th diagonal entry is $\lambda_i$ for $i=1,\cdots,n$.
Then the isotropy subgroup of $I_\lambda$ is given by 
	\[
		\begin{pmatrix}
			U(k_1) & 0 & \cdots & 0 \\
			0	& U(k_2) & \cdots & 0 \\
			\vdots & \vdots & \ddots & \vdots \\
			0 & 0 & \cdots &  U(k_{r+1}) 
		\end{pmatrix} \subset U(n).
	\]
If we denote by $\mathcal{O}_\lambda$ the $U(n)$-orbit of $I_\lambda$, then we have 
\[
	\mcal{O}_\lambda \cong U(n) / \left(U(k_1) \times \cdots \times U(k_{r+1}) \right),
\] which is diffeomorphic to $\F(n_1, \cdots, n_r; n)$. We call $\mcal{O}_\lambda$ the \emph{co-adjoint orbit associated with eigenvalue pattern} $\lambda$.

\begin{remark}\label{remark_property_hermitian_matrix}
	Any two similar matrices have the same eigenvalues with the same multiplicities and any Hermitian matrix is unitarily diagonalizable.
	Thus the co-adjoint orbit $\mathcal{O}_{\lambda}$ is the set of all Hermitian matrices having the eigenvalue pattern $\lambda$ respecting the multiplicities.
\end{remark}

\vspace{0.2cm}
\subsection{Symplectic structure on a co-adjoint orbit $\mathcal{O}_{\lambda}$}~\label{ssecSymplecticStructureOnMathcalOLambda}

For any compact Lie group $G$ with the Lie algebra $\mathfrak{g}$ and for any dual element $\lambda \in \mathfrak{g}^*$,
there is a {\em canonical} $G$-invariant symplectic form $\omega_{\lambda}$, called  the \emph{Kirillov-Kostant-Souriau symplectic form} (or \emph{KKS form} shortly),
on the orbit $\mathcal{O}_{\lambda}$. 
Furthermore, $\mathcal{O}_\lambda$ admits a unique $G$-invariant K\"{a}hler metric compatible with $\omega_\lambda$, and therefore
$(\mathcal{O}_\lambda, \omega_\lambda)$ forms a K\"{a}hler manifold. We refer the reader to \cite[p.150]{Br} for more details.

The KKS form $\omega_\lambda$ can be described more explicitly 
in the case where $G = U(n)$ as below.
For each $h \in \mathcal{H}_n$, we define a real-valued skew-symmetric bilinear form $\widetilde{\omega}_h$ 
on $\uu(n) = i\mathcal{H}_n$ by
	\[
		\widetilde{\omega}_h(X,Y) := \text{tr}(ih[X,Y]) = \text{tr}(iY[X,h]), \quad X, Y \in \uu(n).
	\]
The kernel of $\widetilde{\omega}_h$ is then
	\[
		\text{ker} ~\widetilde{\omega}_h = \{ X \in \uu(n) ~|~ [X,h] = 0 \}.
	\]
Since $\mathcal{O}_\lambda$ is a homogeneous $U(n)$-space, we may express each tangent space $T_h \mathcal{O}_\lambda$ at a point $h \in \mcal{O}_\lambda$ as 
	\[
		T_h \mathcal{O}_\lambda = \{ [X,h] \in T_h \mathcal{H}_n = \mathcal{H}_n~|~ X \in \uu(n) \}.
	\]
Then we get a non-degenerate two form $\omega_\lambda$ on $\mathcal{O}_\lambda$ defined by
	\[
		\left(\omega_{\lambda}\right)_h([X,h], [Y,h]) := \widetilde{\omega}_h(X,Y), \quad h \in \mathcal{O}_\lambda, \quad X,Y \in \uu(n).
	\]
The closedness of $\omega_\lambda$ then follows from the Jacobi identity on $\uu(n)$, see \cite[p.52]{Au} for more details.

\begin{remark}
The diffeomorphism type of $\mathcal{O}_\lambda$ does not depend on
the choice of $\lambda$ but on $k_\bullet$'s (or $n_\bullet$'s). However, the symplectic form $\omega_\lambda$
depends on the choice of $\lambda$. Indeed, assuming $\lambda$ being integral, the choice of $\lambda$ determines a projective embedding of $\mcal{O}_\lambda$. 
For instance, two co-adjoint orbits $\mathcal{O}_{\lambda}$ and $\mathcal{O}_{\lambda'}$
have $k_1 = k_2 = 1$ when $\lambda = \{1,-1\}$ and $\lambda' = \{1,0\}$ and both orbits
are diffeomorphic to $U(2) / \left( U(1) \times U(1)\right) \cong \p^1$.
However, the symplectic area of $(\mathcal{O}_\lambda, \omega_\lambda)$ and $(\mathcal{O}_{\lambda'}, \omega_{\lambda'})$ are two
and one, respectively. 
\end{remark}

Any partial flag manifold is known to be a Fano manifold and hence it admits a monotone K\"{a}hler form.
The following proposition gives a complete description of the monotonicity of $\omega_\lambda$.
\begin{proposition}[p.653-654 in \cite{NNU1}]\label{proposition_monotone_lambda}
	The symplectic form $\omega_\lambda$ on $\mathcal{O}_\lambda$ satisfies
	\[
		c_1(T\mathcal{O}_\lambda) = [\omega_\lambda]
	\]
	if and only if
\[
  \lambda = (\underbrace{n-n_1, \cdots}_{k_1} ~,
  \underbrace{n-n_1-n_2, \cdots}_{k_2} ~, \cdots ~, \underbrace{n-n_{r-1}-n_r, \cdots}_{k_r} ~, \underbrace{-n_r, \cdots, -n_r}_{k_{r+1}} )
  +  (\underbrace{m, \cdots, m}_{n}),
  \label{canonical}
\]
for some $m \in \R$.
\end{proposition}

\vspace{0.2cm}
\subsection{Completely integrable system on $\mathcal{O}_\lambda$}~
\label{ssecCompletelyIntegrableSystemOnMathcalOLambda}

We adorn a co-adjoint orbit $(\mcal{O}_\lambda, \omega_\lambda)$ with a completely integrable system, called the Gelfand-Cetlin system. We recall a standard definition of a completely integrable system.

\begin{definition}\label{definition_CIS_smooth}
A {\em completely integrable system} on a $2n$-dimensional
symplectic manifold $(M,\omega)$ is an $n$-tuple of smooth functions
	\[
		\Phi := (\Phi_1, \cdots, \Phi_n) \colon M \rightarrow \R^n
	\]
such that
\begin{enumerate}
	\item $\{\Phi_i, \Phi_j\} = 0$ for every $1 \leq i, j \leq n$ and
	\item $d\Phi_1, \cdots, d\Phi_n$ are linearly independent on an open dense subset of $M$.
\end{enumerate}
\end{definition}

If $\Phi$ is a proper map, the Arnold-Liouville theorem states that for any regular value $\textbf{\textup{u}} \in \R^n$ of $\Phi$ the connected component of the 
preimage $\Phi^{-1}(\textbf{\textup{u}})$ is a Lagrangian torus. However, if $\textbf{\textup{u}}$ is a critical value, the fiber might not be a manifold in general.

Harada and Kaveh \cite{HK} proved that any smooth projective variety equips a completely integrable system whenever it admits a flat toric degeneration. 
(See Section \ref{secDegenerationsOfFibersToTori} for more details.)
However, the terminology ``completely integrable system'' used in \cite{HK} is a weakened version of Definition \ref{definition_CIS_smooth} in the following sense. 

\begin{definition}\label{definition_CIS_continuous}
	A {\em (continuous) completely integrable system} on a $2n$-dimensional symplectic manifold $(M,\omega)$ is a collection of $n$ \emph{continuous} functions
		\[
			\Phi := \{\Phi_1, \cdots, \Phi_n\} \colon M \rightarrow \R^n
		\]
	such that there exists an open dense subset $\mathcal{U}$ of $M$ on which
	\begin{enumerate}
		\item each $\Phi_i$ is smooth,
		\item $\{\Phi_i, \Phi_j\} = 0$ for every $1 \leq i, j \leq n$, and
		\item $ d\Phi_1, \cdots, d\Phi_n $ are linearly independent.
	\end{enumerate}
\end{definition}

For any co-adjoint orbit $(\mathcal{O}_\lambda, \omega_\lambda)$, Guillemin and Sternberg \cite{GS2} constructed a completely integrable system (in the sense of Definition \ref{definition_CIS_continuous})
	\[
		\Phi_\lambda : \mathcal{O}_\lambda \rightarrow \R^{\dim_{\C} \mathcal{O}_\lambda},
	\]
called {\em the Gelfand-Cetlin system on $\mathcal{O}_\lambda$} (\emph{GC system} for short) with respect to the KKS symplectic form $\omega_\lambda$. 
The GC system on $\mathcal{O}_\lambda$ is in general continuous but not smooth. From now on, a completely integrable system will be meant to be a conitnuous completely integrable system in Definition~\ref{definition_CIS_continuous} unless stated otherwise.

We briefly recall the construction of the GC system on $\mathcal{O}_\lambda$ as follows. (See also \cite[p.7-9]{NNU1}.)
Let $n_\bullet$'s and $k_\bullet$'s be given in~\eqref{nidef} and~\eqref{kidef}, respectively, and 
let $\lambda$ be a non-increasing sequence satisfying \eqref{lambdaidef}.
From~\eqref{dimofflagma}, it follows that
 	\[
 		\dim_{\R} \mathcal{O}_\lambda = 2\dim_{\C} \mathcal{O}_\lambda = n^2 - \sum_{i=1}^{r+1} k_i^2.
 	\]
For any $x \in \mathcal{O}_\lambda \subset \mathcal{H}_n$, let $x^{(k)}$ be the $k \times k$ leading principal submatrix of $x$ for each $k=1,\cdots,n-1$.
Since $x^{(k)}$ is also a Hermitian matrix, the eigenvalues are all real. 
Let
\begin{equation}\label{equation_index_global}
	\mcal{I} = \{(i,j) ~|~ i,j \in \N, ~i+j \leq n + 1\}
\end{equation}
be an index set (where $|\mcal{I}| = \frac{n(n+1)}{2}$). We then define the real-valued function
	\[
		\Phi_\lambda^{i,j}\colon \mcal{O}_\lambda \to \R, \quad (i,j) \in \mcal{I}
	\]
where $\Phi_\lambda^{i,j}(x)$ is defined to be the $i$-th largest eigenvalue of $x^{(i+j-1)}$. In particular, the eigenvalues of $x^{(k)}$ are arranged by
\[
	\Phi_\lambda^{1,k}(x) \geq \Phi_\lambda^{2,k-1}(x) \geq \cdots \geq \Phi_\lambda^{k,1}(x)
\]
in descending order. Collecting all $\Phi_\lambda^{i,j}$'s for $(i,j) \in \mcal{I}$, 
we obtain $\Phi_\lambda$.
	
\begin{definition}\label{definition_GC_system}
Let $\lambda$ be given in \eqref{lambdaidef}. 
The {\em Gelfand-Cetlin system $\Phi_\lambda$ associated with $\lambda$} is defined by the collection of real-valued functions
\begin{equation}\label{philambdaij}
	\Phi_\lambda := \left( \Phi_\lambda^{i,j} \right)_{(i,j) \in \mcal{I}} \colon \mathcal{O}_\lambda \rightarrow \R^{\frac{n(n+1)}{2}}.
\end{equation}
\end{definition}

\begin{remark}
We label each component of $\Phi_\lambda$ by multi-index $(i,j) \in \mcal{I}$ such that $\Phi_\lambda^{i,j}$ corresponds to the lattice point $(i,j) \in \Z^2$ 
in a ladder diagram in Definition \ref{definition_ladder_diagram}. (See also Figure \ref{figure_GC_to_ladder}.)
Notice that the labeling of components of $\Phi_\lambda$ used in \cite{NNU1} is different from ours. 
\end{remark}

Now, we consider the coordinate system of $\R^{n(n+1)/2}$ 
\begin{equation}\label{equation_coordinate}
\left\{ \left( u_{i,j} \right) \in \R^{n(n+1)/2} ~|~ (i,j) \in \mcal{I} \right\}. 
\end{equation}
For $x \in \mcal{O}_\lambda$, if $(u_{i,j})_{(i,j) \in \mcal{I}} = \Phi_\lambda(x)$, that is, $u_{i,j} = \Phi_\lambda^{i,j}(x)$ for each $(i,j) \in \mcal{I}$, 
the min-max principle implies that the components $u_{i,j}$'s satisfy the following pattern : 

\vspace{0.2cm}
\begin{equation} 
\begin{alignedat}{17}
  \lambda_1 &&&& \lambda_2 &&&& \lambda_3 && \cdots && \lambda_{n-1} &&&& \lambda_n  \\
  & \uge && \dge && \uge && \dge &&&&&& \uge && \dge & \\
  && u_{1,n-1} &&&& u_{2, n-2} &&&&&&&& u_{n-1, 1} && \\
  &&& \uge && \dge &&&&&&&& \dge &&& \\
  &&&& u_{1, n-2} &&&&&&&& u_{n-2, 1} &&&& \\
  &&&&& \uge &&&&&& \dge &&&&& \\
  &&&&&& \dndots &&&& \updots &&&&&& \\
  &&&&&&& \uge && \dge &&&&&&& \\
  &&&&&&&& u_{1,1} &&&&&&&&&
\end{alignedat}
\label{equation_GC-pattern}
\vspace{0.2cm}
\end{equation}

Note that by definition, $\Phi_\lambda^{i, n+1-i} \equiv \lambda_i$ is a constant function for each $i=1,\cdots, n$.
More generally, since $\lambda_{n_{i-1} + 1} = \cdots =  \lmd {n_i}$ by our assumption~\eqref{lambdaidef}, 
the min-max principle \eqref{equation_GC-pattern} implies that 
\[
\Phi_\lambda^{j, n + 1 -k}\equiv \lambda_{n_i}
\]
is a constant function on $\mcal{O}_\lambda$ for all $j = n_{i-1}+1, \cdots, n_i$ and each $k = j, \cdots, n_i $. 
Therefore, there are exactly $\frac{1}{2}(n^2 - \sum_{i=1}^{r+1} k_i^2)$ non-constant
functions among $\{\Phi^{i,j}_\lambda\}_{(i,j) \in \mcal{I}}$ on $\mathcal{O}_\lambda$.
Let
\begin{equation}\label{collectionofindicesset}
	\mcal{I}_\lambda := \{ (i,j) \in \mcal{I} ~|~ \Phi^{i,j}_\lambda \textup{ is \emph{not} a constant function on $\mcal{O}_\lambda$}\}. 
\end{equation}
Collecting all non-constant components of $\Phi_\lambda$, we rename 
\begin{equation}\label{Gelfand-Cetlinsystemdef}
\Phi_\lambda = \left( \Phi^{i,j}_\lambda \right)_{(i,j) \in \mcal{I}_\lambda} \colon \mathcal{O}_\lambda \to \R^{| \mcal{I}_\lambda |}, \quad 
| \mcal{I}_\lambda | = \dim_{\C} \mathcal{O}_\lambda = \frac{1}{2}(n^2 - \sum_{i=1}^{r+1} k_i^2)
\end{equation}
as the {\em Gelfand-Cetlin system}.
By abuse of notation, the collection is still denoted by $\Phi_\lambda$.
Guillemin and Sternberg \cite{GS2} prove that $\Phi_\lambda$ satisfies all properties given in Definition \ref{definition_CIS_continuous},
and hence it is a completely integrable system on $\mathcal{O}_\lambda$ \footnote{In general, the Gelfand-Cetlin system is never smooth on the whole space $\mathcal{O}_\lambda$ unless $\mathcal{O}_\lambda$ is a projective space.}.
We will not distinguish two notations ~\eqref{philambdaij} and~\eqref{Gelfand-Cetlinsystemdef} unless any confusion arises.

\begin{definition}
	The {\em Gelfand-Cetlin polytope}\footnote{It is straightforward to see that $\Delta_\lambda$ is a convex polytope, since $\Delta_\lambda$ is the intersection
	of half-spaces
	defined by inequalities in (\ref{equation_GC-pattern}).} $\Delta_\lambda$ (or {\em GC polytope} for short) is the collection of points $(u_{i,j})$ satisfying (\ref{equation_GC-pattern}). 
\end{definition}

\begin{proposition}[\cite{GS2}]
	The Gelfand-Cetlin polytope $\Delta_\lambda$ coincides with the image of $\mcal{O}_\lambda$ under $\Phi_\lambda$. 
\end{proposition}

\vspace{0.2cm}
\subsection{Smoothness of $\Phi_\lambda$}~
\label{ssecSmoothnessOfPhiLambda}

Let $\lambda$ be given in \eqref{lambdaidef} and let $\Phi_\lambda$ be the GC system 
on $(\mcal{O}_\lambda, \omega_\lambda)$. In general, $\Phi_\lambda$ is \emph{not} smooth on the whole 
$\mcal{O}_\lambda$. However, the following proposition due to Guillemin-Sternberg 
states that $\Phi_\lambda$ is smooth on $\mcal{O}_\lambda$ {\em almost everywhere}.
\begin{proposition}[Proposition 5.3, p.113, and p.122 in \cite{GS2}]\label{proposition_GS_smooth}
	For each $(i,j) \in \mcal{I}_\lambda$, the component $\Phi_\lambda^{i,j}$ is smooth at 
	$z \in \mcal{O}_\lambda$ if 
	\begin{equation}\label{equation_smooth_region}
		\Phi_\lambda^{i+1, j}(z) < \Phi_\lambda^{i, j}(z) < \Phi_\lambda^{i, j+1}(z). 
	\end{equation}
	In particular, $\Phi_\lambda$ is smooth on the open dense subset 
	$\Phi_\lambda^{-1}(\mathring{\Delta}_\lambda)$ of $\mathcal{O}_\lambda$ where $\mathring{\Delta}_\lambda$ is the interior of $\Delta_\lambda$.
	Furthermore, $d\Phi_\lambda^{i,j}(z) \neq 0$ for every point $z$ satisfying \eqref{equation_smooth_region}.  
\end{proposition}

One important remark is that a Hamiltonian trajectory of each $\Phi_\lambda^{i,j}$ passing through a point 
$z \in \mcal{O}_\lambda$ satisfying \eqref{equation_smooth_region} is periodic with integer period. 
Therefore, each $\Phi_\lambda^{i,j}$ generates a Hamiltonian circle action on the subset of $\mathcal{O}_\lambda$ 
on which $\Phi_\lambda^{i,j}$ is smooth.
See \cite[Theorem 3.4 and Section 5]{GS2} for more details.

\vspace{0.2cm}
\section{Ladder diagram and its face structure}
\label{secLadderDiagramAndItsFaceStructure}

In order to visualize a GC polytope, it is convenient to use an alternative description of its face structure 
in terms of certain graphs in the ladder diagram provided by the first named author with An and Kim in \cite{ACK}. 
The goal of this section is to review their description of the face structure.

We begin by the definition of a ladder diagram.
Let $\lambda = \{\lambda_1, \cdots, \lambda_n\}$ be given in \eqref{lambdaidef}. 
Then $\lambda$ uniquely determines $n_\bullet$'s and $k_\bullet$'s in~\eqref{nidef} and~\eqref{kidef}, respectively. 

\begin{definition}[\cite{BCKV}, \cite{NNU1}]\label{definition_ladder_diagram}
Let $\Gamma_{\Z^2} \subset \R^2$ be the square grid graph satisfying
\begin{enumerate}
\item its vertex set is $\Z^2 \subset \R^2$ and
\item each vertex $(a,b) \in \Z^2$ connects to exactly four vertices $(a, b \pm 1)$ and $(a \pm 1, b)$.
\end{enumerate}
The {\em ladder diagram} $\Gamma (n_1, \cdots, n_r; n)$ is defined as the induced subgraph of $\Gamma_{\Z^2}$ 
that is formed from the set $V_{\Gamma (n_1, \cdots, n_r; n)}$ of vertices given by
\[
	V_{\Gamma (n_1, \cdots, n_r; n)} := \bigcup_{j=0}^r \,\, \left\{ (a,b) \in \Z^2 ~\big{|}~ \, (a,b) \in [n_j, n_{j+1}] \times [0,n-n_{j+1}] \right\}.
\]
As $\lambda$ determines $n_\bullet$'s, we simply denote $\Gamma(n_1, \cdots, n_r; n)$ by $\Gamma_\lambda$. We call $\Gamma_\lambda$ the \emph{ladder diagram associated with} $\lambda$.
\end{definition}

\begin{figure}[H]
	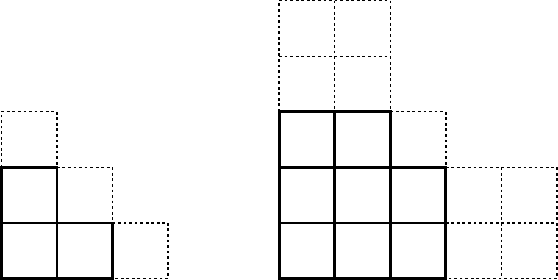
	\bigskip
	\caption{\label{figure_ld123} Ladder diagrams $\Gamma(1,2;3)$ and $\Gamma(2,3;5)$.}	
\end{figure}
\vspace{-0.2cm}

We call the vertex of $\Gamma_\lambda$ located at $(0,0)$ {\em the origin}.
Also, we call $v \in V_\Gamma$ a {\em top vertex} if $v$ is a farthest vertex from the origin 
with respect to the taxicab metric.
Equivalently, a vertex $v = (a,b) \in V_\Gamma$ is a top vertex if $a+b = n$. 

\vspace{0.2cm}
\begin{figure}[H]
	\scalebox{0.9}{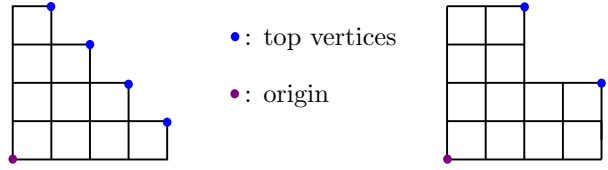}
	\vspace{0.2cm}
	\caption{\label{figure_top_bottom} Top vertices for $\Gamma(1,2,3,4;5)$ and $\Gamma(2,4;6)$.}	
\end{figure}
\vspace{-0.3cm}		

\begin{definition}[Definition 2.2 in \cite{BCKV}]
	A {\em positive path} on a ladder diagram $\Gamma_\lambda$ is a shortest path from the origin to some top vertex in $\Gamma_\lambda$.
\end{definition}

Now, we define the face structure of $\Gamma_\lambda$ as follows.

\begin{definition}[Definition 1.5 in \cite{ACK}]\label{definition_face}
	Let $\Gamma_\lambda$ be a ladder diagram.
	\begin{itemize}
		\item A subgraph $\gamma$ of $\Gamma_\lambda$ is called a {\em face} of $\Gamma_\lambda$ if
			\begin{enumerate}
				\item $\gamma$ contains all top vertices of $\Gamma_\lambda$, 
				\item $\gamma$ can be represented by a union of positive paths.
			\end{enumerate}
		\item For two faces $\gamma$ and $\gamma'$ of $\Gamma_\lambda$, $\gamma$ is said to be a {\em face} of $\gamma'$ if $\gamma \subset \gamma'$.
		\item The \emph{dimension} of a face $\gamma$ is defined by
		\[
			\dim \gamma := \text{rank}_\Z~H_1(\gamma; \Z),
		\]
		regarding $\gamma$ as a $1$-dimensional CW-complex. In other words, $\dim \gamma$ is the number of minimal cycles in $\gamma$.
	\end{itemize}
\end{definition}

It is straightforward from Definition~\ref{definition_face} that for any two faces $\gamma$ and
$\gamma'$ of $\Gamma_\lambda$, $\gamma \cup \gamma'$ is also a face of $\Gamma_\lambda$, which
is the smallest face containing $\gamma$ and $\gamma'$.
The following theorem tells us that the face structures of $\Delta_\lambda$ and $\Gamma_\lambda$ are equivalent. 
\begin{theorem}[Theorem 1.9 in \cite{ACK}]\label{theorem_equiv_CG_LD}
For a sequence $\lambda$ of real numbers given in~\eqref{lambdaidef},
let $\Gamma_\lambda$ be the ladder diagram and $\Delta_\lambda$ be the associated Gelfand-Cetlin polytope. 
Then there exists a bijective map
	\[
			\{~\text{faces of}~\Gamma_\lambda \} \stackrel{\Psi} \longrightarrow \{~\text{faces of} ~\Delta_\lambda\}
	\]
such that for faces $\gamma$ and $\gamma'$ of $\Gamma_\lambda$
\begin{itemize}
	\item (Order-preserving) $\gamma \subset \gamma'$ if and only if $\Psi(\gamma) \subset \Psi(\gamma')$,
	\item (Dimension) $\dim \Psi(\gamma) = \dim \gamma$.
\end{itemize}
\end{theorem}

\begin{example}\label{example_123}
	Let $\lambda = \{\lambda_1, \lambda_2, \lambda_3 \}$ with $\lambda_1 > \lambda_2 > \lambda_3$.
	The co-adjoint orbit $(\mcal{O}_\lambda, \omega_\lambda)$ is diffeomorphic to the full flag manifold $\F(3)$.
	Let $\Gamma_\lambda$ be the ladder diagram associated with $\lambda$ as in Figure
	\ref{figure_ldF123}.
	Here, the blue dots are top vertices and the purple dot is the origin of $\Gamma_\lambda$.
	\begin{figure}[ht]
		\scalebox{1.2}{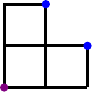}
		\caption{\label{figure_ldF123} Ladder diagram $\Gamma_\lambda$.}	
	\end{figure}	
		
	The zero, one, two, and three-dimensional faces of $\Gamma_\lambda$ are respectively listed in Figure \ref{figure_zero_dim_face_F123}, \ref{figure_one_dim_face_F123}, \ref{figure_two_dim_face_F123}, and \ref{figure_three_dim_face_F123}.
	Here,
	$v_i$ denotes a vertex for $i \in \{1, \cdots, 7\}$,
	$e_{ij}$ is the edge containing $v_i$ and $v_j$, $f_I$ is  the facet
	containing all $v_i$'s for $i \in I$, and $I_{1234567}$ is the three dimensional face, i.e., the whole $\Gamma_\lambda$.

	\vspace{0.5cm}
	\begin{figure}[ht]
		\scalebox{1.2}{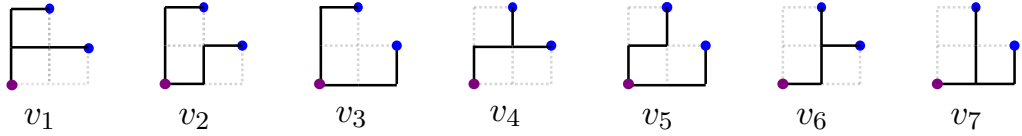}
		\caption{\label{figure_zero_dim_face_F123} The zero-dimensional faces of $\Gamma_\lambda$.}	
	\end{figure}		
	
	\begin{figure}[ht]
		\scalebox{1.2}{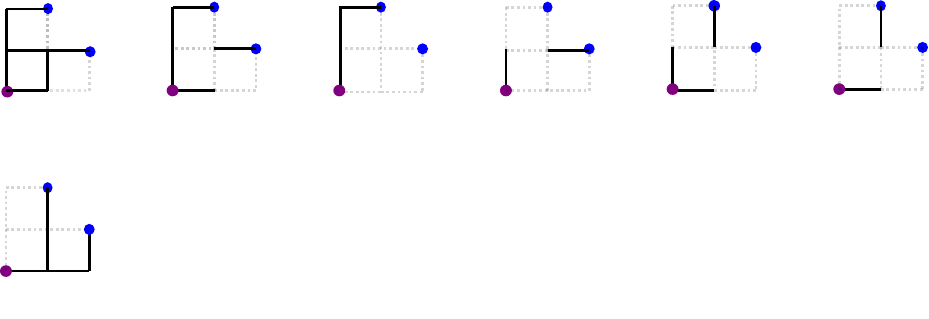}
		\caption{\label{figure_one_dim_face_F123} The one-dimensional faces of $\Gamma_\lambda$.}	
	\end{figure}		
	
	\begin{figure}[H]
		\scalebox{1.2}{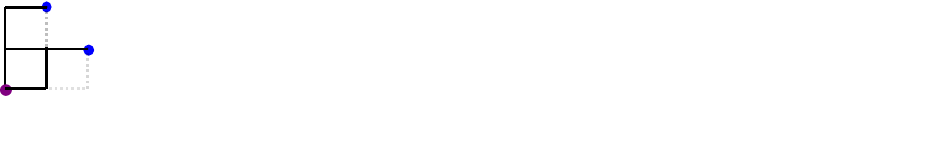}
		\bigskip
		\caption{\label{figure_two_dim_face_F123} The two-dimensional faces of $\Gamma_\lambda$.}	
	\end{figure}		

	\vspace{-0.5cm}	
	\begin{figure}[H]
		\scalebox{1.2}{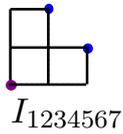}
		\bigskip
		\caption{\label{figure_three_dim_face_F123} The three-dimensional face of $\Gamma_\lambda$.}	
	\end{figure}

	The GC polytope $\Delta_\lambda$ is given in Figure \ref{figure_GC_moment_123}.
	(See Figure 5 in \cite{K} or Figure 4 in \cite{NNU1}.)
	We can easily see that the correspondence $\Psi(v_i) = w_i$ of vertices naturally
	extends to the set of faces of $\Gamma_\lambda$, satisfying
	the conditions in Theorem \ref{theorem_equiv_CG_LD}.
		
	\begin{figure}[ht]
	\scalebox{1.2}{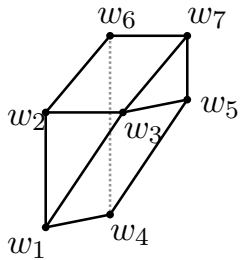}
	\caption{\label{figure_GC_moment_123} The GC polytope $\Delta_\lambda$ for $\lambda = \{ \lambda_1 > \lambda_2 > \lambda_3 \}$.}	
	\end{figure}		
\end{example}

For our convenience, we describe each point in $\Delta_\lambda$ by using $\Gamma_\lambda$ \emph{with a filling}, putting each component ${u_{i,j}}$ of the coordinate system 
\eqref{equation_coordinate} of $\R^{|\mcal{I}|} = \R^{\frac{n(n+1)}{2}}$ into the unit box $\square^{(i,j)}$ whose top-right vertex is $(i,j)$.
The GC pattern~\eqref{equation_GC-pattern} implies that $\{u_{i,j} \}_{(i,j) \in \mcal{I}}$ is 
\begin{equation}\label{rem:diagram-pattern}
\begin{cases}
(1) \,\, \text{\emph{increasing} along the columns of $\Gamma_\lambda$ } and \\
(2) \,\, \text{\emph{decreasing} along the rows of $\Gamma_\lambda$}.
\end{cases}
\end{equation}

\begin{example}
	Let $\mcal{O}_\lambda \simeq \mathcal{F}(3)$ be the co-adjoint orbit from Example~\ref{example_123}. 
	Recall that the pattern~\eqref{equation_GC-pattern} consists of the following inequalities:
	$$
		u_{1,2} \geq u_{1,1},\,\,\, u_{1,1} \geq u_{2,1},\,\,\, \lambda_1 \geq u_{1,2},\,\,\, u_{1,2} \geq \lambda_2,\,\,\, \lambda_2 \geq u_{2,1},\,\,\, u_{2,1} \geq \lambda_3.
	$$
	The ladder diagram $\Gamma_\lambda$ with a filling by the variables $u_{i,j}$'s is as in Figure~\ref{figure_GC_to_ladder}.
	\begin{figure}[H]
	\scalebox{0.9}{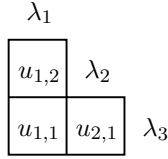}
	\caption{\label{figure_GC_to_ladder} The ladder diagram $\Gamma_\lambda$ with $u_{i,j}$'s variables.}	
\end{figure}	
\end{example}

Now, we briefly explain how the map $\Psi$ in Theorem \ref{theorem_equiv_CG_LD}
is defined. For a given  face $\gamma$ of $\Gamma_\lambda$, consider $\gamma$ with a filling by the variables $u_{i,j}$'s for ${(i,j) \in \mcal{I}}$. 
The image of $\gamma$ under $\Psi$ is the intersection of facets supported by the hyperplanes that are given by equating two adjacent variables $u_{i,j}$'s \emph{not} divided by any positive paths in $\gamma$. 

\begin{example}
Suppose that $\lambda = \{4,4,3,2,1\}$ and let $\gamma$ be a face given as in Figure \ref{figure_one_to_one_face}.
Then, the corresponding face $\Psi(\gamma)$ in $\Delta_\lambda$ is defined by
\[
	\Psi(\gamma) =  \Delta_\lambda \cap  \{ u_{2,1} = u_{1,1} \} \cap  \{ u_{1,1} = u_{1,2} \} \cap  \{ u_{2,2} = u_{1,2} \} \cap  \{ u_{2,1} = u_{2,2} \} \cap  \{ u_{2,2} = u_{2,3} \} \cap  \{ u_{3,1} = 4 \} .
\]
\begin{figure}[H]
	\scalebox{0.9}{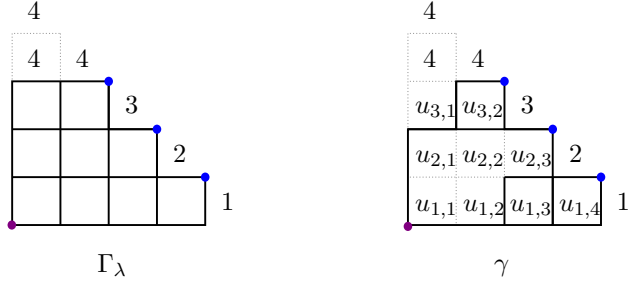}
	\caption{\label{figure_one_to_one_face} Correspondence of faces of $\Delta_\lambda$ and $\Gamma_\lambda$}
\end{figure}
\end{example}

\section{Classification of Lagrangian fibers}
\label{secLagrangianFibersOfGelfandCetlinSystems}

Our first main theorem \ref{theoremA} states that every fiber of the GC system $\Phi_\lambda$ 
is an isotropic submanifold of $(\mathcal{O}_\lambda, \omega_\lambda)$ and is the total space of certain iterated bundle
\begin{equation}\label{cha4iteratedbund}
		E_{n-1} \stackrel {p_{n-1}} \longrightarrow E_{n-2} \stackrel{p_{n-2}} \longrightarrow \cdots \stackrel{p_2} \longrightarrow E_1
		\stackrel{p_1} \longrightarrow E_0= \mathrm{point}
\end{equation}
such that the fiber at each stage is either a point or a product of odd dimensional spheres.
In this section, we provide a combinatorial way of ``reading off'' the topology of the fiber of each projection map $p_i$ from the ladder diagram (Theorem \ref{theorem_main}).
Furthermore, we classify all positions of Lagrangian fibers in the Gelfand-Cetlin polytope $\Delta_\lambda$
(Corollary~\ref{corollary_L_fillable}).

We first consider a $2n$-dimensional compact symplectic toric manifold $(M,\omega)$ with a moment map $\Phi \colon M \to \R^n.$
It is a smooth completely integrable system on $M$ in the sense of Definition \ref{definition_CIS_smooth} and the 
Atiyah-Guillemin-Sternberg convexity theorem \cite{At, GS} yields that the image $\Delta := \Phi(M)$ is an $n$-dimensional convex polytope.
It is well-known that for any $k$-dimensional face $f$ of $\Delta$ and a point $\textbf{\textup{u}} \in \mathring{f}$ in its relative interior $\mathring{f}$ of $f$,
the fiber $\Phi^{-1}(\textbf{\textup{u}})$ is a $k$-dimensional isotropic torus. 
In particular, a fiber $\Phi^{-1}(\textbf{\textup{u}})$ is Lagrangian if and only if $\textbf{\textup{u}} \in \mathring{\Delta}$. .

In contrast, in the GC system case the preimage of a point in the inteior of a $k$-th dimensional face of the GC polytope $\Delta_\lambda$ might 
have the dimension greater than $k$.
In particular, $\Phi_\lambda$ might admit a Lagrangian fiber over a point not contained in the interior of $\Delta_\lambda$.

\begin{definition}\label{definition_lagrangian_face}
	We call a face $f$ of $\Delta_\lambda$ {\em Lagrangian} if it contains a point $\textbf{\textup{u}}$ in its relative interior $\mathring{f}$
	such that the fiber $\Phi_\lambda^{-1}(\textbf{\textup{u}})$ is Lagrangian. Also, we call a face $\gamma$ of $\Gamma_\lambda$ \emph{Lagrangian} if the corresponding face
	$f_\gamma := \Psi(\gamma)$ of $\Delta_\lambda$ is Lagrangian where $\Psi$ is given in Theorem~\ref{theorem_equiv_CG_LD}.
\end{definition}

\begin{remark}
We will see later that if $\textbf{\textup{u}}$ and $\textbf{\textup{u}}'$ are contained in the relative interior of a same face of $\Delta_\lambda$, then 
$\Phi_\lambda^{-1}(\textbf{\textup{u}})$ and $\Phi_\lambda^{-1}(\textbf{\textup{u}}')$ are diffeomorphic.
In particular if $f$ is a Lagrangian face of $\Delta_\lambda$, then every fiber over any point in $\mathring{f}$ is Lagrangian, see Corollary~\ref{corollary_L_fillable}.
\end{remark}

\begin{example}[\cite{K,NNU1}]\label{example_NNU_S3}
	For a full flag manifold $\F(3) \simeq \mcal{O}_\lambda$ in Example~\ref{example_123}, the fiber $\Phi_\lambda^{-1}(w_3)$ is 
	a Lagrangian submanifold diffeomorphic to $S^3$ where the vertex $w_3$ is given in Figure~\ref{figure_GC_moment_123}.
	See Example 3.8 in \cite{NNU1}.
\end{example}

From now on,  we tacitly identify faces of the ladder diagram $\Gamma_\lambda$ with faces of the Gelfand-Cetlin polytope $\Delta_\lambda$ via the map $\Psi$ in Theorem \ref{theorem_equiv_CG_LD}.
For example, ``a point $r$ in a face $\gamma$'' of $\Gamma_\lambda$ means a point $r$ in the face $\Psi(\gamma) = f_\gamma$ of $\Delta_\lambda$.

\vspace{0.2cm}

\subsection{$W$-shaped blocks and $M$-shaped blocks}~
\label{ssecWShapedBlocks}

For each $(a,b) \in \Z^2$, let $\square^{(a,b)}$ be the simple closed region bounded by the unit square in $\R^2$
whose vertices are lattice points in $\Z^2$ and the top-right vertex is $(a,b)$. The region $\square^{(a,b)}$ is simply said to be the \emph{box} at $(a, b)$.
\begin{definition}\label{definition_W_shaped_block}
	For each integer $k \geq 1$, the {\em $k$-th $W$-shaped block} denoted by $W_k$, or simply {\em the $W_k$-block}, is defined by
	\[
		W_k := \bigcup_{(a,b)} \square^{(a,b)}
	\]
	where the union is taken over all $(a,b)$'s in $(\Z_{\geq 1})^2$ such that $k+1 \leq a+b \leq k+2$.
	A lattice point closest from the origin (with respect to the taxicab metric) in the $W_k$-block is called a \emph{bottom vertex} of $W_k$.
\end{definition}

	The following figures illustrate $W_1$, $W_2$, and $W_3$ where the red dots in each figure
	indicate the vertices over which the union is taken in Definition \ref{definition_W_shaped_block}. The purple dots are bottom vertices of each $W$-blocks.
	\begin{figure}[H]
	\scalebox{0.9}{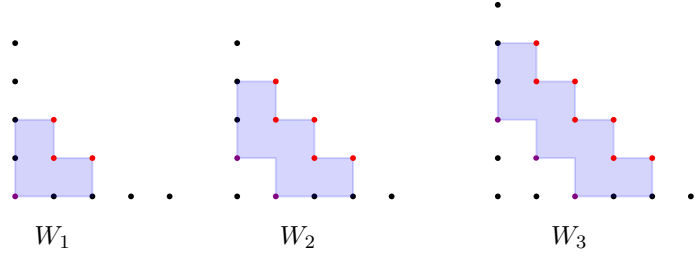}
	\caption{\label{figure_W_k_shaped_block} $W$-shaped blocks}
	\end{figure}	
\vspace{-0.5cm}

For a given face $\gamma$ of $\Gamma_\lambda$, the set of edges of $\gamma$ divides $W_k$
into several pieces of simple closed regions. See Figure \ref{figure_W_k_shaped_block} for example. 
\begin{definition}\label{defn:W-gamma} For the $W_k$-block and a face $\gamma$ of $\Gamma_\lambda$, we denote by $W_k(\gamma)$ the $W_k$-block with `walls', 
where a wall is an edge of $\gamma$ lying on the interior of the $W_k$-block.
\end{definition}

\begin{example}\label{example_dividing_block_123}
	In Example \ref{example_123}. 
	consider the vertex $v_3$ of $\Gamma_\lambda$ given in Figure~\ref{figure_zero_dim_face_F123}. 
	There is no edge of $v_3$ inside $W_1$ and hence we get
	$W_1(v_3) = W_1$ with no walls. The $W_2$-block is divided by $v_3$ into three pieces of simple
	closed regions so that $W_2(v_3)$ is $W_2$ with two walls indicated with red line segments in Figure~\ref{figure_W_k_shaped_block_v3}.
\vspace{0.2cm}
	\begin{figure}[H]
	\scalebox{0.9}{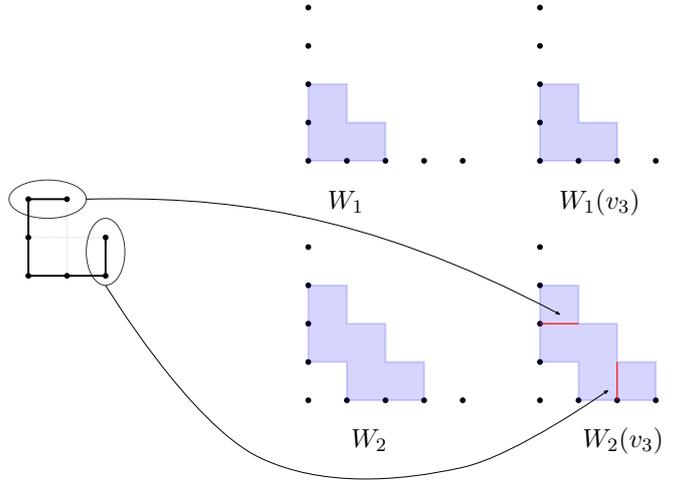}
	\caption{\label{figure_W_k_shaped_block_v3}  $W_i(v_3)$-blocks}
	\end{figure}
\vspace{-0.4cm}			
\end{example}

Next, we introduce the notion of \emph{$M$-shaped blocks}.

\begin{definition}\label{definition_M_shaped_block}
	For each positive integer $k \geq 1$, a {\em $k$-th $M$-shaped block} denoted by $M_k$, or simply an {\em $M_k$-block}, is defined, \emph{up to translation in} $\R^2$,
	by
	\[
		M_k := \bigcup_{(a,b)} \square^{(a,b)}
	\]
	where the union is taken over all $(a,b)$'s in $(\Z_{\geq 1})^2$ such that
	\begin{itemize}
		\item $k+1 \leq a+b \leq k+2$,
		\item $(a,b) \neq (k+1,1)$, and
		\item $(a,b) \neq (1,k+1)$.
	\end{itemize}
\begin{figure}[H]
	\scalebox{0.95}{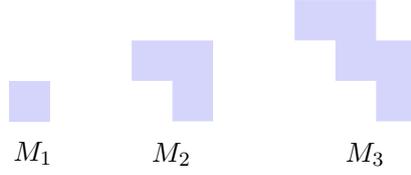}
		\vspace{-0.3cm}			
	\caption{\label{figure_block_sphere} $M$-shaped blocks.}
\end{figure}	
\vspace{-0.5cm}			
\end{definition}

\begin{remark}
	Note that $M_k$ can be obtained from $W_k$ by deleting two boxes $\square^{(k+1,1)}$ and $\square^{(1,k+1)}$.
	The reader should keep in mind that each $W$-shaped block $W_k$ is located at the specific position, but $M_k$ is not
	since it is defined up to translation in $\R^2$.
\end{remark}

For each divided simple closed region $\mcal{D}$ in $W_k(\gamma)$, we assign a topological space to $S_k (\mcal{D})$   by the following rule :
\begin{equation*}
S_k(\mathcal{D}) =
\begin{cases}
S^{2\ell-1} \quad \mbox{if $\mathcal{D}$ is the $M_\ell$-block and $\mcal{D}$ contains at least one bottom vertex of the $W_k$-block,} \\
 \mathrm{point} \quad \mbox{otherwise.}
\end{cases}
\end{equation*}
We then put
\begin{equation}\label{equation_block_sphere}
	S_k(\gamma) := \prod_{\mathcal{D} \subset W_k(\gamma)} S_k(\mathcal{D})
\end{equation}
where the product is taken over all simple closed regions in $W_k(\gamma)$ distinguished by walls coming from edges of $\gamma$.

\begin{example}\label{example_computation_S_2_v3}
	Again, we consider the vertex $v_3$ of $\Gamma_\lambda$ in Example \ref{example_123}, see also Example \ref{example_dividing_block_123}.
	Note that $S_1(v_3) = \mathrm{pt}$ since $W_1(v_3)$ consists of one simple closed region $W_1$ which is not an $M$-shaped block.
	In contrast, $W_2(v_3)$ consists of three simple closed regions $\mathcal{D}_1$, $\mathcal{D}_2$, and $\mathcal{D}_3$
	as in the figure below. (See also Figure \ref{figure_W_k_shaped_block_v3}.)
	Even though $\mcal{D}_1$ and $\mcal{D}_3$ are $M_1$-blocks, they do not contain any bottom vertices of $W_2$ so that 
	$S_2(\mcal{D}_1) = \mathrm{point}$ and $S_2(\mcal{D}_3) = \mathrm{point}$.
	On the other hand, $\mathcal{D}_2$ is an $M_2$-block containing two bottom vertices of $W_2$. 	
	Therefore, we have
	\[
		S_2(v_3) = S_2(\mathcal{D}_1) \times S_2(\mathcal{D}_2) \times S_2(\mathcal{D}_3) \cong S^3.
	\]
	For $k>2$, $W_k(v_3)$ has no walls and hence $W_k(v_3)$ consists of one simple closed region $W_k$ which is never an
	$M$-shaped block. Thus $S_k(v_3) = \mathrm{point}$ for every positive integer $k>2$.

	\begin{figure}[H]
		\scalebox{0.7}{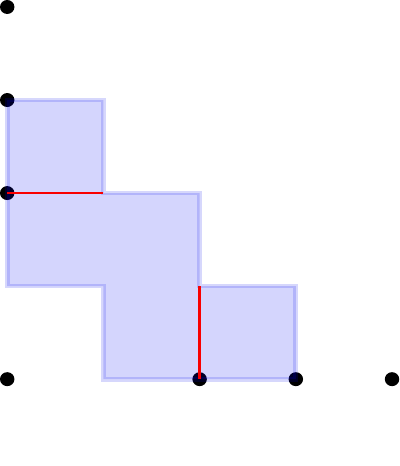}	
	\end{figure}
\end{example}

\begin{example}\label{example_gr24_W_block}
	Let $\lambda = \{t,t,0,0\}$ with any fixed $t > 0$. Then, the co-adjoint orbit $\mathcal{O}_\lambda$ is the complex Grassmannian $\mathrm{Gr}(2,4)$.
	Let $\gamma$ be the one-dimensional face of $\Gamma_\lambda$ given by
	\begin{figure}[H]
	\scalebox{0.9}{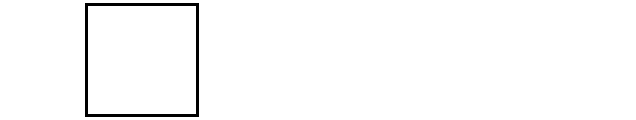}
	\end{figure}			
	As we see in the following figure, $W_1(\gamma)$ consists of one simple closed region that is not an
	$M$-shaped block. Thus we have $S_1(\gamma) = \mathrm{point}$.
	
	\begin{figure}[H]
		\scalebox{0.9}{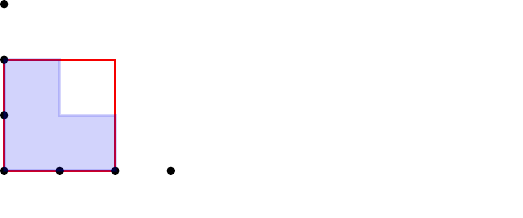}
	\end{figure}
	\vspace{-0.1cm}
	
	On the other hand, $W_2(\gamma)$ has two walls (red line segments in the figure below) and is exactly the same as $W_2(v_3)$ in Example 		
	\ref{example_computation_S_2_v3}. So, we have $S_2(\gamma) = S^3$.
	
	\begin{figure}[H]
		\scalebox{0.9}{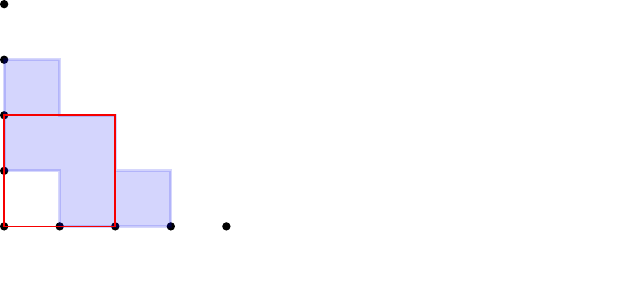}
	\end{figure}
	\vspace{-0.1cm}
	
	Finally, $W_3(\gamma)$ has two walls as below, and therefore there are three simple closed regions, 
	$\mathcal{D}_1$, $\mathcal{D}_2$, and $\mathcal{D}_3$, in $W_3(\gamma)$.
	
	\begin{figure}[H]
		\scalebox{0.9}{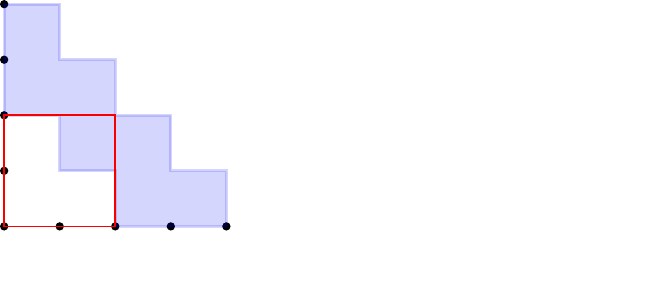}
	\end{figure}
	\vspace{-0.1cm}
	
	Since $\mathcal{D}_1$ and $\mathcal{D}_3$ are not $M$-shaped blocks, we have $S_3(\mathcal{D}_1) = S_3(\mathcal{D}_3) = \mathrm{point}$.
	Also, since $\mathcal{D}_2 = M_1$ and contains a botton vertex of $W_3$, we have $S_3(\mathcal{D}_2) = S^1$.
	Therefore, we obtain
	\[
		S_3(\gamma) = S_3(\mathcal{D}_1) \times S_3(\mathcal{D}_2) \times S_3(\mathcal{D}_3) \cong S^1.
	\]
	For $k>3$, $W_k(\gamma)$ consists only one simple closed region, the $W_k$-block itself, and it is never an $M$-shaped block.
	Thus we have $S_k(\gamma) = \mathrm{point}$ for $k>3$.

\end{example}

\begin{proposition}\label{proposition_M_1_block_simple_closed_region}
	Let $\lambda$ be a sequence of real numbers satisfying~\eqref{lambdaidef} and $\gamma$ be a face of the ladder diagram $\Gamma_\lambda$.
	For each $i \geq 1$, let $\{\mathcal{D}_1, \cdots, \mathcal{D}_{m_i} \}$ be simple closed regions in $W_i(\gamma)$ such that
	$
		S_i(\mathcal{D}_j) = S^1
	$ for every $j=1,\cdots, m_i$.
	Then we have
	\[
		\dim \gamma = \sum_{i=1}^{n-1} m_i.
	\]
\end{proposition}

\begin{proof}
	Note that $\dim \gamma$ is the number of minimal cycles in $\gamma$ by Definition \ref{definition_face}. Also, each minimal cycle $\sigma$ in $\gamma$
	can be represented by the union of two shortest paths connecting the bottom-left vertex and the top-right vertex of $\sigma$. We denote by
	$v_\sigma$ the top-right vertex of $\sigma$. 
	Note that $\square^{v_\sigma}$
	(blue-colored region in Figure \ref{figure_M_1_block_simple_closed_region}) is contained in the simple closed region bounded by $\sigma$.
	Therefore, if we denote by $\Sigma := \{\sigma_1, \cdots, \sigma_m\}$ the set of minimal cycles in $\gamma$, then
	there is a one-to-one correspondence
	between $\Sigma$ and $\{v_{\sigma_i} \}_{1 \leq i \leq m}$.
	
	On the other hand, observe that $\square^{v_\sigma}$ is appeared as an $M_1$-block in $W_i(\gamma)$ where
	\[
		i+1 = a+b, \quad v_\sigma = (a,b)
	\]	
	for each $\sigma \in \Sigma$.
	Also, every $M_1$-block in $W_i(\gamma)$ for some $i$ having a bottom vertex should be one of such $\square^{v_\sigma}$'s.
	Consequently, there is a one-to-one correspondence between $\Sigma$ and $M_1$-blocks having bottom vertices in $W_i(\gamma)$ for $i \geq 1$.
	Since $|\Sigma| = \dim \gamma$ by definition, this completes the proof.
\end{proof}
	\vspace{-0.5cm}
	\begin{figure}[H]
		\scalebox{0.9}{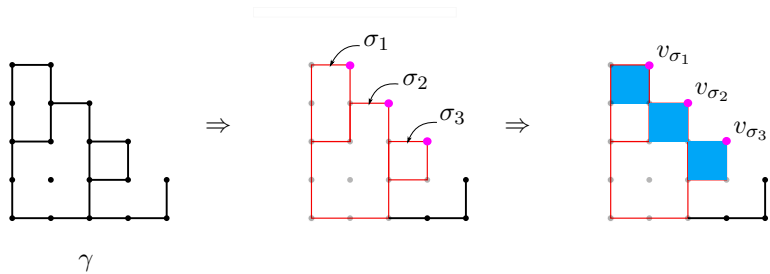}
		\caption{\label{figure_M_1_block_simple_closed_region} Correspondence between minimal cycles and $M_1$-blocks}
	\end{figure}
	\vspace{-0.3cm}

Now, we state one of our main theorem which characterizes the topology of each fiber of $\Phi_\lambda$, where the proof will be given in Section \ref{section4}.

\begin{theorem}\label{theorem_main}
	Let $\lambda = \{ \lambda_1, \cdots, \lambda_n \}$ be a non-increasing sequence of real numbers satisfying~\eqref{lambdaidef}.
	Let $\gamma$ be a face of $\Gamma_\lambda$ and $f_\gamma = \Psi(\gamma)$ be the corresponding face of $\Delta_\lambda$
	described in Theorem \ref{theorem_equiv_CG_LD}.
	For any point $\textbf{\textup{u}}$ in the relative interior of $f_\gamma$, the fiber $\Phi_\lambda^{-1}(\textbf{\textup{u}})$ is an isotropic submanifold
	of $(\mathcal{O}_\lambda, \omega_\lambda)$ and is the total space of an iterated bundle
	\begin{equation}\label{theoremmaindia}
		\Phi_\lambda^{-1}(\textbf{\textup{u}}) \cong \bar{S_{n-1}}(\gamma) \xrightarrow {p_{n-1}} \bar{S_{n-2}}(\gamma)
		 \rightarrow \cdots
		 \xrightarrow{p_2} \bar{S_1}(\gamma) \xrightarrow{p_1} \bar{S_0}(\gamma) := \mathrm{point}
	\end{equation}
	where
	$p_k \colon \bar{S_k}(\gamma) \rightarrow \bar{S_{k-1}}(\gamma)$ is an $S_k(\gamma)$-bundle over $\bar{S_{k-1}}(\gamma)$ for $k=1,\cdots, n-1$.
	In particular, the dimension of $\Phi_\lambda^{-1}(\textbf{\textup{u}})$ is given by 
	\[
		\dim \Phi_\lambda^{-1}(\textbf{\textup{u}}) = \sum_{k=1}^{n-1} \dim S_k(\gamma).
	\]
\end{theorem}

\begin{remark}
As a consequence of the description, every Gelfand-Cetlin fiber is a smooth submanifold. In \cite{Lane1}, We remark that Lane proved the smoothness of fibers of completely integrable systems constructed by Thimm's trick. 
\end{remark}

We illustrate Theorem \ref{theorem_main} by the following examples.

\begin{example}\label{example_lag_fiber_123_vertex}
	Let $\lambda$ be given in Example \ref{example_123}. 
	We apply Theorem \ref{theorem_main} to compute the fiber $\Phi_\lambda^{-1}(v_i)$ of each vertex $v_i$ given in Figure~\ref{figure_zero_dim_face_F123}
	for $i=1,\cdots,7$. 
	\begin{figure}[H]
		\scalebox{0.8}{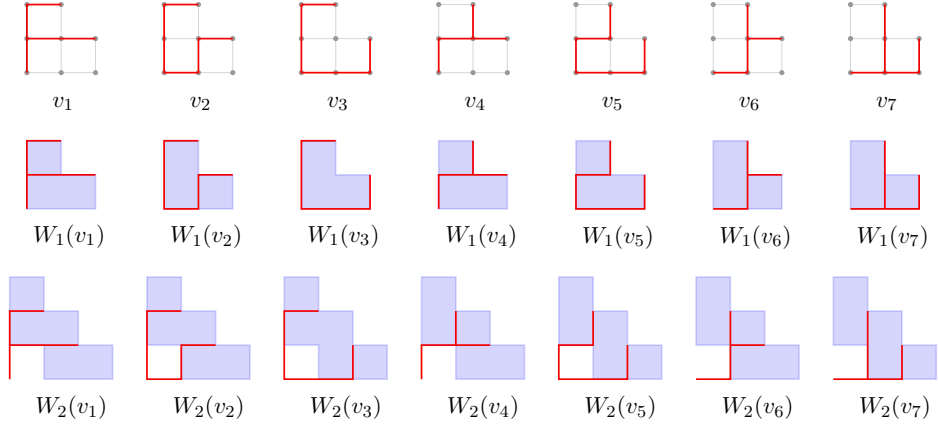}
		\caption{\label{figure_inverse_vertex_123} $W_1(v_i)$'s and $W_2(v_i)$'s for $\F (3)$}
	\end{figure}
		\vspace{-0.5cm}

	Figure \ref{figure_inverse_vertex_123} shows that $S_1(v_i) = \mathrm{pt}$ for every $i=1,\cdots,7$ since each $W_1(v_i)$ does not contain
	any $M$-shaped block containing a bottom vertex of $W_1$, that is, the origin $(0,0)$. Also, we can easily check that $S_2(v_i) = \mathrm{point}$ unless
	$i = 3$. When $i=3$, there is one $M$-shaped block $M_2$ inside $W_2(v_3)$ containing a bottom vertex of $W_2$. Thus we have
	\[
		S_2(v_3) = \mathrm{point} \times S^3 \times \mathrm{point} \cong S^3. 
	\] 
	Since $\Phi_\lambda^{-1}(v_i)$ is an $S_2(v_i)$-bundle over $S_1(v_i)$
	and $S_1(v_i)$ is a point for every $i=1,\cdots, 7$, by Theorem \ref{theorem_main}, we obtain
	\[
		\Phi_\lambda^{-1}(v_i) =
		\begin{cases}
			S^3 \quad &\mbox{ if } i = 3 \\
			\mathrm{point} \quad &\mbox{ otherwise.}
		\end{cases}
	\]
\end{example}

\begin{example}\label{example_lag_fiber_123_edge}	
	Again, we consider Example \ref{example_123} and compute the fibers over the points in the relative interior of some higher dimensional face of $\Delta_\lambda$ as follows.
	
	Consider the edge $e = e_{12}$ in Figure~\ref{figure_one_dim_face_F123} and let $\textbf{\textup{u}} \in \mathring{e}_{12}$. 
	By Theorem \ref{theorem_main}, $\Phi_\lambda^{-1}(\textbf{\textup{u}})$ is an $S_2(e_{12})$-bundle 
	over $S_1(e_{12})$ so that
	it is diffeomorphic to $S^1$. See the figure below. 
	For any other edge $e$ of $\Gamma_\lambda$,
	we can show that $\Phi_\lambda^{-1}(\textbf{\textup{u}}) \cong S^1$ for every point $\textbf{\textup{u}} \in \mathring{e}$ in a similar way. 
	\vspace{-0.2cm}
	\begin{figure}[H]
	\scalebox{0.9}{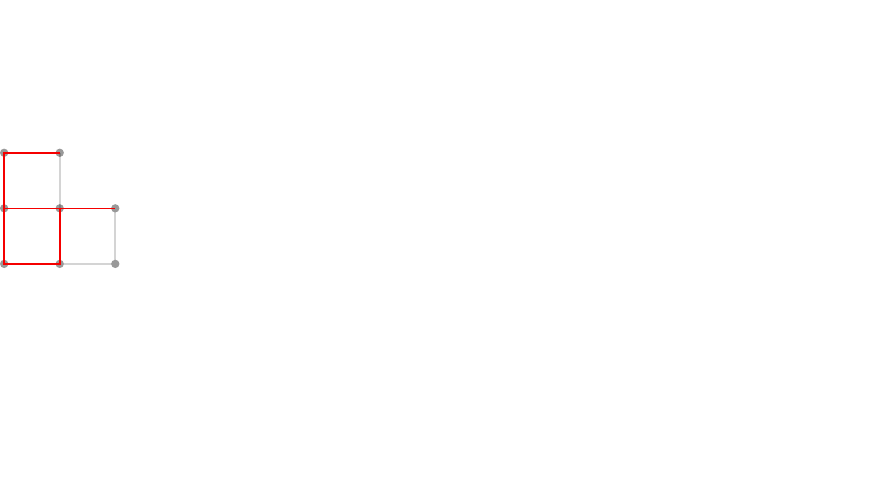}
	\end{figure}
	\vspace{-0.1cm}	
	
	For a two-dimensional face of $\Gamma_\lambda$, we first consider the face $f_{1345}$ of $\Delta_\lambda$ described in Figure
	\ref{figure_two_dim_face_F123}. 	
	Again by Theorem \ref{theorem_main}, 
	we have $\Phi_\lambda^{-1}(\textbf{\textup{u}}) \cong T^2$, an $S_2(f_{1345})$-bundle over $S_1(f_{1345})$, for every point $\textbf{\textup{u}} \in \mathring{f}_{1345}$. 
	See the figure below.
	Similarly, we obtain $\Phi_\lambda^{-1}(\textbf{\textup{u}}) \cong T^2$ for every interior point $\textbf{\textup{u}}$ of any two-dimensional face of $\Delta_\lambda$.

	\vspace{-0.2cm}
	\begin{figure}[H]
	\scalebox{0.9}{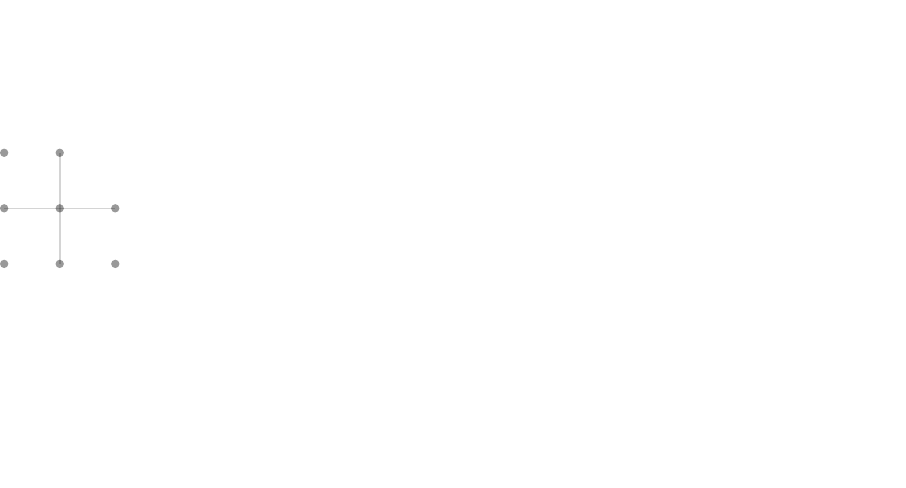}
	\end{figure}	
	\vspace{-0.1cm}	
	Finally, consider $I_{1234567}$, the improper face of $\Delta_\lambda$.
	Then Theorem \ref{theorem_main} tells us that 
	$\Phi_\lambda^{-1}(\textbf{\textup{u}})$ is an $S^1 \times S^1$-bundle over $S^1$ for every interior point $\textbf{\textup{u}}$ of $\Gamma_\lambda$.
	In fact, the Liouville-Arnold theorem implies that the bundle is trivial, i.e., $\Phi_\lambda^{-1}(\textbf{\textup{u}})$ is a torus $T^3$ for every 
	$\textbf{\textup{u}} \in \mathring{\Delta_\lambda}$, see also Theorem~\ref{theorem_contraction}.
	\vspace{-0.2cm}
	\begin{figure}[H]
	\scalebox{0.9}{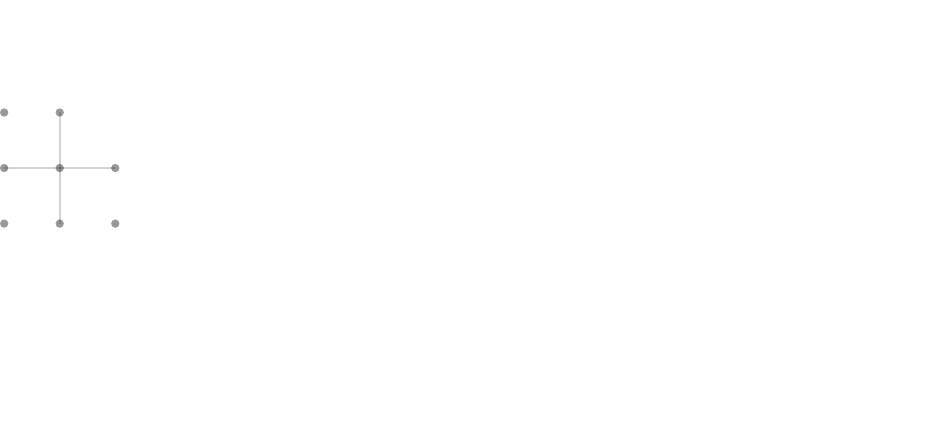}
	\end{figure}	
	\vspace{-0.2cm}
		
	Consequently, a Lagrangian fiber of the GC system on $(\mathcal{O}_\lambda, \omega_\lambda)$ is diffeomorphic to either $T^3$ (a fiber over an
	interior point of $\Delta_\lambda$) or $S^3$ (a fiber over $v_3$). Other fibers are isotropic but not Lagrangian submanifolds of
	$(\mathcal{O}_\lambda, \omega_\lambda)$ for dimensional reasons.
\end{example}

\begin{remark}\label{remark_su(3)}
In general, one should \emph{not} expect that every iterated bundle in~\eqref{theoremmaindia} is trivial. Namely, $\Phi^{-1}_\lambda(\textbf{\textup{u}})$ might not be the product 
space $\prod_{k=1}^{n-1} S_k (\gamma)$. For instance, consider the co-adjoint orbit $\mcal{O}_\lambda \simeq \mcal{F}(2,3;5)$ associated with 
$\lambda = (3,3,0,-3,-3)$ as in Figure~\ref{Fig_SU33}. By Theorem~\ref{theorem_main}, the GC fiber $\Phi_\lambda^{-1}(\textbf{0})$ over the origin is an $S^3$-bundle over $S^5$. 
Meanwhile, Proposition 2.7 in \cite{NU2} implies that $\Phi_\lambda^{-1}(\textbf{0})$ is $SU(3)$. It is however that $SU(3)$ and $S^5 \times S^3$ are not homotopy equivalent.
	\begin{figure}[H]
		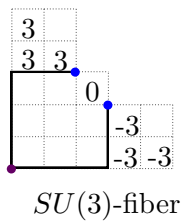
		\caption{\label{Fig_SU33} $SU(3)$-fiber}
	\end{figure}
\end{remark}

\subsection{Classification of Lagrangian faces}~
\label{ssecClassificationOfLagrangianFaces}

Recall that Theorem \ref{theorem_main} implies that a face $\gamma$ is Lagrangian if and only if the GC fiber over an interior point of $\gamma$ is of dimension  
$\frac{1}{2} \dim_{\R} \mathcal{O}_\lambda.$
Therefore to determine whether $\gamma$ is Lagrangian or not, it suffices to check that 
 \[
	\sum_{k=1}^{n-1} \dim S_k(\gamma) = \frac{1}{2} \dim_{\R} \mathcal{O}_\lambda.
 \]
In this section, we present an efficient way of checking whether a given face of $\Gamma_\lambda$ is Lagrangian or not by using so called ``L-shaped blocks''. 

\begin{definition}\label{definition_L_shaped_block}
	For each positive integer $k \in \Z_{\geq 1}$ and every lattice point $(p,q) \in \Z^2 \subset \R^2$, {\em the $k$-th L-shaped block at $(p,q)$}, 
	or simply {\em the $L_k$-block at $(p,q)$}, is the closed region defined by 
	\[
		L_k(p,q) := \bigcup_{(a,b)} \square^{(a,b)}
	\]
	where the union is taken over all $(a,b)$'s in $\Z^2$ such that
	\begin{itemize}
		\item $(a,b) = (p,q+i)$ \quad for $0 \leq i \leq k-1$, and
		\item $(a,b) = (p+i,q)$ \quad for $0 \leq i \leq k-1$.
	\end{itemize}
\end{definition}
	\vspace{-0.5cm}
	\begin{figure}[H]
	\scalebox{0.9}{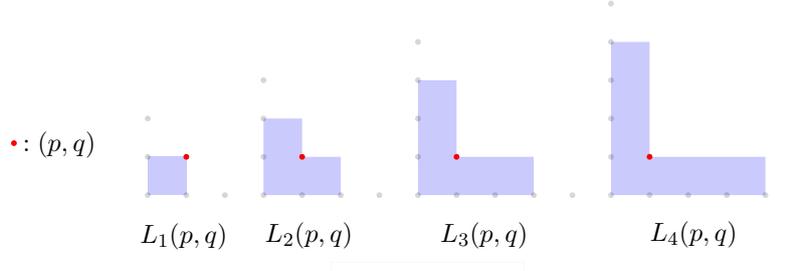}
	\caption{\label{figure_L_shape} $L$-shaped blocks}
	\end{figure}
	\vspace{-0.5cm}
	
\begin{remark}\label{rem:linearly-orderd}
    Note that GC patterns described in \eqref{equation_GC-pattern} are
    linearly ordered on any of $W$-shaped, $M$-shaped, and $L$-shaped
    blocks in the direction from the right or bottom most block to the left or top most
    block.
\end{remark}

\begin{definition}\label{definition_rigid}
	Let $\gamma$ be a face of $\Gamma_\lambda$.
	For a given positive integer $k \in \Z_{\geq 1}$ and a lattice point $(p,q) \in \Z^2$, we say that $L_k(p,q)$ is {\em rigid in} $\gamma$ if
	\begin{enumerate}
		\item the interior of $L_k(p,q)$ does not contain an edge of $\gamma$ and
		\item the rightmost edge and the top edge of $L_k(p,q)$ should be edges of $\gamma$.
	\end{enumerate}
\end{definition}

\begin{example}
	Let us consider $\Gamma = \Gamma(2,5;7)$ and consider a face $\gamma$ given as follows.
	\begin{figure}[H]
	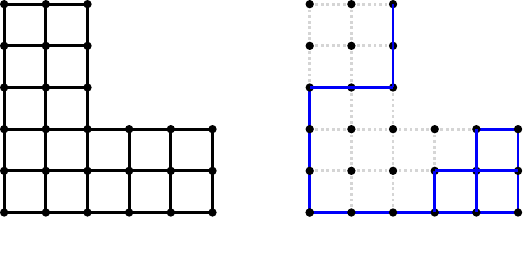
	\end{figure}
	In this example, there are exactly four rigid $L_k$-blocks : $L_3(1,1)$, $L_1(4,1)$, $L_1(5,1)$, and $L_1(5,2)$.
	\begin{figure}[H]
	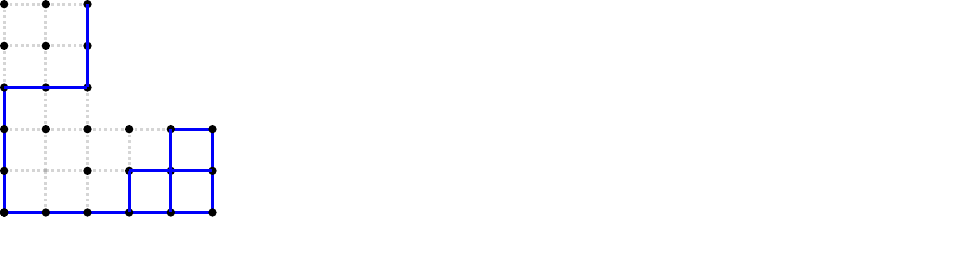
	\end{figure}
	We can check that any other $L$-blocks are \emph{not} rigid.
	For instance, $L_3(2,1)$ is not rigid since its interior contains an edge of $\gamma$. 
	\begin{figure}[H]
	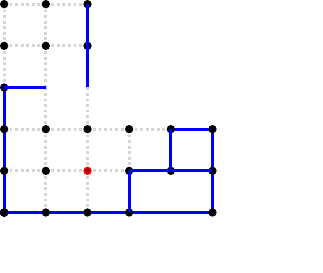
	\end{figure}
	Also, $L_2(2,2)$ is not rigid because its rightmost edge is \emph{not} an edge of $\gamma$ so that it violates the condition (2) in Definition
	\ref{definition_rigid}.
	\begin{figure}[H]
		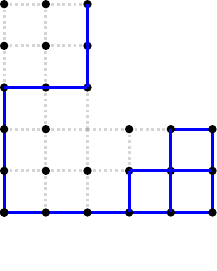
	\end{figure}
\end{example}

The following lemma follows from the min-max principle of GC pattern \eqref{equation_GC-pattern}, 
or more specifically from~\eqref{rem:diagram-pattern}.

\begin{lemma}\label{lemma_L_Q}
	Suppose that $L_k(a,b)$ is rigid in a face $\gamma$ of $\Gamma_\lambda$.
	Let $Q_k(a,b)$ be the closed region defined by
	\[
		Q_k(a,b) := \bigcup_{0 \leq i,j \leq k-1} \square^{(a+i,b+j)},
	\]
	i.e., $Q_k(a,b)$ is the square of size $(k \times k)$ that contains $L_k(a,b)$.
	Then there are no edges of $\gamma$ in the interior of $Q_k(a,b)$.
\end{lemma}

\begin{proof}
	If $k=1$, then $L_1(a,b) = Q_1(a,b)$ and it has no edge of $\gamma$ in its interior so that there is nothing to prove. Now, assume that $k \geq 2$
	and suppose that there is an edge $e = [v_0v_1]$ of $\gamma$ contained in the interior of $Q_k(a,b)$.
	Then, without loss of generality, we may assume that
	$v_0 = (a_0, b_0)$ is in the interior of $Q_k(a,b)$ so that $a \leq a_0 < a+k-1$ and $b \leq b_0 < b+k-1$. 
	
	By Definition \ref{definition_face}, there exists a positive path $\delta$ contained in $\gamma$ passing through $v_0$. 
	Then $\delta$ should pass through the interior of $L_k(a,b)$ since $\delta$ is a shortest path from the origin of $\Gamma_\lambda$ and $v_0$, 
	which intersects the interior of $L_k(a,b)$. This contradicts the rigidity (1) in Definition \ref{definition_rigid}.
\end{proof}

\begin{lemma}\label{lemma_L_disjoint}
	If two different $L$-blocks $L_i(a,b)$ and $L_j(c,d)$ are rigid in the same face $\gamma$, 
	then they cannot be overlapped, i.e.,
	\[
		\mathring{L_i}(a,b) \cap \mathring{L_j}(c,d) = \emptyset
	\]
	where $\mathring{L_i}(a,b)$ and $\mathring{L_j}(c,d)$ denote the interior of $\mathring{L_i}(a,b)$ and $\mathring{L_j}(c,d)$, respectively.
\end{lemma}

\begin{proof}
	If $(a,b) = (c,d)$ and $i \neq j$, then both $L_i(a,b)$ and $L_i(c,d)$ cannot be rigid since at least one of these two blocks violates (1) in Definition~\ref{definition_rigid}. 
	Suppose that $(a,b) \neq (c,d)$ and $\mathring{L_i}(a,b) \cap \mathring{L_j}(c,d) \neq \emptyset$. Without loss of generality, we may assume that $j \leq i$.
	Then it is straightforward that either the top edge or the rightmost edge of $L_i(c,d)$ lies on the interior of
	$Q_j(a,b)$. It leads to a contradiction to Lemma \ref{lemma_L_Q} since
	the top edge and the rightmost edge of $L_i(c,d)$ are edges of $\gamma$ by Definition \ref{definition_L_shaped_block}.
	\begin{figure}[H]
		\scalebox{0.9}{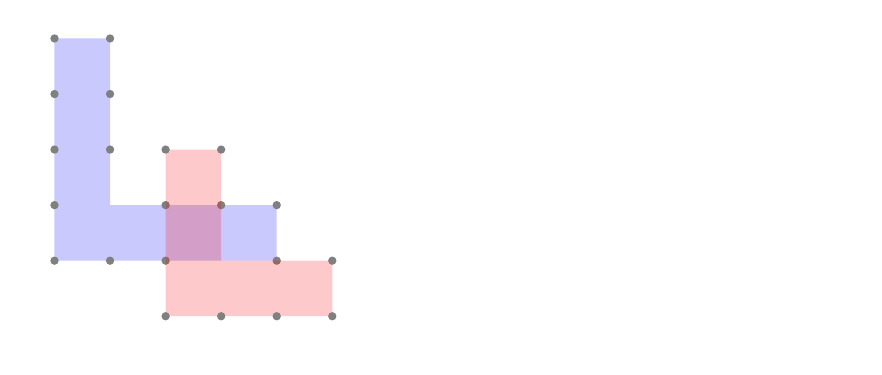}
	\end{figure}
	\vspace{-0.2cm}
\end{proof}

\begin{proposition}\label{proposition_classification_lagrangian_face}
	Let $\lambda = (\lambda_1, \cdots, \lambda_n)$ be given in \eqref{lambdaidef}. 
	For a given face $\gamma$ of $\Gamma_\lambda$, let $\mathcal{L}(\gamma)$ be the set of all rigid L-shaped blocks in $\gamma$.
	Then, for any point $\textbf{\textup{u}}$ in the relative interior of the face $f_\gamma = \Psi(\gamma)$ of $\Delta_\lambda$, we have 
	\[
		\dim \Phi_\lambda^{-1}(\textbf{\textup{u}}) = \sum_{L_k(a,b) \in \mathcal{L}(\gamma)} |L_k(a,b)| = \sum_{L_k(a,b) \in \mathcal{L}(\gamma)} (2k-1)
	\]	
	where $|L_k(a,b)| = 2k - 1$ is the are of $L_k(a,b)$.
\end{proposition}

\begin{proof}
	By Theorem \ref{theorem_main}, it is enough to show that
	\[
		\sum_{k=1}^{n-1} \dim S_k(\gamma) =  \sum_{L_k(a,b) \in \mathcal{L}(\gamma)} |L_k(a,b)|
	\] where 
	\[
		S_k(\gamma) = \prod_{\mathcal{D} \subset W_k(\gamma)} S_k(\mathcal{D})
	\]
	as defined in (\ref{equation_block_sphere}).
	
	Let $\mathcal{D}$ be a simple closed region in $W_k(\gamma)$. Suppose that $\mathcal{D}$ contains a bottom vertex of $W_k$ and
	$\mathcal{D} = M_j(a,b)$ for some $j \geq 1$ where $M_j(a,b)$ denotes the $j$-th $M$-shaped block whose top-left vertex is
	$(a,b)$.
	Then there are two edges $e_1$ and $e_2$, top-left-horizontal one and bottom-right-vertical one, of $\gamma$ on the boundary of $M_j$ as we see below.
	\begin{figure}[h]
	\scalebox{0.9}{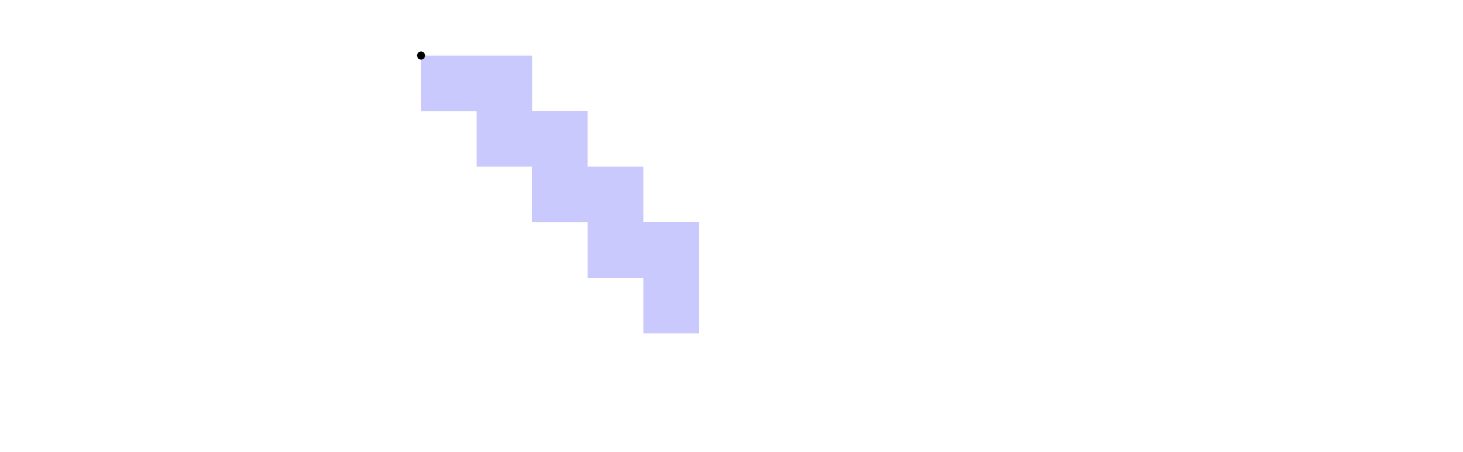}
	\end{figure}
	By the min-max principle, we see that 
	there is no edge of $\gamma$ in the interior of $Q_j(a+1,b-j+1)$. Thus $L_j(a+1,b-j+1)$
	is a rigid $L_j$-block in $\gamma$.
	
	Similarly, for any rigid $L$-shaped block $L_j(a+1,b-j+1)$ in $\gamma$ (for some $j \in \Z_{\geq 1}$ and $(a,b) \in (\Z_{\geq 1})^2$), 
	we can find a closed region $\mcal{D}$ in some $W_k(\gamma)$ such that 
	$\mcal{D}$ is $M_j(a,b)$ containing some bottom vertex of $W_k$. 
	Therefore, 
	there is a one-to-one correspondence between the set of rigid $L$-shaped blocks in $\gamma$ and the set of simple closed regions $\mathcal{D}$ appeared
	in some $W_k(\gamma)$ and containing a bottom vertex of $W_k(\gamma)$. To sum up, we obtain
	\[
		\begin{array}{rcl}
			\displaystyle \bigcup_{j=1} \,\, \{\textrm{rigid~$L_j$}$-$\mathrm{blocks~in~}\gamma \}  & \stackrel{1:1} \Leftrightarrow & \displaystyle \bigcup_{j=1} \bigcup_{k=1}^{n-1} 
			\,\, \left\{ \mathcal{D} \subset W_k(\gamma) ~|~ S_k(\mathcal{D}) = S^{2j-1} \right\} \\ [0.5em]
			L_j(a+1,b-j+1) & \leftrightarrow & \mcal{D} = M_j(a,b).
		\end{array}
	\]
	Moreover, we have $|M_j| = |L_j| = \dim S_k(\mathcal{D}) = 2j-1$. This completes the proof.
\end{proof}

The following corollary is derived from Theorem \ref{theorem_main},
Lemma \ref{lemma_L_disjoint} and Proposition \ref{proposition_classification_lagrangian_face}.

\begin{corollary}\label{corollary_L_fillable}
Let $\gamma$ and $f_\gamma$ be as in Proposition \ref{proposition_classification_lagrangian_face}.
Then the followings are equivalent.
	\begin{enumerate}
		\item For an interior point $\textbf{\textup{u}}$ of $f_\gamma$, the fiber $\Phi_\lambda^{-1}(\textbf{\textup{u}})$ is a Lagrangian submanifold of $(\mathcal{O}_\lambda, \omega)$.
		\item For any interior point $\textbf{\textup{u}}$ of $f_\gamma$, the fiber $\Phi_\lambda^{-1}(\textbf{\textup{u}})$ is a Lagrangian submanifold of $(\mathcal{O}_\lambda, \omega)$.
		\item The set of rigid L-shaped blocks in $\gamma$ covers the whole $\Gamma_\lambda$.
	\end{enumerate}
\end{corollary}

Also, we have the following corollary which follows from Corollary \ref{corollary_L_fillable}.
\begin{corollary}\label{corollary_unless_projective}
	A Gelfand-Cetlin system $\Phi_\lambda \colon \mathcal{O}_\lambda \rightarrow \Delta_\lambda$ on a co-adjoint orbit $(\mathcal{O}_\lambda, \omega_\lambda)$
	always possesses a non-torus Lagrangian fiber unless $\mathcal{O}_\lambda$ is a projective space.
\end{corollary}

\begin{example}\label{example_gr26}
	Let $\lambda = \{t,t,0,0,0,0\}$ with $t >0$. The co-adjoint orbit $\mathcal{O}_\lambda$ is a complex Grassmannian $\mathrm{Gr}(2,6)$
	of two planes in $\C^6$ and the corresponding ladder diagram
	$\Gamma_\lambda$ is given as below.
	\begin{figure}[H]
	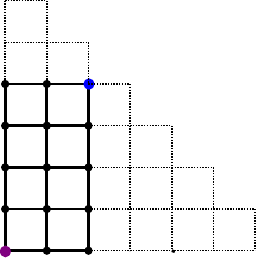
	\end{figure}
	Observe that any faces of $\Gamma_\lambda$ do not admit rigid $L_k$-blocks of $k > 2$. 
	Also, note that there are three Lagrangian faces $\gamma_1$, $\gamma_2$ and $\gamma_3$ of $\Gamma_\lambda$
	which have only one rigid $L_2$-block as follows.
	\begin{figure}[H]
	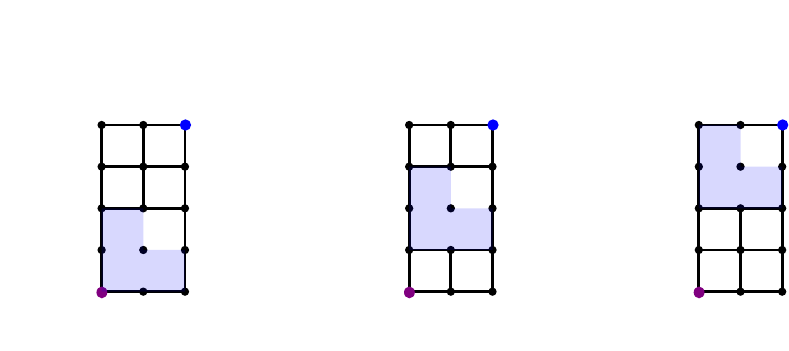
	\end{figure}
	Finally, there is exactly one Lagrangian face $\gamma_4$ which has two rigid $L_2$-blocks as below.
	\begin{figure}[H]
	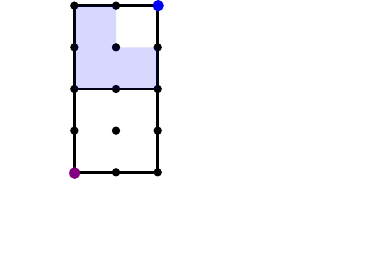
	\end{figure}
	Thus there are exactly four proper Lagrangian faces $\gamma_1, \gamma_2, \gamma_3$ and $\gamma_4$ of $\Gamma_\lambda$.
\end{example}

\begin{example}
	Let $\lambda = \{3,2,1,0\}$. Then the co-adjoint orbit $\mathcal{O}_\lambda$ is the full flag manifold $\mcal{F}(4)$ and the corresponding ladder diagram $\Gamma_\lambda$ is as follows.
	\begin{figure}[H]
	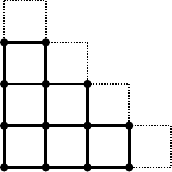
	\end{figure}
	We can easily see that any face of $\Gamma_\lambda$ does not have a rigid $L_k$-block for $k \geq 3$.
	On the other hand, there are exactly three Lagrangian faces of $\Gamma_\lambda$ containing one rigid $L_2$-block as below.
	 \begin{figure}[H]
	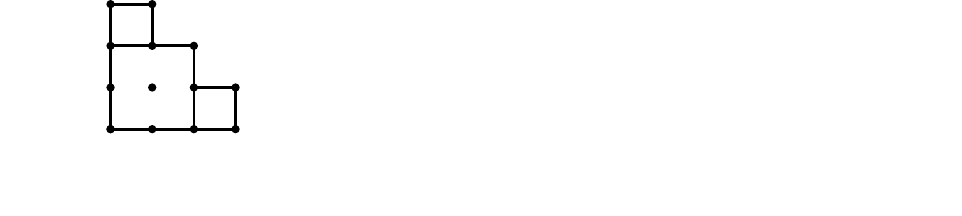
	\end{figure}
	Finally, it is not hard to see that there is no Lagrangian face that contains more than one $L_2$-block.
	Thus $\gamma_1$, $\gamma_2$, and $\gamma_3$ are the only proper Lagrangian faces of $\Gamma_\lambda$.
\end{example}

\vspace{0.1cm}
\section{Iterated bundle structures on Gelfand-Cetlin fibers}
\label{secIteratedBundleStructuresOnGelfandCetlinFibers}
\label{section4}

In this section, for each point $\textbf{\textup{u}}$ in a GC polytope $\Delta_\lambda$, we will construct the iterated bundle $E_\bullet$ described in 
Theorem \ref{theoremA} whose total space is the fiber $\Phi^{-1}_\lambda(\textbf{\textup{u}})$. Using this construction, we complete the proof of Theorem \ref{theorem_main} by showing that
each $\Phi^{-1}_\lambda(\textbf{\textup{u}})$ is an isotropic submanifold of $(\mathcal{O}_\lambda, \omega_\lambda)$.

	For each $\ell \geq 1$, we denote by $\mathcal{H}_\ell$ the set of $\ell \times \ell$ Hermitian matrices and by $\mathrm{sp}(x)$ the spectrum of $x$.
	For a fixed integer $k \geq 1$, consider sequences  $\fa = (a_1, \cdots, a_{k+1})$ and $\fb = (b_1, \cdots, b_k)$ of real numbers satisfying
	\begin{equation}\label{patternaandb}
		a_1 \geq b_1 \geq a_2 \geq b_2 \geq \cdots \geq a_k \geq b_k \geq a_{k+1}.
	\end{equation}
	Let
	\[
			\mathcal{O}_\fa =  \left\{ x \in \mathcal{H}_{k+1} ~\big{|}~  \mathrm{sp}(x) = \{a_1, \cdots, a_{k+1}\} \right\}
	\]
	be the co-adjoint $U(k+1)$-orbit of the diagonal matrix $I_\fa := \mathrm{diag}(a_1,\cdots, a_{k+1})$ in $\mathcal{H}_{k+1}  \cong \uu(k+1)^*$. Then,
	\begin{align*}
		\mathcal{A}_{k+1}(\fa,\fb) & = \left\{ x \in \mathcal{H}_{k+1} ~\big{|}~ \mathrm{sp}(x) = \{a_1, \cdots, a_{k+1}\}, ~\mathrm{sp}(x^{(k)}) = \{b_1, \cdots, b_k\} \right\}
	\end{align*}
	is a subset of $\mcal{O}_\fa$
	where $x^{(k)}$ denotes the $(k \times k)$ leading principal submatrix of $x$.
	It naturally comes with the projection map
	\[
		\begin{array}{ccccc}
		\rho_{k+1} \colon & \mathcal{A}_{k+1}(\fa,\fb) & \rightarrow & \mathcal{O}_\fb \\
		                     &                   x                     & \mapsto & x^{(k)}\\
		\end{array}
	\]
	from $\mathcal{A}_{k+1}(\fa,\fb)$ to the co-adjoint $U(k)$-orbit $\mathcal{O}_\fb$ of the diagonal matrix $I_\fb$ in $\mathcal{H}_k \cong \uu(k)^*$.

	Let $W_k(\fa,\fb)$ be the $k$-th $W$-shaped block $W_k$ together with walls defined by the equalities of $a_i$'s and $b_j$'s as in Figure~\ref{figure_block_a_b}.
	By comparing the divided regions by the walls on $W_k(\fa,\fb)$ with $M$-shaped blocks as in~\eqref{equation_block_sphere},
	we define a topological space $S_k(\fa,\fb)$, which is either a single point or a product space of odd dimensional spheres.
	
	\begin{figure}[H]
	\scalebox{0.9}{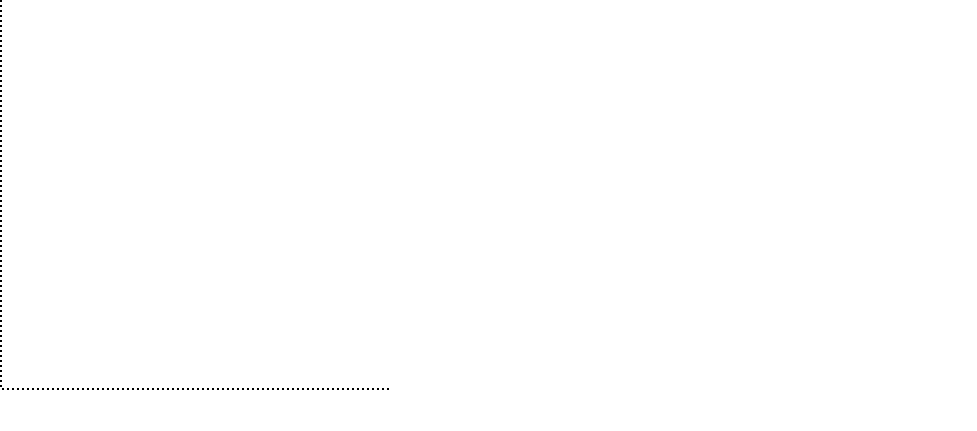}
	\caption{\label{figure_block_a_b} $W_k(\fa, \fb)$}
	\end{figure}	

	\begin{example}
	\begin{enumerate}
	\item
		For $\fa = (a_1,a_2,a_3) = (5,4,2)$ and $\fb = (b_1, b_2) = (4,2)$, $W_2(\fa,\fb)$ is divided by three simply closed regions $\mcal{D}_1, \mcal{D}_2$ and $\mcal{D}_3$. Since $\mcal{D}_1$ does not contain any bottom vertices and neither $\mcal{D}_2$ nor $\mcal{D}_3$ match with $M$-shaped blocks, $S_2(\fa,\fb) = S_2(\mathcal{D}_1) \times S_2(\mathcal{D}_2)  \times S_2(\mathcal{D}_3)  \cong \mathrm{point}$.
		 \begin{figure}[H]
				\scalebox{0.9}{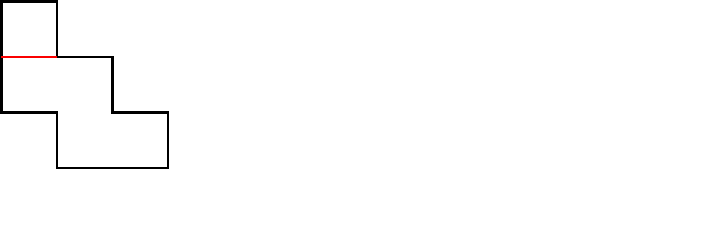}
		\end{figure}
		\vspace{-0.2cm}

		\item		For $\fa = (5,4,2)$ and $\fb = (4,4)$, $W_2(\fa,\fb)$ is divided by three simply closed regions $\mcal{D}_1, \mcal{D}_2$ and $\mcal{D}_3$. Observe that $\mcal{D}_1$ and $\mcal{D}_3$ do not contain any bottom vertices and $\mathcal{D}_2$ is an $M_2$-block containing bottom vertices of $W_2$. Therefore, we have $S_2(\fa,\fb) = \mathrm{point} \times S^3 \times \mathrm{point} \cong S^3.$

		\begin{figure}[H]
				\scalebox{0.9}{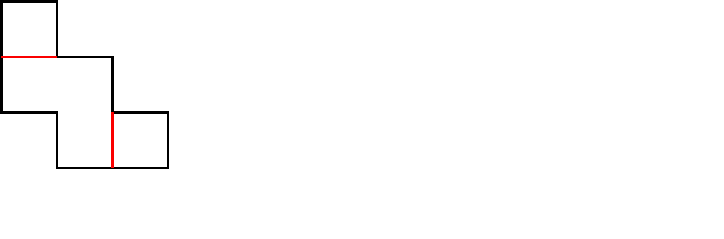}
		\end{figure}
		\vspace{-0.3cm}	
	\end{enumerate}
	\end{example}
	
	\begin{proposition}\label{proposition_A_k_S_k_bundle}\label{proposition_bundle}
		With the notations above, $\rho_{k+1} \colon \mathcal{A}_{k+1}(\fa,\fb) \rightarrow \mathcal{O}_\fb$ is an $S_k(\fa,\fb)$-bundle over $\mathcal{O}_\fb$.
	\end{proposition}
	
	Before starting the proof of Proposition \ref{proposition_A_k_S_k_bundle}, we introduce some notations and prove lemmas. We denote by $\widetilde{\mathcal{A}}_{k+1}(\fa,\fb)$ 
	the subset of $\mathcal{A}_{k+1}(\fa,\fb)$ whose $k \times k$ leading principal submatrix is the diagonal matrix $I_\fb$. So, a matrix in $\widetilde{\mathcal{A}}_{k+1}(\fa,\fb)$ 
	is of the form
	\[
		Z_{(\fa,\fb)}(z) =   \begin{pmatrix} b_1 & & 0 & \bar{z}_1\\
                	  & \ddots & & \vdots\\
                  	0 & & b_k & \bar{z}_k\\
                  	z_1 & \hdots & z_k & z_{k+1}
		  \end{pmatrix}
		  \vspace{0.2cm}
	\]
	for $z = (z_1, \cdots, z_k) \in \C^k$. Since $Z_{(\fa,\fb)}(z)$ has the eigenvalues $\fa = \{a_1, \cdots, a_{k+1} \}$, the $(k+1,k+1)$-entry of $Z_{(\fa,\fb)}(z)$ has to be constant
	\[
		z_{k+1} = \sum_{i=1}^{k+1} a_i - \sum_{i=1}^k b_i \in \R
	\]
	by computing the trace of $Z_{(\fa,\fb)}(z)$.
	The characteristic polynomial of $Z_{(\fa,\fb)}(z)$ is expressed as
	\begin{equation}\label{equation_det}
		\det(xI - Z_{(\fa,\fb)}(z)) = (x - z_{k+1}) \cdot \prod_{i=1}^k (x-b_i) - \sum_{j=1}^k \left( \frac{|z_j|^2}{x-b_j} \cdot \prod_{i=1}^k (x-b_i) \right) = 0,
	\end{equation}
	whose zeros are $x = a_1, \cdots, a_{k+1}$ by our assumption. By inserting $x = a_1, \cdots, a_{k+1}$ into~\eqref{equation_det}, we obtain the system of $(k+1)$ equations of real coefficients, which are linear with respect to $(|z_1|^2, \cdots, |z_k|^2) \in (\R_{\geq 0})^k$.
	We sometimes regard an element $\widetilde{\mathcal{A}}_{k+1}(\fa,\fb)$ as an element $\C^k$ under the identification $Z_{(\fa, \fb)}(z) \mapsto z$.
	The following lemma implies that the solution space is \emph{never} empty as long as $(\fa, \fb)$ obeys~\eqref{patternaandb}.
	
	\begin{lemma}[Lemma 3.5 in \cite{NNU1}]\label{lemma_NNU}
		Let $a_1, b_1, \cdots, a_k, b_k$, and $a_{k+1}$ be real numbers satisfying~\eqref{patternaandb}.
	Then there exists $z_1, \dots, z_k \in \mathbb{C}$ and $z_{k+1} \in \mathbb{R}$ such that
	\[
		  \begin{pmatrix} b_1 & & 0 & \bar{z}_1\\
                	  & \ddots & & \vdots\\
                  	0 & & b_k & \bar{z}_k\\
                  	z_1 & \hdots & z_k & z_{k+1}
		  \end{pmatrix}
	\]
	has eigenvalues $a_1, \dots , a_{k+1}$.
	\end{lemma}

	\begin{lemma}\label{lemma_fiber_torus}
		Suppose in addition that $a_1, b_1, \cdots, a_k, b_k$, and $a_{k+1}$ in Lemma~\ref{lemma_NNU} are all distinct.
		Then there exist positive numbers $\delta_1, \cdots, \delta_k$ such that
		\[
			\begin{array}{ccl}
			\widetilde{\mathcal{A}}_{k+1}(\fa,\fb) & = & \left \{
				  \begin{pmatrix} b_1 & & 0 & \bar{z}_1\\
                	  & \ddots & & \vdots\\
                  	0 & & b_k & \bar{z}_k\\
                  	z_1 & \hdots & z_k & z_{k+1}
		  \end{pmatrix}
		\in \mathcal{A}_{k+1}(\fa,\fb)
		  \,:\, (z_1,\cdots, z_k) \in \C^k, ~|z_i|^2 = \delta_i \textup{ for } i=1,\cdots, k \right \} \\[0.8em]
		  \end{array}
		\]
		where $z_{k+1} = \sum_{i=1}^{k+1} a_i - \sum_{i=1}^k b_i$.
		In particular, we have $\widetilde{\mathcal{A}}_{k+1}(\fa,\fb) \cong T^k$.
	\end{lemma}
	\begin{proof}
		We first note that if $|z_j|^2 = 0$ for some $j$, then the equation (\ref{equation_det}) (with respect to $x$) has a solution $x=b_j$. This implies that
		$b_j \in \{a_1, \cdots, a_{k+1} \}$, which contradicts to our assumption that $a_i$'s and $b_j$'s are all distinct.
		Thus, it is enough to show existence and uniqueness of
		a solution $(|z_1|^2, \cdots, |z_k|^2)$ of the system of linear equations in (\ref{equation_det}).
		
		The existence immediately follows from Lemma \ref{lemma_NNU}.
		Let
		\[
			\{|z_1|^2 = \delta_1 > 0, \cdots, |z_k|^2 = \delta_k > 0\}
		\]
		be one of solutions of (\ref{equation_det}) so that
		$\widetilde{\mathcal{A}}_{k+1}(\fa,\fb)$ contains a real $k$-torus
		\[
			T^k = \{(z_1, \cdots, z_k) \in \C^k ~|~ |z_i|^2 = \delta_i,~ i=1,\cdots,k \},
		\]
		which yields that
		$\dim_\R \widetilde{\mathcal{A}}_{k+1}(\fa,\fb) \geq k$.
		Since (\ref{equation_det}) is a system of non-homogeneous \emph{linear} equations with respect to the variables $|z_1|^2, \cdots, |z_k|^2$,
		the set
		$$
			\left\{ (|z_1|^2, \cdots, |z_k|^2) ~|~ (z_1, \cdots, z_k) \in \widetilde{\mathcal{A}}_{k+1}(\fa,\fb) \right\}
		$$
		is an affine subspace of $\R^k$. Therefore,
		$\dim_\R \widetilde{\mathcal{A}}_{k+1}(\fa,\fb) = k$ if and only if the equations (\ref{equation_det})
		has a unique solution $(|z_1|^2, \cdots, |z_k|^2) = (\delta_1, \cdots, \delta_k)$. It suffices to show that
		$\dim_\R \widetilde{A}_{k+1}(\fa,\fb) = k$.
		
		Note that $\mathcal{A}_{k+1}(\fa,\fb)$ is an $ \widetilde{\mathcal{A}}_{k+1}(\fa,\fb)$-bundle over $\mathcal{O}_{\fb}$ whose projection map is
		\[
			\begin{array}{ccccc}
					\rho_{k+1}  \colon & \mathcal{A}_{k+1}(\fa,\fb) & \rightarrow & \mathcal{O}_\fb \\
					                    &                x             & \mapsto    &  x^{(k)}.\\
			\end{array}
		\]
		More precisely, for each element $y \in \mathcal{O}_\fb$, there exists a unitary matrix $g_y \in U(k)$ (depending on $y$) such that
		\[
			g_y y g_y^{-1} = I_\fb
		\]
		where $I_\fb$ is the diagonal matrix $\textup{diag}(b_1, \cdots, b_k)$.
		Then the preimage $\rho_{k+1}^{-1}(y)$ of $y$ can be identified with $\widetilde{\mathcal{A}}_{k+1}(\fa,\fb)$ via
		\[
		\begin{array}{ccl}
			\rho_{k+1}^{-1}(y) & \longrightarrow & \widetilde{\mathcal{A}}_{k+1}(\fa,\fb) \\[0.5em]
			          Y  =  \begin{pmatrix} y & * \\
					                  	*^t & z_{k+1} \end{pmatrix}
					                  	& \mapsto & \begin{pmatrix} g_y & 0 \\
					                  	0 & 1 \end{pmatrix}
					                  	\cdot  Y \cdot
					                  	\begin{pmatrix} g_y^{-1} & 0 \\
					                  	0 & 1 \end{pmatrix}\\[2em]
					                  	
					                  	& & = \begin{pmatrix} g_y \cdot y \cdot g_y^{-1} & g_y \cdot * \\
					                  	*^t \cdot g_y^{-1} & z_{k+1}
								 \end{pmatrix} \\[2em]
								 & & =
								 \begin{pmatrix} I_\fb & g_y \cdot * \\
					                  	*^t \cdot g_y^{-1} & z_{k+1}
								 \end{pmatrix}
		\end{array}
		\]
		so that $\mathcal{A}_{k+1}(\fa,\fb)$ is an $ \widetilde{\mathcal{A}}_{k+1}(\fa,\fb)$-bundle over $\mathcal{O}_{\fb}$ via $\rho_{k+1}$.
		
		Now, we consider a sequence of real numbers $\fc = \{c_1, \cdots, c_{k-1}\}$ such that
		\[
		b_1 > c_1 > \cdots > b_{k-1} > c_{k-1} > b_k.
		\]
		Restricting the fibration $\rho_{k+1}$ to $\mathcal{A}_k(\fb,\fc)$, we similary have an $\widetilde{\mathcal{A}}_{k+1}(\fa,\fb)$-bundle over
		$\mathcal{A}_k(\fb,\fc)$.
		Note that its total space is the collection of $(k+1)\times (k+1)$ Hermitian matrices such that the spectra,
		the spectra of the $k \times k$ leading principal submatrix and the spectra of the $(k-1) \times (k-1)$ leading principal submatrix are $\fa, \fb$ and $\fc$,
		respectively.
		
		By a similar way described as above, we see that
		$\mathcal{A}_k(\fb,\fc)$ is an $\widetilde{\mathcal{A}}_k(\fb,\fc)$-bundle over $\mathcal{O}_\fc$
		with the projection map $\rho_k \colon \mathcal{A}_k(\fb,\fc) \rightarrow \mathcal{O}_\fc$ such that
		$\dim_\R \widetilde{\mathcal{A}}_k(\fb,\fc) \geq k-1$. Taking a sequence of real numbers $\fd = \{d_1, \cdots, d_{k-2}\}$
		so that
		\[
		c_1 > d_1 > \cdots > c_{k-2} > d_{k-2} > c_{k-1},
		\]
		the restriction of $\rho_k$ to $\mathcal{A}_{k-1}(\fc,\fd)$ induces an $\widetilde{\mathcal{A}}_k(\fb,\fc)$-bundle over $\mathcal{A}_{k-1}(\fc,\fd)$.
		
		Proceeding this procedure inductively, we end up obtaining a tower of bundles such that the total space $E$ is a generic fiber of the GC system of $\mcal{O}_\fa$. 
		Namely, $E$ is the preimage of a point in the interior of the GC polytope
		$\Delta_\fa$. By Proposition \ref{proposition_GS_smooth}, $E$ is a smooth manifold of dimension
		\[
			\dim_\R E = \frac{1}{2} \dim_\R \mathcal{O}_\fa = \frac{k(k+1)}{2}.
		\]
		On the other hand, by our construction, the dimension of $E$ is the sum of dimensions of all fibers of $\rho_i$'s for $i=2,\cdots,k+1$ so that
		\[
			\dim E = \dim \widetilde{\mathcal{A}}_{k+1} + \dim \widetilde{\mathcal{A}}_{k} + \cdots + \dim \widetilde{\mathcal{A}}_{2}, \quad \quad 
			\widetilde{\mathcal{A}}_{k+1} := \widetilde{\mathcal{A}}_{k+1}(\fa, \fb), \quad \widetilde{\mathcal{A}}_{k} := \widetilde{\mathcal{A}}_{k}(\fb, \fc), \quad \cdots 
		\]
		Since $\dim \widetilde{\mathcal{A}}_{i+1} \geq i$ for each $i$, we get $\dim \widetilde{\mathcal{A}}_{i+1} = i$ for every $i=1,\cdots,k$.
		This finishes the proof.
	\end{proof}
	
	Note that Lemma \ref{lemma_fiber_torus} deals with the case where $a_j \not \in \{b_1, \cdots, b_k\}$ for $j=1,\cdots,k+1$.
	Now, let us consider the case where $a_{j+1} \in \{b_1, \cdots, b_k\}$ for some $j \in \{0,1,\cdots, k\}$. 			
	Denoting the multiplicity of $a_{j+1}$ in $\fa$ by $\ell$, without loss of generality, we assume that $a_j > a_{j+1} = \cdots = a_{j+\ell} > a_{j+\ell+1}$.
     Then either $a_{j+1} = b_j$ or $a_{j+1} = b_{j+1}$. For the first case, there are two possible cases:
     \begin{equation}\label{eq:inequality1}
        \begin{cases}
        (1) ~a_j > b_j = a_{j+1} = \cdots = a_{j+\ell} > b_{j+\ell},   ~\mathrm{or}\\
        (2) ~a_j > b_j = a_{j+1} = \cdots = a_{j+\ell} = b_{j+\ell} > a_{j+\ell+1}.
        \end{cases}
     \end{equation}
     For the second case, we also have two possible cases :
     \begin{equation}\label{eq:inequality2}
        \begin{cases}
        (3) ~b_j > a_{j+1} = b_{j+1} = \cdots = a_{j+\ell} = b_{j+\ell} > a_{j+\ell+1},  ~\mathrm{or}\\
		(4) ~b_j > a_{j+1} = b_{j+1} =  \cdots = a_{j+\ell} > b_{j+\ell},\\
        \end{cases}
     \end{equation}
     Note that only the case (2) above have more multiplicity of $b$'s than $a$'s, i.e.,
     multiplicity $\ell+1$ and $\ell$ respectively: The cases (1) and (3) have
     the same multiplicities of both $a$ and $b$ while in the case (4) $a$ has multiplicity $\ell$ and $b$ has
     multiplicity $\ell -1$.

	We start with the inequality (3) of (\ref{eq:inequality2}).
	\begin{lemma}[case (3) of~\eqref{eq:inequality2}]\label{lemma_fiber_zero_a_b}
		Suppose that $b_j > a_{j+1} = b_{j+1} = \cdots = a_{j+\ell} = b_{j+\ell} > a_{j+\ell+1}.$ Then, every solution $(z_1, \cdots, z_k) \in \C^k$ of the equation
		(\ref{equation_det}) satisfies
		\[
			z_{j+1} = \cdots = z_{j+\ell} = 0.
		\]
	\end{lemma}

	\begin{proof}
	 	Observe that each term of the equation (\ref{equation_det})
	 	\[
	 		\det(xI - Z_{(\fa,\fb)}(z)) = (x - z_{k+1}) \cdot \prod_{i=1}^k (x-b_i) - \sum_{i=1}^k \left( \frac{|z_i|^2}{x-b_i} \cdot \prod_{m=1}^k (x-b_m) \right) = 0
		\]
		is divisible by $(x-b_{j+1})^{\ell-1}$ by our assumption. In particular, the first term of the equation
		\[
		(x - z_{k+1}) \cdot \prod_{i=1}^k (x-b_i)
		\]
		is divisible by $(x-b_{j+1})^{\ell}$. Similarly, since $b_{j+1} = \cdots = b_{j+\ell}$,
		\[
		\frac{|z_i|^2}{x-b_i} \cdot \prod_{m=1}^k (x-b_m)
		\]
		is divisible by $(x-b_{j+1})^{\ell}$ for each $i \not \in \{j+1, \cdots, j+\ell\}$.
		Taking
		\[
			g(x) := \det(xI - Z_{(\fa,\fb)}(z)) / (x-b_{j+1})^{\ell-1},
		\]
		we have $g(b_{j+1}) = g(a_{j+1}) = 0$ because $x = a_{j+1} = b_{j+1}$ is a solution of~\eqref{equation_det} with multiplicity $\ell$. It yields
		\[
			\left(|z_{j+1}|^2 + \cdots + |z_{j+\ell}|^2\right) \cdot \prod_{\substack{m=1 \\ m \not \in \{j+1, \cdots, j+\ell\}}}^k (b_{j+1}-b_m) = 0.
		\]
		Since
		\[
			\prod_{\substack{m=1 \\ m \not \in \{j+1, \cdots, j+\ell\}}}^k (b_{j+1}-b_m) \neq 0,
		\]
		we deduce that $|z_{j+1}|^2 + \cdots + |z_{j+\ell}|^2 = 0$ and hence $z_{j+1} = \cdots = z_{j+\ell} = 0$.
	\end{proof}
	
	Lemma \ref{lemma_fiber_zero_a_b} enables us to ignore the variables $z_{j+1}, \cdots, z_{j+\ell}$ for solving the equation \eqref{equation_det} for the case
	(3) in \eqref{eq:inequality2}.
	That is, if $\fa$ and $\fb$ contain subsequences satisfying (3) in \eqref{eq:inequality2} on $\fa$ and $\fb$, 
	the equation (\ref{equation_det}) is written by
	\[
			\det(xI - Z_{(\fa,\fb)}(z))
			        = (x-b_{j+1})^{\ell} \cdot
			        \left \{(x - w_{k+1-\ell}) \cdot \displaystyle \prod_{i=1}^{k-\ell} (x-b'_i) - \sum_{i=1}^{k-\ell}
			        \left( \frac{|w_i|^2}{x-b'_i} \cdot \prod_{m=1}^{k-\ell} (x-b'_m) \right) \right \}\\
			        =  0
	\]
		where
		\begin{itemize}
			\item $(b_1'. \cdots, b_{k-\ell}') = (b_1, \cdots, b_j, \hat{b_{j+1}}, \cdots, \hat{b_{j+\ell}}, b_{j+\ell+1}, \cdots, b_k)$,
			\item $(w_1. \cdots, w_{k-\ell}) = (z_1, \cdots, z_j, \hat{z_{j+1}}, \cdots, \hat{z_{j+\ell}}, z_{j+\ell+1}, \cdots, z_k)$, and
			\item $w_{k-\ell+1} = z_{k+1}$.
		\end{itemize}
	Observe that the equation
	\[
		\det(xI - Z_{(\fa,\fb)}(z)) / (x-b_{j+1})^{\ell} = 0
	\]	
	is same as the equation $\det(xI - Z_{(\fa',\fb')}(w)) = 0$ where
	\begin{itemize}
		\item $\fa' = (a_1', \cdots, a_{k-\ell+1}') = (a_1, \cdots, a_j, \hat{a_{j+1}}, \cdots, \hat{a_{j+\ell}}, a_{j+\ell+1}, \cdots, a_{k+1})$ and
		\item $\fb' = (b_1', \cdots, b_{k-\ell}')$.
	\end{itemize}
	Thus, $\widetilde{\mathcal{A}}_{k+1}(\fa,\fb)$ can be identified with $\widetilde{\mathcal{A}}_{k+1-\ell}(\fa', \fb')$ via
	\begin{equation}\label{identificationof}
		\begin{array}{cccc}
				  & \widetilde{\mathcal{A}}_{k+1}(\fa,\fb) & \rightarrow & \widetilde{\mathcal{A}}_{k+1-\ell}(\fa', \fb') \\
				                    &               		(z_1, \cdots, z_j, 0, \cdots, 0, z_{j+\ell+1}, \cdots, z_k)             & \mapsto    &  		(z_1, \cdots, z_j,  \hat{z_{j+1}}, \cdots, \hat{z_{j+\ell}}, z_{j+\ell+1}, \cdots, z_k).\\
		\end{array}
	\end{equation}
	Therefore, we obtain the following corollary.
	
	\begin{corollary}\label{corollary_fiber_zero_a_b}
		For sequences $\fa = (a_1, \cdots, a_{k+1})$ and $\fb = (b_1, \cdots, b_k)$ of real numbers obeying~\eqref{patternaandb}, suppose that there exist $j, \ell \in \Z_{>0}$ such that
		\[
			b_j > a_{j+1} = b_{j+1} = \cdots = a_{j+\ell} = b_{j+\ell} > a_{j+\ell+1}.
		\]
		Setting $\fa'$ (respectively $\fb'$) to be the sequence of real numbers obtained by deleting $a_{j+1}, \cdots, a_{j+\ell}$ (respectively $b_{j+1}, \cdots, b_{j+\ell}$), $\widetilde{\mathcal{A}}_{k+1}(\fa, \fb)$ can be identified with $\widetilde{\mathcal{A}}_{k+1-\ell}(\fa', \fb')$ under~\eqref{identificationof}.
	\end{corollary}
	
	The following two lemmas below are about the cases of (4) in \eqref{eq:inequality2} and (1) of \eqref{eq:inequality1}.
	Since they can be proven by using exactly same method of the proof of Lemma \ref{lemma_fiber_zero_a_b}, we omit the proofs.
	
	\begin{lemma}[case (1) of \eqref{eq:inequality1}]\label{lemma_fiber_zero_b_a}
		Suppose that $a_j > b_j = a_{j+1} = \cdots = a_{j+\ell} > b_{j+\ell}.$ Then
		$(z_1, \cdots, z_k) \in \widetilde{\mathcal{A}}_{k+1}(\fa,\fb)$ if and only if
		\begin{itemize}
			\item $z_{j} = \cdots =z_{j+\ell-1} = 0$, and
			\item $(z_1, \cdots, z_{j-1}, z_{j+\ell}, \cdots, z_{k}) \in \widetilde{\mathcal{A}}_{k-\ell+1}(\fa',\fb')$
		\end{itemize}
		where $\fa'$ is obtained by deleting $\{a_{j+1}, \cdots, a_{j+\ell} \}$ from $\fa$ and
		$\fb' = \{b_1', \cdots, b_{k-\ell}'\}$ is obtained by deleting $\{b_{j}, \cdots, b_{j+\ell-1} \}$ from $\fb$, respectively.
	\end{lemma}
	
	\begin{lemma}[case (4) of \eqref{eq:inequality2}]\label{lemma_fiber_zero_a_a}
		Suppose that $b_j > a_{j+1} = \cdots = a_{j+\ell} > b_{j+\ell}.$ Then
		$(z_1, \cdots, z_k) \in \widetilde{\mathcal{A}}_{k+1}(\fa,\fb)$ if and only if
		\begin{itemize}
			\item $z_{j+1} = \cdots =z_{j+\ell-1} = 0$, and
			\item $(z_1, \cdots, z_{j}, z_{j+\ell}, \cdots, z_{k}) \in \widetilde{\mathcal{A}}_{k-\ell+2}(\fa',\fb')$
		\end{itemize}
		where $\fa' = \{a_1', \cdots, a_{k-\ell+2}'   \}$ is obtained by deleting $\{a_{j+2}, \cdots, a_{j+\ell} \}$ from $\fa$ and
		$\fb' = \{b_1', \cdots, b_{k-\ell+1}' \}$ is obtained by deleting $\{b_{j+1}, \cdots, b_{j+\ell-1} \}$ from $\fb$.
	\end{lemma}
	
	Finally, we consider the case (2) of \eqref{eq:inequality1}.
	
	\begin{lemma}[case (2) of \eqref{eq:inequality1}]\label{lemma_fiber_sphere_b_b}
		Suppose that
		\[
			a_j > b_j = a_{j+1} = \cdots = a_{j+\ell} = b_{j+\ell} > a_{j+\ell+1}
		\]
		Then there exists a unique positive real number $C_j > 0$ such that
		\[
			|z_j|^2 + \cdots + |z_{j+\ell}|^2 = C_{j}.
		\]
		for any $(z_1, \cdots, z_k) \in \widetilde{\mathcal{A}}_{k+1}(\fa,\fb)$.
	\end{lemma}

	\begin{proof}
		By Lemma \ref{lemma_fiber_zero_a_b}, \ref{lemma_fiber_zero_b_a}, and \ref{lemma_fiber_zero_a_a},
		we may reduce $\fa = \{a_1, \cdots, a_{k+1} \}$ and $\fb = \{ b_1, \cdots, b_k \}$ to
		\[
			\fa' = \{a_1', \cdots, a_{r+1}' \}, \quad \text{and} \quad \fb' = \{b_1', \cdots, b_r' \}
		\]
		for some $r > 0$ so that
		there are no subsequences of type (3), (4) of (\ref{eq:inequality2}) or (1) of \eqref{eq:inequality1} in $(a_1' \geq b_1' \geq \cdots \geq a_r' \geq b_r' \geq a_{r+1}')$.
		Also, the above series of lemmas implies that $\widetilde{\mathcal{A}}_{k+1}(\fa,\fb)$ is identified with $\widetilde{\mathcal{A}}_{r+1}(\fa',\fb')$ under the
		 identification of $w = (w_1, \cdots, w_r)$ with
		suitable sub-coordinates $(z_{i_1}, \cdots, z_{i_r})$ of $(z_1, \cdots, z_{k+1})$.
		Therefore, it is enough to prove Lemma \ref{lemma_fiber_sphere_b_b}
		in the case where $(a_1, b_1, \cdots, a_k, b_k, a_{k+1})$ does not contain any pattern of type
	        (3), (4) of \eqref{eq:inequality2} or (1) of \eqref{eq:inequality1}.
		
		We temporarily assume that
		\[
			a_j > b_j = a_{j+1} = \cdots = a_{j+\ell} = b_{j+\ell} > a_{j+\ell+1}.
		\]
		is the unique pattern
		of type (2) of \eqref{eq:inequality1} in $(a_1, b_1, \cdots, a_k, b_k, a_{k+1})$.
		Then the equation (\ref{equation_det}) is written as
		\[
			\det(xI - Z) = (x-b_j)^{\ell} \cdot g(x)
		\]
		where
		\[
			\begin{array}{ccl}
				g(x) & = & (x-z_{k+1}) \displaystyle \prod_{\substack{i=1 \\ i \not\in \{j+1, \cdots, j+\ell\}}}^k (x-b_i)
				                   -  \sum_{i=1}^k \left( \frac{|z_i|^2}{x-b_i} \cdot \prod_{\substack{m=1 \\ m \not\in \{j+1, \cdots, j+\ell\}}}^k (x-b_m) \right) \\[0.2em]
			\end{array}\vspace{0.2cm}
		\]
		is a polynomial of degree $(k-\ell+1)$ with respect to $x$.
		For the sake of simplicity, we denote by
		\[
			\displaystyle B(x) := \prod_{\substack{m=1 \\ m \not\in \{j+1, \cdots, j+\ell\}}}^k (x-b_m).
		\]
		Since our assumption says that $\displaystyle \frac{1}{x-b_{j}} = \cdots = \frac{1}{x-b_{j+\ell}}$,
		the second part of $g(x)$ can be written by
		\begin{displaymath}
			\begin{array}{ccl}
			\displaystyle \sum_{i=1}^k \left( \frac{|z_i|^2}{x-b_i} \cdot B(x) \right) & = &
			\displaystyle
			\left( \displaystyle \frac{(|z_j|^2 + \cdots + |z_{j+\ell}|^2)}{x-b_j} + \sum_{\substack{i=1 \\ i \not\in \{j, \cdots, j+\ell\}}}^k  \frac{|z_i|^2}{x-b_i} \right) \cdot B(x)  \\
			\end{array}
		\end{displaymath}
		By substituting $\fa' = \{a_1' \cdots, a_{k+1-\ell}'\}$ and $\fb' = \{b_1', \cdots, b_{k-\ell}' \}$ where
		\begin{itemize}
			\item $a_i' = a_i$ \quad and  \quad $b_i' = b_i$ \quad for $1 \leq i \leq j$,
			\item $a_i' = a_{i+\ell}$ \quad for $j+1 \leq i \leq k-\ell+1$, and
			\item $b_i' = b_{i+\ell}$ \quad for $j+1 \leq i \leq k-\ell$,
		\end{itemize}
		$\fa' $ and $\fb'$ satisfies
		\[
			a_1' > b_1' > \cdots > a_{k-\ell}' > b_{k-\ell}' > a_{k+1-\ell}'.
		\]
		Then we have $g(x) = \det(xI - Z_{(\fa',\fb')}(w))$ where
		\begin{itemize}
			\item $|w_i|^2 =|z_i|^2$ \quad for $1 \leq i \leq j-1$,
			\item $|w_j|^2 = |z_j|^2 + \cdots + |z_{j+\ell}|^2$,
			\item $|w_i|^2 = |z_{i+\ell}|^2$ \quad for $j+1 \leq i \leq k-\ell$, and
			\item $w_{k-\ell+1} = \sum_{i=1}^{k-\ell+1} a_i' - \sum_{i=1}^{k-\ell}b_i' = \sum_{i=1}^{k+1} a_i - \sum_{i=1}^{k}b_i = z_{k+1}.$
		\end{itemize}
		Thus Lemma \ref{lemma_fiber_torus} implies that there exist positive constants $C_1, \cdots, C_{k-\ell}$ such that
		\[
			\widetilde{\mathcal{A}}_{k-\ell+1}(\fa',\fb') = \left\{(w_1, \cdots, w_{k-\ell}) \in \C^{k-\ell} ~|~ |w_j|^2 = C_j,~j=1,\cdots,k-\ell \right\}.
		\]
		In particular, we have $|w_j|^2 = |z_j|^2 + \cdots + |z_{j+\ell}|^2 = C_j$.
		
		It remains to prove the case where  $(a_1, b_1, \cdots, a_k, b_k, a_{k+1})$ contains more than one pattern of
                type (2) of \eqref{eq:inequality1}. However, since all patterns of
        	type (2) of \eqref{eq:inequality1} are disjoint from one another, we can apply the same argument to each
		pattern inductively. This completes the proof.
	\end{proof}
	
	Now, we are ready to prove Proposition \ref{proposition_A_k_S_k_bundle}.
	
	\begin{proof}[Proof of Proposition \ref{proposition_A_k_S_k_bundle}]
		For a given sequence $a_1\geq  b_1 \geq  \cdots \geq a_k \geq b_k \geq a_{k+1}$, let us consider the $W$-shaped block $W_k(a,b)$ with
		walls defined by strict inequalities $a_j > b_j$ or $b_j > a_{j+1}$ for each $j=1,\cdots, k$.  (See Figure \ref{figure_block_a_b}.)
		Note that each pattern of type (2) in \eqref{eq:inequality1} corresponds to an $M$-shaped block inside of $W_k(\fa,\fb)$.
		More specifically, if
		\[
			a_j > b_j = a_{j+1} = \cdots = a_{j+\ell} = b_{j+\ell} > a_{j+\ell+1}.
		\]
		is one of patterns of type (2) in \eqref{eq:inequality1} for some $j$, then it corresponds to a simple closed region which is an $M$-shaped block $M_{\ell+1}$.
		In particular, we have
		\[
			|M_{\ell+1}| = 2\ell + 1= \dim \left\{ (z_j, \cdots, z_{j+\ell}) \in \C^{\ell+1}~|~ |z_j|^2 + \cdots + |z_{j+\ell}|^2 = C_j \right\} = \dim S^{2\ell+1}.
		\]
		for a positive real number $C_j$.
		Combining the series of Lemma \ref{lemma_fiber_zero_a_b}, \ref{lemma_fiber_zero_b_a}, \ref{lemma_fiber_zero_a_a}, and \ref{lemma_fiber_sphere_b_b},
		we see
		$\widetilde{\mathcal{A}}_{k+1}(\fa,\fb) \cong S_{k}(\fa,\fb)$.
		Note that $\mathcal{A}_{k+1}(\fa,\fb)$ is an $\widetilde{\mathcal{A}}_{k+1}(\fa,\fb)$-bundle over $\mathcal{O}_\fb$ via
		\begin{equation}\label{equation_fibration_over_flag_manifold}
			\begin{array}{ccccc}
					\rho_{k+1}  \colon & \mathcal{A}_{k+1}(\fa,\fb) & \rightarrow & \mathcal{O}_\fb \\
					                    &                x             & \mapsto    &  x^{(k)}.\\
			\end{array}
		\end{equation}
		Thus $\mathcal{A}_{k+1}(\fa,\fb)$ is an $S_k(\fa,\fb)$-bundle over $\mathcal{O}_\fb$.
	\end{proof}

	\begin{corollary}\label{corollary_first_part}
		Let $f$ be a face of the Gelfand-Cetlin polytope $\Delta_\lambda$ and $\gamma$ be the face of the ladder diagram
		$\Gamma_\lambda$ corresponding to $f$. For any point $\textbf{\textup{u}}$ in the interior of $f$, the fiber $\Phi_\lambda^{-1}(\textbf{\textup{u}})$ has
		an iterated bundle structure given by
		\[
			\Phi_\lambda^{-1}(\textbf{\textup{u}}) = \bar{S_{n-1}}(\gamma) \xrightarrow {p_{n-1}} \bar{S_{n-2}}(\gamma) \rightarrow \cdots
			 \bar{S_1}(\gamma) \xrightarrow{p_1} \bar{S_0}(\gamma) = \mathrm{pt}
		\]
		where
		$p_{k} : \bar{S_k}(\gamma) \rightarrow \bar{S_{k-1}}(\gamma)$ is an $S_k(\gamma)$-bundle over $\bar{S_{k-1}}(\gamma)$ for $k=1,\cdots, n-1$.
		In particular, $\Phi_\lambda^{-1}(\textbf{\textup{u}})$ is of dimension
		\[
			\dim \Phi_\lambda^{-1}(\textbf{\textup{u}}) = \sum_{k=1}^{n-1} \dim S_k(\gamma).
		\]
	\end{corollary}
	\begin{proof}
		For each $(i,j) \in \Z^2_{\geq 1}$, we denote by $\Phi_\lambda^{i,j} \colon \mathcal{O}_\lambda \rightarrow \R$ be the component of $\Phi_\lambda$ which
		corresponds to the unit box $\square^{(i,j)}$ of $\Gamma_\lambda$.
		For each $k \in \Z_{\geq 1}$, 
		define
		\[
			a_i(k) := \Phi_\lambda^{i, k+1-i}(\textbf{\textup{u}}), \quad 1 \leq i\leq k \quad \quad \mathrm{and} \quad b_i(k) := a_i(k-1), \quad 1 \leq i\leq k-1. 
		\]
		Let $\fa{(k)} := (a_1(k), \cdots, a_k(k))$ and $\fb{(k)} := (b_1(k), \cdots, b_{k-1}(k))$.
		By applying Proposition \ref{proposition_A_k_S_k_bundle} repeatedly and following the observation that
		$S_k(\gamma) = S_k(\fa{(k+1)}, \fb{(k+1)})$, the fiber $\Phi_\lambda^{-1}(\textbf{\textup{u}})$ admits the iterated bundle structure described in ~\eqref{figure_iterated_bundle}.
		Then the dimension formula immediately follows.
	\end{proof}

\begin{equation}\label{figure_iterated_bundle}
	\xymatrix{
		\quad \overline{S_{n-1}(\gamma)} \quad \ar[r] \ar[d]^{{p_{n-1}}}& \quad \cdots \quad \ar[r] \ar[d] & \iota_{(n-1)}^* \left( \mcal{A} (\frak{a}{(n)} , \frak{b}{(n)})\right)   \ar[r] \ar[d]^{\iota^*_{(n-1)} \rho_n}  & \,  \mcal{A} (\frak{a}{(n)}, \frak{b}{(n)}) \, \ar@{^{(}->}[r]^{\,\,\,\,\,\,\,\,\, \iota_{(n)}} \ar[d]^{\rho_n} & \mcal{O}_{\frak{a}{(n)}}\\
		\quad \overline{S_{n-2}(\gamma)} \quad \ar[r] \ar[d]^{{p_{n-2}}} & \quad \cdots \quad \ar[r] \ar[d] & \,\,\, \mcal{A}(\frak{a}{(n-1)}, \frak{b}{(n-1)}) \,\,\, \ar@{^{(}->}[r]^{\,\,\,\,\,\,\,\,\,\,\,\,\,\,\,\, \iota_{(n-1)}} \ar[d]^{\rho_{n-1}}   &  \quad \mcal{O}_{\frak{a}{(n-1)}} \quad  & \\
		\quad \quad \vdots \quad \quad \ar[r] \ar[d] & \quad {\cdots} \quad   \ar@{^{(}->}[r] \ar[d] & \quad \quad \mcal{O}_{\frak{a}{(n-2)}} \quad \quad &  &\\
		\overline{S_1(\gamma)} = \mcal{A}(\frak{a}{(2)}, \frak{b}{(2)}) \ar@{^{(}->}[r]^{\quad \quad \,\,\,\,\,\,\,\,\, \iota_{(2)}} \ar[d]^{{p_1} = \rho_{2}} & \quad \cdots \quad &  & & \\
		\overline{S_0(\gamma)} = \mcal{O}_{\frak{a}{(1)}} &   &  & &}
\end{equation}
\vspace{0.2cm}

		To complete the proof of Theorem \ref{theorem_main}, it remains to verify that $\Phi^{-1}_\lambda(\textbf{\textup{u}})$ is an isotropic submanifold
		of $(\mathcal{O}_\lambda, \omega_\lambda)$ for every $\textbf{\textup{u}} \in \Delta_\lambda$.
		We first recall our notations as follows. 
		For a fixed positive integer $k>1$, let $\fa = (a_1, \cdots, a_{k+1})$ and $\fb = (b_1, \cdots, b_k)$ be sequences of real numbers satisfying~\eqref{patternaandb}
		and let $\rho_{k+1} \colon \mathcal{A}_{k+1}(\fa,\fb) \rightarrow \mathcal{O}_\fb$ be the map defined by $\rho_{k+1}(x) = x^{(k)}$.
		Then $\rho_{k+1}$ makes $\mathcal{A}_{k+1}(\fa,\fb)$ into a $\widetilde{\mathcal{A}}_{k+1}(\fa,\fb)$-bundle over $\mathcal{O}_\fb$. 
		See Proposition \ref{proposition_A_k_S_k_bundle}.
		Note that $U(k)$ acts on $\mathcal{A}_{k+1}(\fa,\fb)$ as a subgroup of $U(k+1)$ via the embedding
		\[
		\begin{array}{ccccl}
				i_k \colon& U(k) & \hookrightarrow & U(k+1) \\[0.8em]
				        &  A & \mapsto &
		\begin{pmatrix}A & 0\\[0.3em]
                  	0 & 1\\
		  \end{pmatrix}. \\
		 \end{array}
		\]
%
		\begin{lemma}\label{lemma_transitive}
			$U(k)$ acts transitively on $\mathcal{A}_{k+1}(\fa,\fb)$.
		\end{lemma}
		\begin{proof}
			Since any element $x \in \mathcal{A}_{k+1}(\fa,\fb)$ is conjugate to an element of the following form
			\[
			Z_{(\fa,\fb)}(z) =   \begin{pmatrix} b_1 & & 0 & \bar{z}_1\\
                	  & \ddots & & \vdots\\
                  	0 & & b_k & \bar{z}_k\\
                  	z_1 & \hdots & z_k & z_{k+1}
		  \end{pmatrix} \in \widetilde{\mathcal{A}}_{k+1}(\fa,\fb) \subset \mathcal{A}_{k+1}(\fa,\fb)
		  \]
		with respect to the $U(k)$-action, it is enough to show that the isotropy subgroup $U(k)_{I_\fb}$ of $I_\fb$
		acts on $\widetilde{\mathcal{A}}_{k+1}(\fa,\fb)$ transitively.
		
		Assume that $\fb$ is given such that
		\[
			b_1 = \cdots = b_{i_1} > b_{i_1 + 1} = \cdots = b_{i_2} > \cdots > b_{i_{r-1} + 1} = \cdots = b_{i_r} := b_k.
		\]
		for some $r \geq 1$ provided with $i_0 = 0$ and $i_r = k$.
		Then it is not hard to show that $U(k)_{I_b} = U(k_1) \times \cdots \times U(k_r)$ where $k_j = i_j - i_{j-1}$ for $j=1,\cdots, r$.
		For each $j$, we know that each $(z_{i_{j-1}+1}, \cdots, z_{i_{j}}) \in \C^{k_{j}}$ satisfies either
		\begin{itemize}
			\item $|z_{i_j+1}|^2 + \cdots + |z_{i_{j+1}}|^2 = 0$, or
			\item $|z_{i_{j-1}+1}|^2 + \cdots + |z_{i_{j}}|^2 = C_{j}$ for some positive constant $C_{j} \in \R_{>0}$
		\end{itemize}
		by Lemma \ref{lemma_fiber_zero_a_b}, \ref{lemma_fiber_zero_b_a}, \ref{lemma_fiber_zero_a_a}, and \ref{lemma_fiber_sphere_b_b}.
		In the latter case, $U(k)_{I_b}$-action is written as
		\[
		\begin{pmatrix} g & 0\\
                	  0 & 1\\
		\end{pmatrix}
		\cdot
		\begin{pmatrix} I_b & \bar{z}^t\\
                	  z & z_{k+1}\\
		\end{pmatrix}
		\cdot
		\begin{pmatrix} g^{-1} & 0\\
                	  0 & 1\\
		\end{pmatrix}
		=
		\begin{pmatrix} I_b & g\bar{z}^t\\
                	  zg^{-1} & z_{k+1}\\
		\end{pmatrix}
		\]
		for every $g \in U(k)_{I_b}$ and $z = (z_1, \cdots, z_k)$.
		Note that every $g \in U(k)_{I_b}$ is of the form
		\[
		g = \begin{pmatrix} g_1 & 0 & 0 & \cdots & 0\\
                	  0 & g_2 & 0 & \cdots & 0\\
                	  0 & 0 & \ddots & \cdots & \vdots \\
                	  \vdots & \vdots & \vdots & \ddots & \vdots\\
                	  0 & \cdots & \cdots & \cdots & g_r
		\end{pmatrix}
		\]		
		where $g_i \in U(k_i)$ for $i=1,\cdots,r$. Thus each $g \in U(k)_{I_b}$ acts on $(z)_{j} := (z_{i_{j-1}+1}, \cdots, z_{i_{j}}) \in \C^{k_{j}}$
		by $(z)_{j} \cdot g_{j}^{-1}$ which is equivalent to the standard linear $U(k_{j})$-action on the sphere
		$S^{2k_{j} - 1} \subset \C^{k_j}$ of radius $\sqrt{C_{j}}$.
		Therefore, the action is transitive.
		\end{proof}
		
		\begin{lemma}\label{lemma_isotropic}
			For each $x \in \mathcal{A}_{k+1}(\fa,\fb)$ and any $\xi, \eta \in T_x \mathcal{A}_{k+1}(\fa,\fb)$, we have
			\[
				(\omega_\fa)_x(\xi, \eta) = (\omega_\fb)_{\rho_{k+1}(x)}((\rho_{k+1})_* \xi, (\rho_{k+1})_* \eta).
			\]
			In particular, the vertical tangent space $V_x \subset T_x \mathcal{A}_{k+1}(\fa,\fb)$ of $\rho_{k+1}$ is contained in $\ker (\omega_\fa)_x$.
		\end{lemma}
		
		\begin{proof}
			Note that Lemma \ref{lemma_transitive} implies that any tangent vector in $T_x \mathcal{A}_{k+1}(\fa,\fb)$ can be written as $[(i_k)_*(X), x]$ for some
			$X \in \mathfrak{u}(k)$ where
			\[
				(i_k)_*(X) = 		\begin{pmatrix}X & 0\\[0.3em]
                  	            0 & 0\\
		  \end{pmatrix} \in \mathfrak{u}(k+1).
			\]
			Thus for any $\xi, \eta \in T_x \mathcal{A}_{k+1}(\fa,\fb)$, there exist $X, Y \in \mathfrak{u}(k)$ such that
			\[
				\xi = [(i_k)_*(X), x], \quad \eta = [(i_k)_*(Y), x].
			\]
			Therefore, we have
			\[
				(\omega_\fa)_x(\xi, \eta) = \mathrm{tr}(ix[(i_k)_*(X), (i_k)_*(Y)]) = \mathrm{tr}(ix^{(k)}[X,Y]) = (\omega_\fb)_{x^{(k)}} ([X, x^{(k)}] , [Y, x^{(k)}])	
			\]
			since all entries of the $k+1$-th row and column of $[(i_k)_*(X), (i_k)_*(Y)]$ are zeros so that  
			the $(k+1, k+1)$-th entry of the matrix $x[(i_k)_*(X), (i_k)_*(Y)]$ is zero. Furthermore,
			since $\rho_{k+1}$ is $U(k)$-invariant, we obtain that $[X, x^{(k)}] = (\rho_{k+1})_*(\xi)$ and
			$[Y, x^{(k)}] = (\rho_{k+1})_*(\eta)$.			
			This completes the proof.
		\end{proof}

	\begin{proposition}\label{proposition_second_part}
		For any $\textbf{\textup{u}} \in \Delta_\lambda$, the fiber $\Phi_\lambda^{-1}(\textbf{\textup{u}})$ is an isotropic submanifold of
		$(\mathcal{O}_\lambda,\omega_\lambda)$, i.e.,
		$\omega|_{\Phi_\lambda^{-1}(\textbf{\textup{u}})} = 0$.
	\end{proposition}

	\begin{proof}
		Suppose that $\gamma$ is a face of $\Gamma_\lambda$
		such that the corresponding face $f_\gamma$ contains $\textbf{\textup{u}}$ in its interior.
		Let $x \in \Phi^{-1}_\lambda(\textbf{\textup{u}})$ and let $\xi, \eta \in T_x \Phi^{-1}_\lambda(\textbf{\textup{u}})$.
		Then Corollary \ref{corollary_first_part} says that $\Phi^{-1}_\lambda(\textbf{\textup{u}})$ is the total space of an iterated bundle
		\[
				\Phi_\lambda^{-1}(\gamma) = \bar{S_{n-1}}(\gamma) \xrightarrow {p_{n-1}} \bar{S_{n-2}}(\gamma) \rightarrow \cdots
		 		\bar{S_1}(\gamma) \xrightarrow{p_1} \bar{S_0}(\gamma) = \mathrm{pt}
		\]
		described in~\eqref{figure_iterated_bundle}.

		For each integer $k$ with $ 1 \leq k \leq n$ and $1 \leq i \leq k$,
		let us define
		\[
			a_i(k) := \Phi_\lambda^{i, k+1-i}(\textbf{\textup{u}}).
		\]
		and let $\fa(k) := (a_1(k), \cdots, a_k(k))$.
		In particular, we have $\fa(n) = \lambda = \{ \lambda_1, \cdots, \lambda_n \}$.
		Then
		Lemma \ref{lemma_isotropic} implies that
		\[
			(\omega_{\fa(n)})_x (\xi, \eta) = (\omega_{\fa(n-1)})_{\rho_n(x)}((\rho_n)_*(\xi), (\rho_n)_*(\eta)).
		\]
		Since
		$p_{n-1}$ is the restriction of $\rho_n$ to $\bar{S}_{n-1}(\gamma) \subset \mathcal{A}(\fa(n), \fa(n-1))$,
		both $(\rho_n)_*(\xi)$ and $(\rho_n)_*(\eta)$ are lying on $T_{p_{n-1}(x)} \bar{S}_{n-2}(\gamma) \subset
		T_{p_{n-1}(x)} \mathcal{A}(\fa(n-1), \fa(n-2))$.
		Applying Lemma \ref{lemma_isotropic} inductively, we obtain
		\[	\begin{array}{ccl}
			(\omega_{\fa(n)})_x (\xi, \eta) & = & (\omega_{\fa(n-1)})_{\rho_n(x)}((\rho_n)_*(\xi), (\rho_n)_*(\eta)) \\
                                             & = & (\omega_{\fa(n-2)})_{\rho_{n-1} \circ \rho_n(x)}((\rho_{n-1} \circ \rho_n)_*(\xi), (\rho_{n-1} \circ \rho_n)_*(\eta)) \\
                                             & = & \cdots \\
			                     & = & (\omega_{\fa(2)})_{\rho_2 \circ \cdots \circ \rho_n(x)}((\rho_2 \circ \cdots \circ \rho_n)_*(\xi), (\rho_2 \circ \cdots \circ \rho_n)_*(\eta))\\
			                     & = & 0
			\end{array}
		\]
		where the last equality follows from $\rho_2 \colon \mathcal{A}(\fa(2), \fa(1)) \rightarrow \{ a_1(1) = \Phi_\lambda^{1,1}(\textbf{\textup{u}}) \} = \mathrm{point}.$
	\end{proof}
	
	\begin{proof}[Proof of Theorem~\ref{theorem_main}]
		It follows from Corollary~\ref{corollary_first_part} and  Proposition~\ref{proposition_second_part}.
	\end{proof}

\section{Degenerations of fibers to tori}
\label{secDegenerationsOfFibersToTori}

In this section, we describe how each fiber of a GC system 
degenerates to a torus fiber of a toric moment map. 

Let $\lambda$ be given in \eqref{equation_GC-pattern} and let $f$ be a face of the GC polytope $\Delta_\lambda$ of dimension $r$. 
It is shown in Section ~\ref{secIteratedBundleStructuresOnGelfandCetlinFibers} that a fiber $\Phi_\lambda^{-1}(\textbf{\textup{u}})$ is the total space of an iterated bundle 
	\begin{equation}\label{equation_iterated_S}
		\Phi_\lambda^{-1}(\textbf{\textup{u}}) \cong \bar{S_{n-1}}(\gamma_f) \xrightarrow {p_{n-1}} \bar{S_{n-2}}(\gamma_f)
		 \rightarrow \cdots
		 \xrightarrow{p_2} \bar{S_1}(\gamma_f) \xrightarrow{p_1} \bar{S_0}(\gamma_f) := \mathrm{point}
	\end{equation}
where $\gamma_f$ is the face of  $\Gamma_\lambda$ corresponding to $f$, see Corollary \ref{corollary_first_part}.
The main theorem of this section further claims that 
every $S^1$-factor appeared in each stage of the iterated bundle in \eqref{equation_iterated_S} comes out as a trivial factor. 
Namely, letting $\bar{S_{i}}(\gamma)$ be $(S^1)^{m_i} \times Y_i$ such that $\sum_{i=1}^{n-1} m_i = r$, the fiber is written as the product
$$
\Phi_\lambda^{-1}(\textbf{\textup{u}}) = (S^1)^r \times Y(\textbf{\textup{u}})
$$
where $Y(\textbf{\textup{u}})$ is the total space of an iterated bundle whose fibers at stages are $Y_{\bullet}$'s, see Theorem~\ref{theorem_contraction}. Indeed, $Y(\textbf{\textup{u}})$ shrinks to a point through a toric degeneration, which illustrates how fibers degenerate into toric fibers. As an application, we can compute the fundamental group and the second homotopy group of 
$\Phi_\lambda^{-1}(\textbf{\textup{u}})$, which will be crucially used in \cite{CKO2} and \cite{CKO3}.

Recall that for a given K\"{a}hler manifold $(X,\omega)$, 
a flat family $\pi \colon \mcal{X} \rightarrow \C$ of algebraic varieties is called a 
{\em toric degeneration} $(X, \omega)$ if there exists an algebraic embedding $i \colon \mcal{X} \hookrightarrow \p^N \times \C$ such that 
\begin{enumerate}
	\item the diagram
		\begin{equation}\label{equation_def_toric_degeneration}
		    \xymatrix{
			      \mcal{X}  ~\ar[drr]_{\pi} \ar@{^{(}->}[rr]^{i}
			      & & \p^N \times \C \ar[d]^{q} \\ 
			      & & \C 
			    }
		\end{equation} commutes where $q \colon \p^N \times \C \rightarrow \C$ is the projection to the second factor, 
	\item $\pi^{-1}(\C^*)$ is isomorphic to $ X \times \C^*$	as a complex variety, and 
	\item For the product K\"{a}hler form $\widetilde{\omega} := \omega_{\mathrm{FS}} \oplus \omega_{\mathrm{std}}$ 
	on $\p^N \times \C$ where 
	$\omega_{\mathrm{FS}}$ is some multiple of the Fubini-Study form on $\p^N$ and $\omega_{\mathrm{std}}$ is the standard K\"{a}hler form 
	on $\C$, 
		\begin{itemize}
			\item $(X_1, \widetilde{\omega}|_{X_1})$ is symplectomorphic to  $(X, \omega)$, and 
			\item $X_0$ is a projective toric variety (in $\p^N$) and $\widetilde{\omega}|_{X_0}$ is a torus invariant K\"{a}hler form 
			denoted by $\omega_0$
		\end{itemize}
	where $X_t := \pi^{-1}(t) \cong i(\pi^{-1}(t)) \subset \p^N \times \{t \}$ is a projective variety for every $t \in \C$. 
\end{enumerate}	

Let $\mcal{X}^{\mathrm{sm}} \subset \mcal{X}$ be the smooth locus of $\mcal{X}$.  
The Hamiltonian vector field, denoted by $\widetilde{V}_\pi$, for the imaginary part $\mathrm{Im}(\pi)$ of $\pi$ is defined on $\mcal{X}^{\mathrm{sm}}$.
By the holomorphicity, $\pi$ satisfies the Cauchy-Riemann equation which induces $\nabla\mathrm{Re}(\pi) = - \widetilde{V}_\pi$ where 
$\nabla$ denotes the gradient vector field with respect to a K\"{a}hler metric associated with $\widetilde{\omega}$. 
We call
\[
	V_\pi := \widetilde{V}_\pi / ||\widetilde{V}_\pi||^2
\]
the {\em gradient-Hamiltonian vector field of $\pi$}, see Ruan \cite{Ru}. 
Then the one-parameter subgroup generated by $V_\pi$ induces a symplectomorphism 
\begin{equation}\label{equation_symplecto_open_dense}
	\phi \colon (\mcal{U}, \omega) \rightarrow (\phi(\mcal{U}), \omega_0)
\end{equation}
on an open dense subset $\mcal{U}$ of $X$ ($\cong X_1$) such that $\phi(\mcal{U}) = X_0^{\mathrm{sm}}$ and it is extended to 
a surjective continuous map 
\[
	\phi \colon X \rightarrow X_0
\]
defined on the whole $X$, see Harada-Kaveh \cite[Theorem A]{HK} for more details. 

We may also consider a toric degeneration of a K\"{a}hler manifold {\em ``equipped with a completely integrable system''} as follows.
Consider a triple $(X,\omega,\Theta)$ where $\Theta = (\Theta_\alpha)_{\alpha \in \mcal{I}}$ is a (continuous) completely integrable system on $(X, \omega)$
and $\mcal{I}$ is the index set for $\Theta$ such that $|\mcal{I}| = \dim_\C X$. 
We call $\pi \colon \mcal{X} \rightarrow \C$ a {\em toric degeneration} of $(X, \omega, \Theta)$ if $\pi$ is a toric degeneration of 
$(X,\omega)$ and $\Theta = \Phi \circ \phi$ where $\Phi \colon X_0 \rightarrow \R^{|\mcal{I}|}$ is the toric moment map on $(X_0, \omega_0)$, see \cite[Definition 1.1]{NNU1}.
\begin{equation}\label{equation_toric_degeneration_diagram_Section 7}
	\xymatrix{
		(X_1,\omega_1)\cong (X,\omega)  \ar[dr]_{\Theta} \ar[rr]^{ \phi}
                            & & (X_0, \omega_0) \subset \p^N \times \{0\}
      \ar[dl]^{\Phi} \\
   & \Delta_0 &}
\end{equation}
The Hamiltonian vector field of each component $\Phi_\alpha$ of $\Phi$ ($ = \{ \Phi_\alpha \}_{\alpha \in \mcal{I}}$) is globally defined on $X_0$, even though $X_0$ is singular, by the following reason. 
Note that $X_0 \subset \p^N \times \{0 \}$ is a projective toric variety, which is the Zariski closure of the single $(\C^*)^{|\mcal{I}|}$-orbit on $X_0$.
The $(\C^*)^{|\mcal{I}|}$-action on $X_0$ extends to the linear Hamiltonian action on $\p^N$ with respect to $\omega_{\mathrm{FS}}$. 
We denote by $(S^1)^{|\mcal{I}|}$ the maximal compact subtorus of $(\C^*)^{|\mcal{I}|}$, by $\widetilde{\Phi} = (\widetilde{\Phi}_\alpha)_{\alpha \in \mcal{I}}$
a moment map for the $(S^1)^{|\mcal{I}|}$-action on $\p^N$, and by $\xi_\alpha$ the fundamental vector field of $\widetilde{\Phi}_\alpha$ on $\p^N$ for each $\alpha \in \mcal{I}$.
Then each component $\Phi_\alpha$ coincides with the restriction of $\widetilde{\Phi}_\alpha$ to $X_0$. 
Since $X_0$ is $T^{|\mcal{I}|}$-invariant, the trajectory of the flow of $\xi_\alpha$ passing through any point of $X_0$ 
should be on $X_0$. In other words, the restriction $\xi_\alpha|_{X_0}$ should be tangent\footnote{Every toric variety is a stratified space \cite{LS}
so that each point in $X_0$ is contained in an open smooth stratum and each vector field $\xi_\alpha$ is tangent to the stratum.} 
to $X_0$, and therefore the Hamiltonian vector field of $\Phi_\alpha$ is defined on the whole $X_0$. 

Now, let $\mcal{V}_\alpha$ be the open dense subset of $X$
on which $\Theta_\alpha$ is smooth. 
Then the Hamiltonian vector field, denoted by $\zeta_\alpha$, of $\Theta_\alpha$ is defined on $\mcal{V}_\alpha$. 
For any subset $\mcal{I}' \subset \mcal{I}$, we let
\begin{equation}\label{domainofsmoothofcomponent}
	\displaystyle \mcal{V}_{\mcal{I}'} := \displaystyle \bigcap_{\alpha \in \mcal{I}'} \mcal{V}_\alpha
\end{equation}
so that the Hamiltonian vector field of $\Theta_\alpha$ is defined on $\mcal{V}_{\mcal{I}'}$ for every $\alpha \in \mcal{I}'$. 
If $\Theta_\alpha$ is a periodic Hamiltonian, i.e., if $\Theta_\alpha$ generates a circle action for every $\alpha \in \mcal{I}'$, 
then $\mcal{V}_{\mcal{I}'}$ admits the $T^{|\mcal{I}'|}$-action generated by $\{ \zeta_\alpha \}_{\alpha \in \mcal{I}'}$. 
Note that $\mcal{V}_{\mcal{I}'}$ is open dense in $X$ so that $\mcal{U} \cap \mcal{V}_{\mcal{I}'}$ 
is also open dense in $X$ where $\mcal{U}$ is in \eqref{equation_symplecto_open_dense}.

\begin{lemma}\label{lemma_equivariant}
	For any $\alpha \in \mcal{I}$ and $p \in \mcal{V}_\alpha$, we have 
	\[
		\phi(\exp(t \, \zeta_\alpha) \cdot p) =\exp(t \, \xi_\alpha) \cdot \phi(p)
	\]
	for every $t \in \R$.  
\end{lemma}
\begin{proof}
	Fix $\alpha \in \mcal{I}$.
	From the fact that 
	\begin{itemize}
		\item $\phi^*\omega_0 = \omega$ on $\mcal{U} \cap \mcal{V}_\alpha$, and 
		\item $\Theta = \Phi \circ \phi$, 
	\end{itemize}	
	it follows that
	\[
		\omega_0(\phi_*(\zeta_\alpha), \phi_*(\cdot)) = \omega(\zeta_\alpha, \cdot ) = d\Theta_\alpha(\cdot) = d\Phi_\alpha \circ d\phi(\cdot) = \omega_0(\xi_\alpha, \phi_*(\cdot))
	\]
	so that $\phi_*(\zeta_\alpha) = \xi_\alpha$ on $\mcal{U} \cap \mcal{V}_\alpha$.
	Since $\mcal{U} \cap \mcal{V}_\alpha$ is open dense in $\mcal{V}_\alpha$ and $\xi_\alpha$ is defined on whole $X_0$, the equality 
	\[
		\phi_*(\zeta_\alpha) = \xi_\alpha
	\]
	holds on $\mcal{V}_\alpha$. 
	This completes the proof by the uniqueness of a solution to a first-order differential equation.  
\end{proof}

Let $\mcal{I}' \subset \mcal{I}$ and suppose that $\Theta_\alpha$ is a periodic Hamiltonian on $\mcal{V}_{\mcal{I}'}$ for every $\alpha \in \mcal{I}'$. 
Since $\Phi_\alpha$ is also a periodic Hamiltonian on $X_0$, we deduce the following immediately from Lemma \ref{lemma_equivariant}.

\begin{corollary}\label{corollary_torus_equivariant}
	Let $\mcal{I}' \subset \mcal{I}$ such that $\{ \Theta_\alpha \}_{\alpha \in \mcal{I}'}$ are periodic Hamiltonians on $\mcal{V}_{\mcal{I}'}$. 
	Then $\phi$ is $T^{|\mcal{I}'|}$-equivariant on $\mcal{V}_{\mcal{I}'}$. 
\end{corollary}

We will apply Corollary \ref{corollary_torus_equivariant} to GC systems. Recall that for any $\lambda$ given in \eqref{equation_GC-pattern}, Nishinou-Nohara-Ueda \cite{NNU1}
built a toric degeneration of the GC system $\Phi_\lambda$ on $(\mcal{O}_\lambda, \omega_\lambda)$. 
We briefly explain their construction of the toric degeneration of 
$(\mcal{O}_\lambda, \omega_\lambda, \Phi_\lambda)$ and a continuous map $\phi \colon \mcal{O}_\lambda \rightarrow X_0$ given in 
\eqref{equation_symplecto_open_dense}. (We also refer the reader to \cite{KM} or \cite{NNU1} for more details.) 

We begin with Kogan-Miller's {\em toric degeneration in stage} in \cite{KM}, an $(n-1)$-parameter family $F \colon \mcal{X} \rightarrow \C^{n-1}$
of projective varieties and it can be factored as 
\begin{equation}{\label{equation_KoM}}
	F = q \circ \iota, \quad \quad 
	\mcal{X} \stackrel{\iota}  \hookrightarrow  P \times \C^{n-1} \stackrel {q} \rightarrow \C^{n-1}, \quad 
	P = \prod_{k=1}^r \p_{n_k}
\end{equation}
where $\p_{n_k} := \p\left(\bigwedge^{n_k} \C^n \right) = \p^{{n \choose n_k}-1}$ and $\iota$ is an algebraic embedding with a K\"{a}hler form $\widetilde{\omega}$ on $P \times \C^{n-1}$ such that 
\begin{itemize}
	\item $(F^{-1}(1,\cdots,1), \widetilde{\omega}|_{F^{-1}(1,\cdots,1)}) \cong (\mcal{O}_\lambda, \omega_\lambda)$ and 
	\item $F^{-1}(0,\cdots,0)$ is isomorphic to the toric variety $X_0$ whose moment map image is $\Delta_\lambda$ 
		with the torus-invaraint K\"{a}hler form $\widetilde{\omega}|_{F^{-1}(0,\cdots,0)}$ on $X_0$. 
\end{itemize}
See \cite[Section 5 and Remark 5.2]{NNU1} for more details. Following \cite{NNU1}, we denote the coordinates of $\C^{n-1}$ by $(t_2, \cdots, t_n)$ and 
$F^{-1}(1,\cdots,1, t = t_k , 0,\cdots, 0)$ by $X_{k,t}$ for $2 \leq k \leq n$ and $t \in \C$. 
Then the set 
\[
	\{ X_{k,t} \}_{2 \leq k \leq n, t \in \C}
\] can be regarded as a family of algebraic varieties in $P$ via the embedding $\iota$ where 
$X_{n,1} \subset P$ is the image of the Pl\"{u}cker embedding of $\mcal{O}_\lambda$
and $X_{2,0} \subset P$ is the toric variety $X_0$ associated with $\Delta_\lambda$.

Let $T^{\frac{n(n-1)}{2}}$ be the compact subtorus of $(\C^*)^{\frac{n(n-1)}{2}} \subset X_0$ and we fix a decomposition 
\[
	T^{\frac{n(n-1)}{2}} \cong T^1 \times T^2 \times \cdots \times T^{n-1}.
\]
For each $k = 1,\cdots,n-1$, we denote the $i$-th coordinate of $T^k$ by $\tau_{i,j}$
\footnote{
For the consistency of~\eqref{equation_coordinate}, we use the index $(i,j)$. 
} where $i + j = k+1$. 
Each $S^1$-action on $X_0$ corresponding to $\tau_{i,j}$ can be extended to the linear Hamiltonian $S^1$-action on $P$ and we denote a corresponding moment map by 
\begin{equation}{\label{phiijaa}}
	\Phi^{i,j} : P \rightarrow \R.
\end{equation}

On the other hand, $P$ admits the Hamiltonian $U(n)$-action induced by the standard linear $U(n)$-action on $\C^n$ with the moment map $\mu^{(n)} : P \rightarrow \frak{u}(n)^*$.
We also denote by 
\[
	\mu^{(k)} \colon P \rightarrow \frak{u}(k)^* \cong \{\text{$k \times k$ Hermitian matrices} \}
\]
the moment map for each $U(k)$-action where $U(k)$ is the subgroup of $U(n)$ given by 
\[
	U(1) \subset U(2) \subset \cdots \subset U(n-1) \subset U(n), \quad U(k) := \begin{pmatrix} U(k) & 0 \\ 0 & I_{n-k} \end{pmatrix}
\]
for $k=1,\cdots,n-1$. For each pair $(i,j) \in (\Z_+)^2$ with $i + j = k + 1$, define 
\begin{equation}\label{phiijbb}
	\Phi_\lambda^{i,j} \colon P \rightarrow \R
\end{equation}
which assigns the $i$-th largest eigenvalue of $\mu^{(k)}(p)$ for every $p \in P$. 

\begin{remark}
If one follows the notations used in \cite[Section 5]{NNU1}, then we may express as
\[
	\tau_{i,j} = \tau_i^{(j)}, \quad \Phi^{i,j} = v^{(k)}_i, \quad \Phi^{i,j}_\lambda = \lambda^{(k)}_i, \quad i+j = k+1.
\]
\end{remark}

In general, a fiber of $F$ in \eqref{equation_KoM} is not invariant under neither the $U(n)$-action nor under the $T^1 \times T^2 \times \cdots \times T^{n-1}$-action on $P$,  
but $X_{k,t}$ is invariant under the $U(k-1)$ action and the $T^k \times \cdots \times T^{n-1}$ action. Moreover, $X_{k+1,0} = X_{k,1}$ admits both the $U(k)$-action 
and the $T^k \times \cdots \times T^{n-1}$ action. 

The following theorem states that the maps $\Phi^{i,j}_\lambda$'s in \eqref{phiijbb}
and $\Phi^{i,j}$'s in \eqref{phiijaa} defined on $P$ induces a completely integrable system on $X_{k,t}$ and how the GC system $\Phi_\lambda$ on $X_{n,1} \cong \mcal{O}_\lambda$
degenerates into the toric moment map $\Phi$ on $X_{2,0} = X_0$ {\em in stages}. See also Section 5 and Section 7 of \cite{NNU1}.

\begin{theorem}[Theorem 6.1 in \cite{NNU1}]\label{theorem_NNU_6_1}
For every $k \geq 2$ and $t \in \C$, the map
\[
	\Phi_{k,t} := \left.(\underbrace{\Phi^{1,1}_\lambda}_{1}, \cdots, \underbrace{\Phi^{1, k-1}_\lambda, \cdots, \Phi^{k-1,1}_\lambda}_{k-1}, 
	\underbrace{\Phi^{1, k}_{\vphantom{\lambda}}, \cdots, \Phi^{k,1}_{\vphantom{\lambda}}}_{k}, \cdots, 
	\underbrace{\Phi^{1,n-1}_{\vphantom{\lambda}}, \cdots, \Phi^{n-1,1}_{\vphantom{\lambda}}}_{n-1})\right|_{X_{k,t}} 
	: X_{k,t} \rightarrow \R^{\dim \Delta_\lambda}
\]
is a completely integrable system on $X_{k,t}$ in the sense of Definition \ref{definition_CIS_continuous}. Moreover, 
we have $\Phi_\lambda^{i, k - i} = \Phi_{\vphantom{\lambda}}^{i, k - i}$
on $X_{k, 0} = X_{k-1, 1}$ for every $i=1,\cdots, k-1$.
\end{theorem}

Note that $\Phi_{k,t}$'s are related to one another via the {\em gradient-Hamiltonian flows} introduced by Wei-Dong Ruan \cite{Ru}. 
For each $m=1,\cdots, n-1$, let $F_m$ be the $m$-th component of $F$ in \eqref{equation_KoM} and 
let $\widetilde{V}_m$ be the Hamiltonian vector field of the imaginary part of $F_m$
on the smooth locus $\mcal{X}^{\mathrm{sm}}$ of $\mcal{X}$. Then the gradient-Hamiltonian vector field is defined by 
\[
	V_m := \widetilde{V}_m / ||\widetilde{V}_m||^2. 
\]
The flow of $V_m$, which we denote by $\phi_{m,t}$ where $t$ is a time parameter, 
preserves the fiberwise symplectic form and so $\phi_{m,t}$ induces a symplectomorphism  
on an open subset of each fiber on which $\phi_{m,t}$ is smooth. 
As a corollary, we have the following.
\begin{corollary}[Corollary 7.3 in \cite{NNU1}]\label{corollary_NNU_deform}
	 The gradient-Hamiltonian vector field $V_k$ gives a deformation of $X_{k,t}$ 
	 preserving the structure of completely integrable systems.
 	 In particular, we have the following commuting diagram for every $t \geq 0$ : 
	  \begin{equation}\label{equation_preserve_integrable_system}
		    \xymatrix{
			      X_{k,1}  \ar[dr]_{\Phi_{k,1}} \ar[rr]^{\phi_{k,1-t}}
			      & & X_{k,t} \ar[dl]^{\Phi_{k,t}} \\ 
			      & \Delta_{\lambda} &
			    }
	  \end{equation}
\end{corollary}

By Corollary \ref{corollary_NNU_deform}, we obtain a continuous map 
\begin{equation}\label{phicontimapx}
	\phi \colon \mcal{O}_\lambda = X_{n,1} \rightarrow X_0 = X_{2,0}
\end{equation}
where $\phi = \phi_{2,1} \circ \cdots \circ \phi_{n,1}$ and it satisfies $\Phi \circ \phi = \Phi_\lambda$. 
Note that 
\[
	\phi \colon \Phi_\lambda^{-1}(\mathring{\Delta}_\lambda) \rightarrow \Phi^{-1}(\mathring{\Delta}_\lambda)
\]
is a symplectomorphism. 
\vspace{0.1cm}

Now, we describe how each GC fiber degenerates into a torus along the toric degeneration described above. 
To do this, first we need to determine which component $\Phi_\lambda^{i,j}$ is smooth on the fiber $\Phi_\lambda^{-1}(\textbf{u})$ for each $\textbf{u} \in \Delta_\lambda$. 

Let $\gamma$ be an $r$-dimensional face of $\Gamma_\lambda$ and let $f_\gamma$ be the corresponding face of $\Delta_\lambda$. 
Let 
	\begin{equation}\label{equation_index_cycle}
		\mcal{I}_{\mathcal{C}(\gamma)} := \{ (i,j) ~|~ (i,j) = v_\sigma \quad ~\text{for some minimal cycle}~\sigma~\text{of}~\gamma \}
	\end{equation}
so that $|\mcal{I}_{\mathcal{C}(\gamma)}| = r$. (See Figure \ref{figure_M_1_block_simple_closed_region}.)

\begin{lemma}\label{lemma_smooth_component_free_torus_action}
	Each $\Phi_\lambda^{i,j}$ is smooth on $\Phi_\lambda^{-1}(\mathbf{u})$ for every $(i,j) \in \mcal{I}_{\mathcal{C}(\gamma)}$.
	In particular, $\Phi_\lambda^{i,j}$ is smooth on $\Phi_\lambda^{-1}(\mathring{f}_\gamma)$. 
	Furthermore, $\{ \Phi_\lambda^{i,j} \}_{(i,j) \in \mcal{I}_{\mathcal{C}(\gamma)}}$ generates a fiberwise $T^r$-action on 
	$\Phi_\lambda^{-1}(\textbf{\textup{u}})$ for each $\textbf{\textup{u}} \in \mathring{f}_\gamma$.
\end{lemma}

\begin{proof}
	The smoothness of each $\Phi_\lambda^{i,j}$ on $\Phi_\lambda^{-1}(\mathring{f}_\gamma)$ follows 
	from the condition $(i,j) = v_\sigma$ and Proposition \ref{proposition_GS_smooth}.
	Also, we have seen in Section \ref{ssecSmoothnessOfPhiLambda} that each $\Phi_\lambda^{i,j}$ 
	generates a circle action on an open dense subset of $\mathcal{O}_\lambda$ on which $\Phi_\lambda^{i,j}$  is smooth.
	Since all components of $\Phi_\lambda$ Poisson-commute with each other, it finishes the proof.
\end{proof}

We recall the following well-known fact on toric varieties.
Let $\Delta \subset \R^N \cong \frak{t}^*$ be a convex polytope of dimension $N$ and let $X_\Delta$ be the corresponding projective toric variety
where $\frak{t}$ is the Lie algebra of the maximal compact torus $T^N$ in $(\C^*)^N \subset X_0$ and $\frak{t}^*$ is the dual of $\frak{t}$. 
Let $f$ be an $r$-dimensional face of $\Delta$ and suppose that $F_1, \cdots, F_m$ are facets of 
$\Delta$ such that $f = \cap_{i=1}^m F_i$. Also let $v_i \in \frak{t}$ be the inward primitive integral normal vector of $F_i$ for $i=1,\cdots,m$

\begin{lemma}[Exercise 12.4.7.(d) in \cite{CLS}]\label{lemma_freely}
	Let $\xi_1, \cdots, \xi_r$ be primitive integral vectors in $\frak{t}$ which generates an $r$-dimensional subtorus $T^r$ of $T^N$. 
	Then the $T^r$ acts on $\Phi^{-1}(\mathring{f})$ freely if
	$\Z^N \cong \langle v_1, \cdots, v_m, \xi_1, \cdots, \xi_r \rangle$.	
\end{lemma}

We then obtain the following.

\begin{proposition}\label{proposition_toric_degenerate_equivariant_map}
	Let $\textbf{\textup{u}} \in \mathring{f}_\gamma$ and consider the $T^r$-action on $\Phi_\lambda^{-1}(\textbf{\textup{u}})$ generated by 
	$\{ \Phi_\lambda^{i,j} \}_{(i,j) \in \mcal{I}_{\mathcal{C}(\gamma)}}$.
	Then the $T^r$-action on $\Phi_\lambda^{-1}(\textbf{\textup{u}})$ is free. Furthermore, $\Phi_\lambda^{-1}(\textbf{\textup{u}})$ becomes a trivial principal 
	$T^r$-bundle over 
	$\Phi_\lambda^{-1}(\textbf{\textup{u}}) / T^r$, that is,  
	\[
		\Phi_\lambda^{-1}(\textbf{\textup{u}}) \cong T^r \times \Phi_\lambda^{-1}(\textbf{\textup{u}}) / T^r.
	\]
\end{proposition}

\begin{proof}
	We first show that the $T^r$-action is free on $\Phi^{-1}(\textbf{\textup{u}})$. 
	For each $(i,j) \in \mcal{I}_\lambda$\footnote{See \eqref{collectionofindicesset} for the definition of 
	$\mcal{I}_\lambda$.}, we denote by 
	\[
		\xi_{i,j} := \left( u_{k,l} \right) \in \R^{\frac{n(n-1)}{2}}, \quad \begin{cases} \text{$u_{k,l} = 1$ if $(k,l) = (i,j)$}\\
		\text{$u_{k,l} = 0$ otherwise}.
		\end{cases}
	\]
	Note that an inward primitive integral normal vector $v_F$ for any facet $F$ of $\Delta_\lambda$ is either
	\[
		v_{vert}^{i,j} := \xi_{i,j+1} - \xi_{i,j}, \quad \text{or} \quad v_{hor}^{i,j} := -\xi_{i+1,j} + \xi_{i,j}
	\]
	where $(i,j) \in \mcal{I}_\lambda$. In particular, if $F$ contains $f_\gamma$, then $v_F$ is either
	\[
		\begin{cases}
			v_{vert}^{i, j}, \quad \text{$\square^{i,j}$ and $\square^{i,j+1}$ are in the same simple region of $\gamma$, or} \\
			v_{hor}^{i, j}, \quad \text{$\square^{i,j}$ and $\square^{i+1,j}$ are in the same simple region of $\gamma$}
		\end{cases}
	\] for some $(i,j) \in \mcal{I}_\lambda$. Then, it is not hard to see that 
	\[
		\{ v_F \}_{f_\gamma \subset F} \cup \{\xi_{i,j} \}_{(i,j) \in \mcal{I}_{\mcal{C}(\gamma)}} 
	\]
	generates the full lattice $\Z^N$ where $N = \dim \Delta_\lambda$. (See Figure \ref{figure_final_proof} for example.)
	Therefore the $T^r$-action generated by $\{\Phi^{i,j} \}_{(i,j) \in \mcal{I}_{\mcal{C}(\gamma)}}$ on $\Phi^{-1}(\mathring{f}_\gamma)$ is free by Lemma \ref{lemma_freely}, and hence 
	the $T^r$-action on each fiber $\Phi^{-1}(\textbf{\textup{u}})$ is free for every $\textbf{\textup{u}} \in \mathring{f}_\gamma$.
	Since 
	\[
		\phi_{\textbf{\textup{u}}} := \phi|_{\Phi_\lambda^{-1}({\textbf{\textup{u}}})} \colon \Phi_\lambda^{-1}(\textbf{\textup{u}}) 
		\to \Phi_{\vphantom{\lambda}}^{-1}(\textbf{\textup{u}}) \cong T^r. 
	\]
	is $T^r$-equivariant by Corollary \ref{corollary_torus_equivariant}, we see that the $T^r$-action on $\Phi_\lambda^{-1}(\textbf{\textup{u}})$ is also free. 
	
	The freeness of the $T^r$-action on $\Phi_\lambda^{-1}({\textbf{\textup{u}}})$ implies that the map
	$\phi_{\textbf{\textup{u}}} := \phi|_{\Phi_\lambda^{-1}({\textbf{\textup{u}}})}$ becomes a principal bundle map such that 
	\[
	\xymatrix{
		\Phi_\lambda^{-1}({\textbf{\textup{u}}}) \ar[d]_{/T^r} \ar[r]^{\phi_{\textbf{\textup{u}}}} & \Phi^{-1}({\textbf{\textup{u}}}) \cong T^r \ar[d]^{/ T^r}\\
		\Phi_\lambda^{-1}({\textbf{\textup{u}}}) / T^r \ar[r] & \mathrm{point} 
		}
	\] commutes. In particular, $\Phi_\lambda^{-1}({\textbf{\textup{u}}})$ is a pull-back bundle of the trivial $T^r$-bundle over a point so that $\Phi_\lambda^{-1}({\textbf{\textup{u}}})$ 
	is a trivial $T^r$-bundle as desired.
\end{proof}

\begin{figure}[h]
	\scalebox{1}{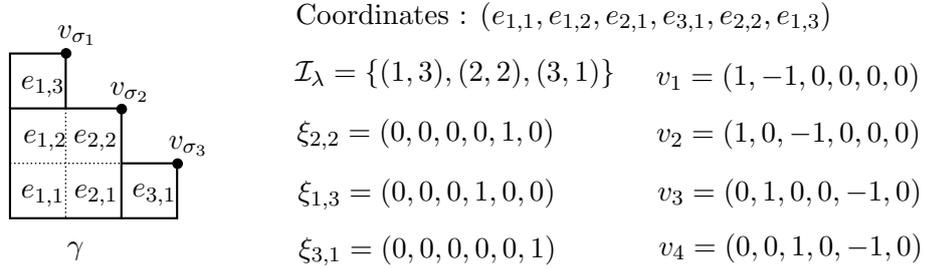}
	\caption{\label{figure_final_proof} Proof of Proposition \ref{proposition_toric_degenerate_equivariant_map} : $\lambda = (3,2,1,0)$ - case} 
\end{figure}

To sum up, we can describe how each fiber of a GC system deforms into a torus fiber of a moment map of $X_0$ via a toric degeneration as follows. 

\begin{theorem}\label{theorem_contraction}
	Let $\gamma$ be a face of the ladder diagram $\Gamma_\lambda$ of dimension $r$ and let $f_\gamma$ be the corresponding face of the Gelfand-Cetlin polytope $\Delta_\lambda$.
	For each point ${\textbf{\textup{u}}}$ in the relative interior $\mathring{f}_\gamma$, all $S^1$-factors that appeared in each stage of the iterated bundle structure $\bar{S_{\bullet}}(\gamma)$ 
	of $\Phi_\lambda^{-1}({\textbf{\textup{u}}})$ given in Theorem \ref{theorem_main} are factored out. That is, 
	\[
		\Phi_\lambda^{-1}({\textbf{\textup{u}}}) \cong T^r \times \bar{S_{\bullet}}(\gamma)'
	\] where 
	$\bar{S_{\bullet}}(\gamma)'$ is the total space of the iterated bundle which can be obtained by the construction of $\bar{S_{\bullet}}(\gamma)$
	ignoring all $S^1$-factors appeared in each stage. Furthermore, the continuous map $\phi$ in~\eqref{phicontimapx} on each fiber $\Phi_\lambda^{-1}({\textbf{\textup{u}}})$
	is the projection map 
	\[ 
		\Phi_\lambda^{-1}({\textbf{\textup{u}}}) \cong T^r \times \bar{S_{\bullet}}(\gamma)' \stackrel{\phi} \longrightarrow T^r \cong \Phi^{-1}({\textbf{\textup{u}}}). 
	\]
\end{theorem}

\begin{proof} 
	Consider the iterated bundle structure of $\Phi_\lambda^{-1}({\textbf{\textup{u}}})$ given in Theorem \ref{theorem_main} : 
	\begin{equation}\label{equation_iterated_bundle}
		\begin{array}{ccccccccc}
		\Phi_\lambda^{-1}({\textbf{\textup{u}}}) \cong \bar{S_{n-1}}(\gamma) & \xrightarrow {p_{n-1}} & \bar{S_{n-2}}(\gamma) &
		 \rightarrow & \cdots & 
		 \xrightarrow{p_2} & \bar{S_1}(\gamma) & \xrightarrow{p_1} & \bar{S_0}(\gamma) := \mathrm{point} \\
		 \hookuparrow & & \hookuparrow & & \hookuparrow & & \hookuparrow & &  \\
		 S_{n-1}(\gamma) & & S_{n-2}(\gamma)  & & \cdots & & S_1(\gamma)  \\
		 \end{array}
	\end{equation}	
	where $S_k(\gamma)$ is the fiber of $p_k \colon \overline{S_k}(\gamma) \rightarrow \overline{S_{k-1}}(\gamma)$ at the $k$-th stage
	defined in \eqref{equation_block_sphere}. 
	Each $S_k(\gamma)$ can be factorized into $S_k(\gamma) = (S^1)^{r_k} \times Y_k$ where $r_k$ is the number of $M_1$ blocks in $W_k(\gamma)$ containing a bottom vertex
	of $W_k$ and $Y_k$ is either a point or a product of odd-dimensional spheres 
	without any $S^1$-factors. (See the proof of Proposition \ref{proposition_bundle}.)
	Then we claim that 
	\begin{enumerate}
		\item there is a one-to-one correspondence between the $S^1$-factors that appeared in each stage and the elements in $\mcal{I}_{\mcal{C}(\gamma)}$, 
		\item $(S^1)^{r_k}$ acts on $\overline{S_k}(\gamma)$ fiberwise with respect to $p_k \colon \overline{S_k}(\gamma) \rightarrow \overline{S_{k-1}}(\gamma)$, and 
		\item the $(S^1)^{r_k}$-action on $\Phi_\lambda^{-1}({\textbf{\textup{u}}}) \cong \bar{S_{n-1}}(\gamma)$ generated by $\{ \Phi_\lambda^{i,j} \}_{(i,j) \in \mcal{I}_{\mcal{C}(\gamma)}, i+j-1=k}$
		is an extension of the $(S^1)^{r_k}$-action on $\overline{S_k}(\gamma)$ given in (2). 
	\end{enumerate}
	The first statement (1) is straightforward since each $(i,j) \in \mcal{I}_{\mcal{C}(\gamma)}$ corresponds to an $M_1$-block in 
	$W_{i+j-1}(\gamma)$ containing a bottom vertex of $W_{i+j-1}$
	so that each $(i,j) \in \mcal{I}_{\mcal{C}(\gamma)}$ assigns an $S^1$-factor in $S_{i+j-1}(\gamma)$. See Section \ref{ssecWShapedBlocks}. \
	The third statement (3) is also clear since each $\Phi_\lambda^{i,j}$ with $i+j = k+1$ descends to a function $\Phi_{\lambda_k}^{i,j}$ on 
	$\bar{S_{k-1}}(\gamma) \subset \mcal{H}_k$ where 
	\[
		\lambda_k = (\Phi_\lambda^{1,k}({\textbf{\textup{u}}}), \cdots, \Phi_\lambda^{k,1}({\textbf{\textup{u}}})). 
	\]

	For the second statement (2), fix $k \geq 1$ and consider the $k$-th stage
	\begin{equation}\label{bundleatstagek}
		\begin{array}{ccc}
			S_k(\gamma) = (S^1)^{r_k} \times Y_k & \hookrightarrow & \overline{S_k}(\gamma) \subset (\mathcal{O}_{\frak{a}}, \omega_\frak{a})\\
			                                     &    &  \downarrow p_k\\
			                                      &    &  \overline{S_{k-1}}(\gamma) \subset (\mcal{O}_{\frak{b}}, \omega_\frak{b}) \\
		\end{array}
	\end{equation}
	of $\overline{S_\bullet}(\gamma)$. 
	As we have seen in the diagram~\eqref{figure_iterated_bundle},  $\overline{S_k}(\gamma)$ is a subset of $\mcal{A}(\frak{a}, \frak{b})$ where $\frak{a} = (a_1, \cdots, a_{k+1})$ and $\frak{b} = (b_1, \cdots, b_k)$ with 
	\begin{itemize}
		\item $a_i = \Phi_\lambda^{i, k+2-i}({\textbf{\textup{u}}})$ for $1 \leq i \leq k+1$ and 
		\item $b_j = \Phi_\lambda^{j, k+1-j}({\textbf{\textup{u}}})$ for $1 \leq j \leq k$. 
	\end{itemize}
	Note that $\Phi_{\frak{a}}^{i,k+1-i}$ is smooth and hence generates a circle action
	on $\mcal{O}_{\frak{a}}$ whenever $a_i < b_i < a_{i+1}$, see Section \ref{ssecSmoothnessOfPhiLambda}.
	There are exactly $r_k$ such $b_i$'s and we only need to prove that the torus action generated by those $\Phi_{\frak{a}}^{i,k+1-i}$'s (with $a_i < b_i < a_{i+1}$)
	is fiberwise with respect to the projection $p_k$.
	
	 For any smooth function $f$ on $\mcal{O}_\frak{b}$ and 
	$\hat{f} = f \circ p_k$ on $\mcal{O}_\frak{a}$, we denote by $\xi_f$ and $\xi_{\hat{f}}$ the Hamiltonian vector fields for $f$ and $\hat{f}$, respectively. 
	Then it satisfies $(p_k)_* \xi_{\hat{f}} = \xi_f$, i.e., $dp_k ( \xi_{\hat f})(x) = \xi_f(p_k(x))$ for every $x \in \mcal{O}_{\frak{b}}$ since
	\begin{equation}\label{equation_pull_back}
		\begin{array}{ccl}
			\omega_\frak{b}((p_k)_* \xi_{\hat{f}}, (p_k)_* (\cdot)) & = & (p_k)^*\omega_\frak{b}(\xi_{\hat{f}}, \cdot) \\
			                                                                         & = & \omega_\frak{a}(\xi_{\hat{f}}, \cdot) \\
			                                                                         & = & d\hat{f}(\cdot) = df((p_k)_* (\cdot)) \\
			                                                                         & = & \omega_\frak{b}(\xi_f, (p_k)_* (\cdot))
		\end{array}
	\end{equation}
	where the second equality comes from Lemma \ref{lemma_isotropic}. 
	
	We apply this to $\Phi_{\frak{a}}^{i,k+1-i} = \Phi_{\frak{b}}^{i,k+1-i} \circ p_k$. 
	Since $\Phi_{\frak{b}}^{i,k+1-i}$ is a constant function on $\mcal{O}_{\frak{b}}$, we see that the Hamiltonian vector field generated by 
	$\Phi_{\frak{a}}^{i,k+1-i}$ is tangent to the vertical direction of $p_k$ in \eqref{bundleatstagek}.
	
	Once (1), (2), and (3) are satisfied, its iterated bundle in \eqref{equation_iterated_bundle} descends to 
	\begin{equation}\label{equation_iterated_bundle_quotient}
		\begin{array}{ccccccccc}
		\Phi_\lambda^{-1}(\textbf{\textup{u}}) / (S^1)^r \cong \bar{S_{n-1}}(\gamma)' & \xrightarrow {p_{n-1}'} & \bar{S_{n-2}}(\gamma)' &
		 \rightarrow & \cdots & 
		 \xrightarrow{p_2'} & \bar{S_1}(\gamma)' & \xrightarrow{p_1'} & \bar{S_0}(\gamma)' := \mathrm{point} \\
		 \hookuparrow & & \hookuparrow & & \hookuparrow & & \hookuparrow & &  \\
		 Y_{n-1} & & Y_{n-2}  & & \cdots & & Y_1  \\
		 \end{array}
	\end{equation}	
	where $\overline{S_k}(\gamma)' = \overline{S_k}(\gamma) / (S^1)^{r_k + \cdots + r_1}$ and the $(S^1)^{r_k + \cdots + r_1}$-action is generated by 
	$\{ \Phi_\lambda^{i,j} \}_{(i,j) \in \mcal{I}_{\mcal{C}(\gamma)}, i+j \leq k+1}$. Since $\Phi_\lambda^{-1}(\textbf{\textup{u}}) \cong T^r \times Y(\textbf{\textup{u}})$ for some $Y(\textbf{\textup{u}})$ by Proposition \ref{proposition_toric_degenerate_equivariant_map},
	we have $Y(\textbf{\textup{u}}) \cong \Phi_\lambda^{-1}(\textbf{\textup{u}}) / T^r \cong \overline{S_\bullet}(\gamma)'$, which completes the proof.
\end{proof}

As an application of Theorem~\ref{theorem_contraction}, one can provide a more explicit description of GC fibers. As mentioned in Remark~\ref{remark_su(3)}, an iterated bundle in Theorem~\ref{theorem_main} is \emph{not} trivial in general. Theorem~\ref{theorem_contraction} guarantees that, in some case, the iterated bundle can be characterized explicitly. 

\begin{example}
In the following examples, the topology of each GC fibers are computed explicitly. 
\begin{enumerate}
\item Let $\mcal{O}_\lambda \simeq \mcal{F}(6)$ be the co-adjoint orbit associated with $\lambda = (5,3,1,-1,-3,-5)$. 
Consider the face $\gamma_1$ defined in Figure~\ref{Figure_diffeotype}. Then, one can see that there are seven $S^1$-factors, which coincides with the number of 
minimal cycles in $\gamma_1$, so that we have 
\[
	\Phi^{-1}_\lambda(\textbf{\textup{u}}) \simeq (S^1)^7 \times Y({\bf u})
\]
where $Y({\bf u}) \cong SU(3)$ by Remark~\ref{remark_su(3)}. 
\item Let $\lambda = (3,3,3, -3,-3,-3)$. 
Then, the co-adjoint orbit $\mcal{O}_\lambda$ is $\mathrm{Gr}(3,6)$. 
Consider the face $\gamma_2$ defined in Figure~\ref{Figure_diffeotype}. 
Similarly, one can easily check that the fiber over a point in the relative interior of $\gamma_2$ has three $S^1$ factors and we get 
\[
	\Phi^{-1}_\lambda(\textbf{\textup{u}}) \simeq (S^1)^3 \times Y({\bf u})
\]
where $Y(\textbf{\textup{u}})$ is an $S^3$-bundle over $S^3$, which should be homeomorphic to $S^3 \times S^3$, (see Steenrod \cite{St}).
\end{enumerate}
\end{example}

\begin{figure}[ht]
	\scalebox{1}{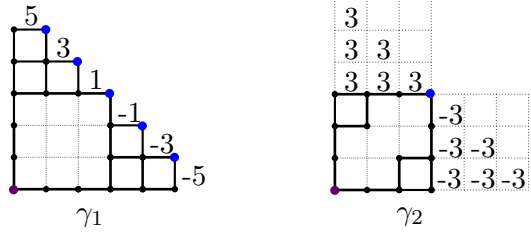}
	\caption{\label{Figure_diffeotype} Faces $\gamma_1$ and $\gamma_2$ of $\mcal{F}(6)$ and $\mathrm{Gr}(3,6)$}	
\end{figure}

Another application of Theorem \ref{theorem_contraction} is to compute the first and the second homotopy groups of each $\Phi_\lambda^{-1}(\textbf{\textup{u}})$ as follows.
Let $\textbf{\textup{u}} \in \Delta_\lambda$ and let $f$ be the face of $\Delta_\lambda$ containing $\textbf{\textup{u}}$ in its relative interior. 
Also, let $\gamma$ be the face of $\Gamma_\lambda$ corresponding to $f$.
For each $k=1, \cdots, n-1$, the fibration~\eqref{bundleatstagek} 
induces the long exact sequence of homotopy groups given by
\begin{equation}\label{equation_homotopy_long_exact_sequence}
	\begin{array}{ccccccccccc}
		\cdots  &
		\rightarrow & \pi_2(S_k(\gamma)) & \rightarrow & \pi_2(\bar{S_k}(\gamma))&
		\rightarrow &  \pi_2(\bar{S_{k-1}}(\gamma))& & & & \\
		& \rightarrow & \pi_1(S_k(\gamma)) & \rightarrow & \pi_1(\bar{S_k}(\gamma)) & \rightarrow &
		\pi_1(\bar{S_{k-1}}(\gamma)) &
		\rightarrow & \pi_0(S_k(\gamma)) & \rightarrow & \cdots
	\end{array}
\end{equation}

\begin{proposition}\label{proposition_pi_2_zero_pi_1_k}
	Let $\textbf{\textup{u}} \in \Delta_\lambda$. Then the followings hold.
	\begin{itemize}
		\item $\pi_2(\Phi_\lambda^{-1}(\textbf{\textup{u}})) = 0$.
		\item If $\textbf{\textup{u}}$ is a point in the relative interior of an $r$-dimensional face of $\Delta_\lambda$, then
		\[
			\pi_1(\Phi_\lambda^{-1}(\textbf{\textup{u}})) \cong \Z^r.
		\]
	\end{itemize}
\end{proposition}
\begin{proof}
	Since $S_k(\gamma)$ in \eqref{bundleatstagek} is a point or a product space of odd dimensional spheres,
	we have $\pi_1(S_k(\gamma)) = \pi_2(S_k(\gamma)) = 0$, and therefore 
	$\pi_2(\overline{S_k}(\gamma)) \cong \pi_2(\overline{S_{k-1}}(\gamma))$ for every $k=1,\cdots, n-1$. 
	Moreover, since $\pi_2(\bar{S_1}(\gamma)) = \pi_2(S_1(\gamma)) = 0$, we get $\pi_2(\bar{S_k}(\gamma)) = 0$ for every $k$ by the induction on $k$. 
	The second statement is deduced from Theorem \ref{theorem_contraction}, since $\bar{S_{\bullet}}(\gamma)'$ is simply connected. 
\end{proof}

\begin{corollary}\label{corollary_Lagrangian_tori}
	For a point $\textbf{\textup{u}} \in \Delta_\lambda$, the fiber $\Phi_\lambda^{-1}(\textbf{\textup{u}})$ is a Lagrangian torus if and only if $\textbf{\textup{u}}$ is an interior point of
	$\Delta_\lambda$
\end{corollary}

\begin{proof}
	The ``if'' statement follows immediately from Theorem \ref{theorem_contraction}, and the ``only if'' part follows from Proposition \ref{proposition_pi_2_zero_pi_1_k}.
\end{proof}

\bibliographystyle{amsalpha}
\bibliography{geometry}

\end{document}